\documentclass[twoside, 12pt, a4paper, reqno]{amsart}
\usepackage{lipsum}
\usepackage[T1]{fontenc}
\usepackage[utf8]{inputenc}
\usepackage{lmodern}
\usepackage[pdftex]{graphicx}
\graphicspath{{./Images/1/}{./Images/}}
\usepackage{float}
\usepackage[margin=30mm]{geometry}
\usepackage{url}
\usepackage{soul}
\usepackage[alphabetic,nobysame,abbrev]{amsrefs}
\usepackage{microtype}
\usepackage[all]{xy}
\usepackage{amssymb,relsize,datetime2}
\usepackage{booktabs,multicol}
\usepackage[pdftex, dvipsnames]{xcolor}
\usepackage{caption}
\usepackage{subcaption}
\usepackage{transparent}
\usepackage{ifpdf}
\usepackage{import}
\usepackage{multirow}
\usepackage{tikz-cd}
\usepackage{adjustbox}
\usepackage{parskip}
\usepackage{xpatch}
\usepackage{comment}
\usepackage[normalem]{ulem}
\usepackage[breaklinks,%
    pdftex,%
    final,%
    colorlinks=true,%
    linkcolor=NavyBlue,%
    citecolor=blue,
    filecolor=NavyBlue,%
    menucolor=NavyBlue,%
    urlcolor=NavyBlue,%
    bookmarks=true,%
    bookmarksdepth=2,%
    bookmarksnumbered=true,%
    bookmarksopen=true,%
    bookmarksopenlevel=2,%
    unicode
]{hyperref}
\usepackage[capitalize,noabbrev]{cleveref}

\usepackage{array}
\makeatletter
\def\env@dmatrix{\hskip -\arraycolsep
  \let\@ifnextchar\new@ifnextchar
  \extrarowheight=2ex
  \array{*\c@MaxMatrixCols{>{\displaystyle}c}}}

\newenvironment{pdmatrix}
  {\left(\env@dmatrix}
  {\endmatrix\right)}
\makeatother

\makeatletter
\newtheorem*{rep@theorem}{\rep@title}
\newcommand{\newreptheorem}[2]{%
\newenvironment{rep#1}[1]{%
 \def\rep@title{#2 \ref{##1}}%
 \begin{rep@theorem}}%
 {\end{rep@theorem}}}
\makeatother
\newreptheorem{theorem}{Theorem}
\newreptheorem{lemma}{Lemma}
\newreptheorem{proposition}{Proposition}
\newreptheorem{corollary}{Corollary}		

\crefname{question}{Question}{Questions}

\crefname{alpharesults}{Theorem}{Theorems}
\crefname{alphatheorem}{Theorem}{Theorems}

\usepackage{enumitem}

\usepackage{listings} %source code

\usepackage{footnotebackref}
\makeatletter
\xpretocmd{\@adminfootnotes}{\let\@makefntext\BHFN@OldMakefntext}{}{}
\renewcommand\@makefntext[1]{%
  \@ifundefined{@makefnmark}
    {}
    {%
     \renewcommand\@makefnmark{%
       \mbox{%
         \textsuperscript{%
           \normalfont
           \hyperref[\BackrefFootnoteTag]{\@thefnmark}%
         }%
       }\,%
     }%
     \BHFN@OldMakefntext{#1}%
  }%
}
\makeatother

\newtheorem{theorem}{Theorem}[section]

\newtheorem{lemma}[theorem]{Lemma}
\newtheorem{corollary}[theorem]{Corollary}
\newtheorem{proposition}[theorem]{Proposition}

\newtheorem{question}[theorem]{Question}

\newcounter{alpharesults}

\newtheorem{alphatheorem}[alpharesults]{Theorem}

\theoremstyle{definition}
\newtheorem{definition}[theorem]{Definition}
\newtheorem*{definition-nonum}{Definition}

\theoremstyle{remark}

\newtheorem{remark}[theorem]{Remark}

\numberwithin{equation}{section}

\newcommand{\ie}{i.e.~}

\newcommand{\eg}{e.g.~}

\newcommand{\R}{\mathbb{R}}

\newcommand{\Z}{\mathbb{Z}}

\newcommand{\id}{\operatorname{id}}
\newcommand{\intr}{\operatorname{int}}
\newcommand{\hm}{\operatorname{Hom}}
\DeclareMathOperator{\MCG}{MCG}

\usepackage{mathtools}
\newcommand{\defeq}{\vcentcolon=}

\AtBeginDocument{%
   \def\MR#1{}
}

\title{On the detection of knots by their traces in high dimensions}

\author[Bais]{Valentina Bais}
\author[Di Prisa]{Alessio Di Prisa}
\author[Hartman]{Daniel Hartman}
\author[Hsueh]{Chun-Sheng Hsueh}
\author[Kegel]{Marc Kegel}
\author[Merz]{Alice Merz}
\author[Pencovitch]{Mark Pencovitch}
\author[Ray]{Arunima Ray}
\author[Santoro]{Diego Santoro}
\author[Truöl]{Paula Truöl}
\author[Wakelin]{Laura Wakelin}

\address{Scuola Internazionale Superiori di Studi Avanzati (SISSA), Via Bonomea
265, 3413, Trieste, Italy}
\email{\href{mailto:vbais@sissa.it}{vbais@sissa.it}}

\address{HUN-REN Alfréd Rényi Institute of Mathematics, Reáltanoda utca 13-15, 1053 Budapest, Hungary}
\email{\href{mailto:diprisa.alessio@renyi.hu}{diprisa.alessio@renyi.hu}}

\email{\href{mailto:merz.alice@renyi.hu}{merz.alice@renyi.hu}}

\email{\href{mailto:diego.santoro95@gmail.com}{diego.santoro95@gmail.com}}

\address{Duke University, 2080 Duke University Road, Durham NC 27708}
\email{\href{mailto:daniel.hartman@duke.edu}{daniel.hartman@duke.edu}}

\address{Humboldt-Universit\"at zu Berlin, Rudower Chaussee 25, 12489 Berlin, Germany.}
\email{\href{mailto:chun-sheng.hsueh@hu-berlin.de}{chun-sheng.hsueh@hu-berlin.de}}

\address{Universidad de Sevilla, Dpto.\ de Álgebra,
Avda.\ Reina Mercedes s/n,
41012 Sevilla,
Spain}
\email{\href{mailto:mkegel@us.es}{mkegel@us.es}, \href{mailto:kegelmarc87@gmail.com}{kegelmarc87@gmail.com}}

\address{University of Glasgow, School of Mathematics and Statistics, University Place, Glasgow, G128QQ, Scotland}
\email{\href{mailto:m.pencovitch.1@research.gla.ac.uk}{m.pencovitch.1@research.gla.ac.uk}}

\address{The University of Melbourne, Peter Hall Building, 8 Monash Rd, Parkville VIC 3010, Australia}
\email{\href{mailto:aru.ray@unimelb.edu.au}{aru.ray@unimelb.edu.au}}
\address{Max Planck Institute for Mathematics, Vivatsgasse 7, 53111 Bonn, Germany}
\email{\href{mailto:paulagtruoel@gmail.com}{paulagtruoel@gmail.com}}

\address{King's College London, Strand, London, WC2R 2LS, United Kingdom}
\email{\href{mailto:laura.1.wakelin@kcl.ac.uk}{laura.1.wakelin@kcl.ac.uk}}

\def\subjclassname{\textup{2020} Mathematics Subject Classification}
\expandafter\let\csname subjclassname@1991\endcsname=\subjclassname

\begin{document}

\begin{abstract}
    For every $n\geq4$, we demonstrate the existence of non-isotopic, smooth $(n-2)$-knots in~$S^n$ with diffeomorphic traces. We give two proofs: the first by generalising the RBG link construction to all dimensions, and the second as an application of work of Plotnick. Conversely, we prove that for every $n\geq 4$, the unknot in $S^n$ is detected by the diffeomorphism type of its surgery and hence of its trace.  
\end{abstract}

\maketitle

\section{Introduction} 

A \emph{knot in $S^n$}, for $n\geq 3$, is a smooth embedding $K \colon S^{n-2} \hookrightarrow S^n$. The \emph{trace} $X(K)$ of a knot $K$ is the oriented, connected, compact, smooth $(n+1)$-manifold obtained by attaching an $(n-1)$-handle to~$D^{n+1}$ (endowed with its natural orientation) along a tubular neighbourhood of $K$ and smoothing corners. In this article, we study knots in ambient dimension~$n\geq4$, where, in contrast to ambient dimension~$n=3$, there is a unique choice of framing (see \Cref{subsec:framings} for a discussion on framings). (Ambiently) isotopic knots have orientation-preservingly diffeomorphic traces. Our main result shows that the converse is not true, \ie the oriented diffeomorphism type of the trace~$X(K)$ does not detect the isotopy class of the knot $K$. From now on, we will abbreviate ``orientation-preserving(ly)'' as~``o.p.''.

\begin{alphatheorem}\label{theorem:non-isotopic}
    For every $n\geq 4$, there exist non-isotopic knots $K_1$ and $K_2$ in $S^n$ whose traces $X(K_1)$ and $ X(K_2)$ are o.p.~diffeomorphic. 
    For every $n \geq 5$, there exist infinite families of non-isotopic knots $\{K_j\}_{j\geq1}$ with o.p.~diffeomorphic traces $X(K_j)$.
\end{alphatheorem}

See \Cref{infinite_example_intro} below for a more precise statement of the $n \geq 5$ version.

For classical knots, \ie those in $S^3$, the existence of non-isotopic framed knots with o.p.~diffeomorphic traces follows from~\cite{Akbulut1}. The existence of non-isotopic framed knots in $S^3$ that share a surgery was first observed in~\cite{Lickorish_sharing_surgery}. 

We provide two independent proofs of \Cref{theorem:non-isotopic} using different techniques. The first applies to dimensions $n\geq 4$ and gives infinitely many pairs of knots with o.p.~diffeomorphic traces. The second produces infinite families of such knots for dimensions~$n\geq 5$ (see \Cref{infinite_example_intro}) and arbitrarily large families for~$n=4$ (see \Cref{rem:fin_many_2-knots_same_trace}). Moreover, both techniques produce knots that are not isotopic even after allowing changes of parametrisation (see \Cref{REM:embeddings} for more details) or orientation.

\subsection{RBG manifold construction}

The first proof is based on our generalisation of the \emph{RBG link construction} from~\cites{piccirillo:shake-genus,manolescu-piccirillo:RBG,Tagami} to all dimensions. This simple and elegant method for constructing knots in $S^3$ with the same trace or surgery was used in a series of papers to settle a host of important questions and conjectures, \eg in~\cites{piccirillo:conway-knot,piccirillo:shake-genus,Baker_Motegi_non_char,Kegel_Piccirillo,Wakelin_Picirillo_Hayden_any_slope}.
Our generalisation to~\emph{RBG manifolds} yields a direct method for constructing, for $n\geq 4$, knots $K \colon S^{n-2}\hookrightarrow S^n$ with o.p.~diffeomorphic traces. An RBG manifold is, roughly speaking, an $(n+1)$-manifold~$W$ with a handle decomposition consisting of a $0$-handle~$D^{n+1}$, an $(n-2)$-handle, and two $(n-1)$-handles attached to~$D^{n+1}$ in such a way that, after an isotopy of the attaching maps, each $(n-1)$-handle cancels the $(n-2)$-handle (see \Cref{def:RBG_mfd}). The images of the attaching spheres of the $(n-1)$-handles under the diffeomorphisms given by the handle cancellations then form a pair of knots in~$S^n$ with the same trace (see \Cref{thm:rbg_knots}). Of course, in overly simple situations, these knots will in fact be isotopic (see \Cref{subsec:trivial_cases}). However, using a specific instance of our RBG manifold construction, we show the existence of infinitely many mutually distinct pairs of non-isotopic knots as in \Cref{theorem:non-isotopic}. The precise statement is given in \Cref{theorem:non-isotopic_compl}. For~$n=4$, our construction can be visualised in Kirby diagrams (see~\Cref{subsec:Kirby_diagram}).

More generally, we define RBG manifolds which lead to framed sphere embeddings $\Sigma \colon S^k \hookrightarrow S^n$ of any codimension $n-k$ with o.p.~diffeomorphic traces; see \Cref{subsec:RBG}.
We call such embeddings \emph{(framed) sphere embeddings}
(see also \Cref{sec:background_knots}) to distinguish from the case of codimension $n-k=2$, in which case we call them \emph{knots}.

\subsection{Distinguishing knots with the same trace}

In \Cref{sec:distinguishing}, we discuss methods by which we can distinguish knots with o.p.~diffeomorphic traces. In dimension~$n=3$, there are many computable knot invariants that can be used to distinguish such knots in $S^3$, \eg the knot group, hyperbolic volume and knot polynomials. In higher dimensions, this is a more difficult task. For~$n \geq 4$, if two knots in~$S^n$ have diffeomorphic traces, then many of their invariants are the same, such as the knot groups and abelian invariants that can be derived from it (see \Cref{proposition:fundamental_group}).

Nevertheless, we successfully use the conjugacy class of the meridians in the knot groups $\pi_1(K_i)$, $i \in \{1,2\}$, to distinguish between knots $K_1$ and $K_2$ with the same trace. The idea is to count the number of homomorphisms of $\pi_1(K_i)$ into a specific finite group $A$ that map an element representing the conjugacy class of the meridian of $K_i$ to a fixed, well-chosen element of $A$. This number is an ambient isotopy invariant of the knot (see~\Cref{subsec:counting_reps}), which suffices to distinguish our pairs of knots~$K_1$ and $K_2$ with the same trace (see \Cref{subsec:explicit_reps}).

\begin{remark}
    As an application of \cref{theorem:non-isotopic}, we disprove a generalisation of the \emph{light bulb theorem} to links. Recall that the classical light bulb trick shows that any 1-knot in $S^1\times S^2$ intersecting $\{*\}\times S^2$ in a single point is isotopic to $S^1\times \{*\}$. Gabai and Litherland showed corresponding results for $(n-2)$-knots in $S^2\times S^{n-2}$~\cites{LBT,litherland:LBT}. In \Cref{cor:LBT}, we show that the analogue of these results for 2-component links in~$S^2\times S^{n-2}$ does not hold. That is, for each $n\geq 4$, we produce a 2-component link~$L$ of $(n-2)$-spheres in $S^2\times S^{n-2}$ such that each component of $L$ is isotopic to the standard $\{z\}\times S^{n-2}$ for~$z\in S^2$, but $L$ is not isotopic to $\{z_1,z_2\}\times S^{n-2}$ for~$z_1\neq z_2\in S^2$.
    
    Indeed, Litherland's paper~\cite{litherland:LBT} answered a question posed by Rolfsen~\cite{rolfsen:PL-I-equivalence} regarding a concordance analogue of the light bulb theorem for links. Specifically, for all odd $n$, Litherland showed the existence of $2$-component links of $(n-2)$-spheres in $S^2\times S^{n-2}$ which are $I$-equivalent, but not concordant, as links, even though the components are individually concordant (see \cite{rolfsen:PL-I-equivalence} for the definition of $I$-equivalence). Rolfsen had already observed in \cite{rolfsen:PL-I-equivalence}*{Theorem~3} that no such links exist for even $n$, as a consequence of the fact that the even knot concordance groups are trivial. In \Cref{rem:LBT}, we explain how a similar construction can be applied to our question for all $n$. 
    These provide a new family of examples compared to ours using the RBG construction. 
\end{remark}

\subsection{Infinite families for \texorpdfstring{$\boldsymbol{n \geq 5}$}{n >=5}}

Our second proof of \Cref{theorem:non-isotopic} is based on a result by Plotnick~\cite{Plotnick1983InfinitelyMD}; see the discussion in \Cref{subsec:intro:Gluck} below. This gives the following generalisation of \Cref{theorem:non-isotopic} for $n\geq 5$.

\begin{alphatheorem}\label{infinite_example_intro}
    Let $n \geq 5$.
   There is an infinite collection $\{X_i\}_{i \geq 1}$ of pairwise non-diffeomorphic smooth $(n+1)$-manifolds such that, for each integer~$i \geq 1$, there exists an infinite collection~$\{K_j^i \}_{j \geq 1}$ of pairwise non-isotopic knots in $S^n$ whose trace is o.p.~diffeomorphic to $X_i$, \ie
    \begin{align*}
        X(K_j^i) \cong X_i \qquad \text{for all integers } \quad i,j \geq 1.
    \end{align*}
    In particular, for each $i \geq 1$, this results in an infinite family of pairwise non-isotopic knots in $S^n$ with o.p.~diffeomorphic traces. 
\end{alphatheorem}

Using the same method, we can construct arbitrarily large families of non-isotopic knots in $S^4$ with o.p.~diffeomorphic traces  (see \Cref{rem:fin_many_2-knots_same_trace}). Moreover, for $n=4$, we also obtain the analogous result to \Cref{infinite_example_intro} in the topological category (see \Cref{Plot4}).

\subsection{Surgeries, traces and Gluck twists}\label{subsec:intro:Gluck}

Our RBG manifold construction only produces \emph{pairs} of knots with o.p.~diffeomorphic traces. To prove the existence of infinitely many knots sharing the same trace as in \Cref{infinite_example_intro}, we need to employ different methods. Our main insight is summarised in the following result. For~$n\geq4$, let~$S^n(K)$ denote the manifold obtained by performing surgery on~$S^n$ along the knot~$K$. See \Cref{subsec:framings} for details.

\begin{alphatheorem}\label{Glucktwist}
    For $n \geq 5$, let $K_1$ and $K_2$ be knots in $S^n$ and let $\varphi \colon S^n(K_1) \to S^n(K_2)$ be an o.p.~diffeomorphism between their surgeries. Then either there is an o.p. diffeomorphism $X(K_1) \cong X(K_2)$ or the Gluck twist~$T(K_1)$ is a knot in the standard smooth $n$-sphere with $X(T(K_1)) \cong X(K_2)$.
\end{alphatheorem}

We will recall the precise definition of the Gluck twist, analogous to the definition for $n=4$ by Gluck~\cite{Gluck}, in \Cref{def:gluck_twist}. Roughly speaking, regluing a tubular neighbourhood of a knot~$K$ in $S^n$ using a non-trivial diffeomorphism of $S^1 \times S^{n-2}$ gives rise to a knot~$T(K)$ in a possibly non-standard $S^n$ (see also \Cref{rem:Gluck_manifold}). 

We will use \Cref{Glucktwist} to deduce \Cref{infinite_example_intro} as follows. For every $n \geq 5$, Plotnick~\cite{Plotnick1983InfinitelyMD} shows the existence of an infinite family $\mathcal{K}_{p,q,r}$ of pairwise inequivalent knots in~$S^n$ with o.p.~diffeomorphic surgeries, starting from a certain Brieskorn sphere~$\Sigma(p,q,r)$. By using \Cref{Glucktwist} and paying particular attention to the framings of these knots, we can ensure that all the knots in $\mathcal{K}_{p,q,r}$ have o.p.~diffeomorphic traces. We can apply this method to infinitely many distinct Brieskorn spheres, yielding the infinite families in \Cref{infinite_example_intro}.

For $n=4$, we obtain a weaker result than \Cref{Glucktwist} (see \Cref{Glucktwisttop}). In particular, given an o.p.~diffeomorphism $\varphi \colon S^4(K_1) \to S^4(K_2)$ for knots~$K_1, K_2$ in~$S^4$, we can only deduce the existence of one of the diffeomorphisms as in \Cref{Glucktwist} if we already know that the Gluck twist along $K_1$ is the standard $4$-sphere. This, together with another construction of Plotnick \cites{plotnick1983homotopy}, allows us to show in \Cref{rem:fin_many_2-knots_same_trace} the existence of arbitrarily large families of knots in $S^4$ with o.p.~diffeomorphic traces.

\Cref{sec:traces_Gluck} also establishes that any o.p.~diffeomorphism $\varphi \colon S^4(K_1) \to S^4(K_2)$ extends to either a trace diffeomorphism $X(K_1)\to X(K_2)$ or to a homeomorphism $X(T(K_1))\to X(K_2)$.\footnote{Note that $T(K_1)$ lies in a possibly non-standard smooth $S^n$, but the homeomorphism type of its trace $X(T(K_1))$ is still well-defined (see the discussion in \Cref{sec:traces_Gluck}).}
The following theorem provides the complete answer as to which of these two cases occurs, depending on the algebraic topology of the traces. This echoes the analogous result of Manolescu--Piccirillo~\cite{manolescu-piccirillo:RBG}*{Theorem~3.7} based on \cite{Boyer1} for extending homeomorphisms in the $n=3$ case.

\begin{alphatheorem}\label{extension_intro}
    Let $K_1$ and $K_2$ be $2$-knots in $S^4$. An o.p.~surgery diffeomorphism \(\varphi\colon S^4(K_1) \to S^4(K_2)\) extends to an o.p.~trace diffeomorphism $\Phi\colon X(K_1) \to X(K_2)$ if and only if the closed $5$-manifold $X(K_1) \cup_{\varphi} - X(K_2)$ is spin.
\end{alphatheorem}

\subsection{Detection results}\label{subsec:detection}

Neither \Cref{theorem:non-isotopic} nor its proofs prevent particular knots from being detected by their traces. 
For example, when $n=3$, it follows from Property R~\cite{gabai:property_R} that if $S^1\times S^2$ is obtained by surgery along a knot in $S^3$, then the knot is isotopic to the unknot. We prove a higher-dimensional analogue. 
Here, an \emph{unknot}~$U$ in~$S^n$, for $n\geq 3$, is an embedding~$S^{n-2}\hookrightarrow S^n$ that extends to an embedding $D^{n-1} \hookrightarrow S^n$. Up to (ambient) isotopy, this determines the unknot uniquely.

\begin{alphatheorem}
\label{theorem:unknot}
    For $n\geq 4$, let $K$ be a knot in~$S^n$. Suppose that its surgery~$S^n(K)$ is (possibly orientation-reversingly) diffeomorphic to the surgery~$S^n(U) \cong S^{n-1} \times S^1$ of the unknot $U$. Then $K$ is isotopic to $U$. 
\end{alphatheorem}

In particular, if $K$ and $U$ have diffeomorphic traces, then $K$ is isotopic to $U$. We prove \Cref{theorem:unknot} in \Cref{sec:property_R} as a corollary of \Cref{thm: property R new}, which is a more general detection result for \emph{unknotted} embeddings $S^{n-2} \hookrightarrow S^n$. We call a knot $K\colon S^{n-2}\hookrightarrow S^n$ \emph{unknotted} if its image bounds a ball (see \Cref{def:unknotted embedding}). The distinction between unknots in $S^n$ and unknotted embeddings $S^{n-2} \hookrightarrow S^n$ is a subtle one, related to whether knots are defined as embeddings or their images (see \Cref{REM:embeddings}). 

There are other unknot detection results in high dimensions in the literature. For instance, it was shown in  \cites{shaneson1968embeddings,Levine,Levine2,Trotter} that if the exterior of a knot in $S^n$, for $n\geq 5$, is homotopy equivalent to a circle, then it is unknotted. It remains an open question whether a 2-knot in $S^4$ with infinite cyclic fundamental group of the complement is (smoothly) unknotted. Such knots are known to be topologically unknotted~\cite{FQ:book}*{Theorem~11.7A}.

The key observation in the proof of \Cref{thm: property R new,theorem:unknot} is that under their hypotheses, 
the exterior of the knot $K$ is diffeomorphic to the exterior of~$U$, \ie $D^{n-1} \times S^1$. This uses the fact that the fundamental group of the surgery along an unknotted embedding is isomorphic to $\Z$ and thus has exactly two possible generators. The proof of \Cref{thm: property R new} is completed for $n\geq 5$ by the results mentioned in the previous paragraph. For $n=4$ we appeal to~\cite{Gluck}, where it was shown that if two 2-knots have diffeomorphic exteriors, then either they are isotopic or one is isotopic to the Gluck twist of the other.

\subsection{Structure} 

In \Cref{sec:prelims}, we first present the background material that will be required later on. \Cref{sec:construction} then outlines the construction of RBG manifolds, before \Cref{sec:distinguishing} focuses on knot invariants. Together, these tools allow us to complete the proof of \Cref{theorem:non-isotopic}. In \Cref{sec:traces_Gluck,sec:inf_family}, we first prove \Cref{Glucktwist,extension_intro}, then \Cref{infinite_example_intro}. \Cref{sec:property_R} contains the proof of \Cref{theorem:unknot}. \Cref{sec:questions} concludes the paper with a discussion of open questions.

\subsection{Conventions}

Unless stated otherwise, all manifolds are assumed to be oriented, connected, compact and smooth.
Maps between such manifolds are assumed to be smooth. Diffeomorphisms are assumed to be orientation-preserving (o.p.) and denoted by~$\cong$. In \Cref{sec:traces_Gluck,sec:inf_family}, we will in some places also be concerned with o.p.~homeomorphisms, denoted by~$\approx$. Homology groups~$H_\ast(\cdot)$ without specified coefficient rings are understood to have $\Z$ coefficients.

The oriented diffeomorphism type of the surgery or trace of a knot $K$ is independent of the orientation of $K$. Thus, all results in the introduction are meant to be understood as results for unoriented knots. However, for our proofs, the orientations are helpful to precisely describe certain constructions. So in the rest of the text, we will assume all our knots to be oriented.

\subsection*{Acknowledgements}

This project began while the authors were at Oberwolfach for the MATRIX-MFO tandem workshop `Invariants in Low-Dimensional Topology: Combinatorics, Geometry, and Computation', organised by Stefan Friedl, Joan Licata, Stephan Tillmann and PT. We thank the organisers for this great event and the~MFO for providing an excellent research environment.

We would like to thank Daniel Galvin, Paolo Lisca and Eric Stenhede for useful discussions, and we are grateful to David Gay, Geunyoung Kim, Lisa Piccirillo and Mark Powell for helpful comments on a first draft. We would also like to thank the referees for their careful reading of the text and their insightful comments and suggestions.

\subsection*{Individual grant support}

VB has been partially supported by GNSAGA – Istituto Nazionale di Alta Matematica ‘Francesco Severi’, Italy. DH, AR and PT would like to thank the Max Planck Institute for Mathematics in Bonn for its hospitality and support. CSH is supported by the Claussen-Simon-Stiftung and is a member of the Berlin Mathematics Research Center MATH+ (EXC-2046/1, project ID: 390685689), funded by the Deutsche Forschungsgemeinschaft (DFG) under Germany’s Excellence Strategy. MK is supported by the DFG, German Research Foundation, (Project: 561898308); by a Ram\'on y Cajal grant (RYC2023-043251-I) and the project PID2024-157173NB-I00 funded by MCIN/AEI/10.13039/501100011033, by ESF+, and by FEDER, EU; and by a VII Plan Propio de Investigación y Transferencia (SOL2025-36103) of the University of Sevilla. AM was supported by the grant Excellence 151337 and acknowledges the support of the CDP C2EMPI, as well as the French State under the France-2030 programme, the University of Lille, the Initiative of Excellence of the University of Lille, the European Metropolis of Lille for their funding and support of the R-CDP-24-004-C2EMPI project. DS was supported by the FWF project P 34318 and by the ERC Advanced Grant KnotSurf4d. PT acknowledges the support of the Swiss National Science Foundation Postdoc.Mobility fellowship 230329, and thanks the University of Glasgow for its hospitality and support. LW~acknowledges that this work was supported by the Additional Funding Programme for Mathematical Sciences, delivered by EPSRC (EP/V521917/1) and the Heilbronn Institute for Mathematical Research (HIMR), as well as the Simons Collaboration on New Structures in Low-Dimensional Topology.

\section{Background on knots and sphere embeddings}\label{sec:prelims}

The results presented in this section are well known in the field. However, some of them are difficult to find in the literature, so we provide further discussion. Moreover, we will fix some notation and conventions.

\subsection{Sphere embeddings}\label{sec:background_knots}

Let $M^n$ be an (oriented, connected, compact, smooth) $n$-manifold with $n \geq 3$ and suppose that~$1 \leq k \leq n$. A 
\emph{($k$-)sphere embedding in~$M$} is a smooth embedding $\Sigma \colon S^{k} \hookrightarrow M$, where $S^k$ is endowed with the standard orientation. We will be particularly interested in \emph{knots}, which correspond to the case~$k=n-2$. That is to say, a \emph{knot in $M$} is a smooth embedding $K \colon S^{n-2} \hookrightarrow M^n$. We will study sphere embeddings, and knots in particular, up to \emph{isotopy}, or equivalently, \emph{ambient isotopy}.\footnote{Recall that the isotopy extension theorem (see \eg \cite{hirsch}*{Chapter~8}) implies that two sphere embeddings are ambiently isotopic if and only if they are isotopic.} See also~\Cref{sec:equivalence}.

The image of an embedding $\Sigma \colon S^{k} \hookrightarrow M$ defines an oriented smooth submanifold~$\Sigma (S^{k})$ of~$M$ of codimension~$n-k$. There is a unique closed tubular neighbourhood (up to isotopy of tubular neighbourhoods) of~$\Sigma$ in~$M$~\cite{hirsch}, denoted by $\nu \Sigma$. We will usually think of $\nu \Sigma$ as (the isotopy class of) a submanifold of $M$. We will explicitly stress when we need to keep track of a parametrisation of $\nu \Sigma$, and for this reason we will discuss framings in a subsequent section. We denote the \emph{exterior} of $\Sigma$ by~$E_\Sigma \defeq M \setminus \overset{\circ}{\nu} \Sigma$, where $\overset{\circ}{\nu} \Sigma$ denotes the interior of $\nu \Sigma$. 

The \emph{unknot} $U$ in $S^n$ is the (isotopy class of the) embedding $S^{n-2} \hookrightarrow S^n$ which extends to an embedding $D^{n-1} \hookrightarrow S^n$. By the disc theorem~\cite{palais}, equivalently we could define the unknot to be the (isotopy class of the) embedding given by the composition of standard inclusions $S^{n-2} \hookrightarrow \R^{n-1} \hookrightarrow S^{n-1} \hookrightarrow S^{n}$.

\begin{remark}\label{REM:embeddings}
    Note that there is a subtle difference between the notions of sphere embeddings as we defined them, \ie smooth embeddings $\Sigma \colon S^k \hookrightarrow S^n$, and oriented embedded spheres $\Sigma(S^k) \subseteq S^n$. This difference is related to the fact that the mapping class group~$\MCG(S^k)$ is not trivial in general~\cite{milnor_exotic_7-spheres}, so we can have different embeddings with the same image that are not a priori isotopic. For~$k\in \{1,2,3\}$, since the mapping class group of $S^k$ is trivial~\cites{smale:2-sphereMCG,cerf}, this distinction does not arise. More precisely, for $k\in \{1,2,3\}$, studying parametrised embeddings of~$S^k$ is equivalent to studying oriented embedded $k$-spheres. When $k\geq 4$, the two theories coincide exactly when $S^{k+1}$ has a unique smooth structure, as we now briefly sketch.
    
    On the one hand, if $\Sigma^{k+1}$ is an exotic copy of $S^{k+1}$, then by the $h$-cobordism theorem~\cite{smale}, $\Sigma^{k+1}$ is obtained by gluing two copies of $D^{k+1}$ along some o.p.~diffeomorphism $f$ of $S^{k}$. If $\iota$ denotes the standard embedding of $S^k$ in $S^n$, \ie the one that extends to an embedding $D^{k+1}\hookrightarrow S^n$, then the two sphere embeddings $\iota$ and~$\iota\circ f$ have the same image but are not isotopic. Indeed, their traces $S^{k+1}\times D^{n-k}$ and~$\Sigma^{k+1}\times D^{n-k}$ are not diffeomorphic by the Product Structure Theorem (see Theorem~5.1 and Remark~2 thereafter in~\cite{KirbySiebenmann}*{Essay I}). 
    
    Conversely, assume that~$S^{k+1}$ has no exotic smooth structures, and let~$f,g$ be two embeddings of~$S^k$ in~$S^n$ defining isotopic submanifolds. Without loss of generality, assume that $f$ and $g$ have the same image, and fix an embedding $T\colon S^k\times [0,1]\to S^n$ with the property that~$T_{0}=g$. Such an embedding exists since we can always find a nowhere vanishing section of the normal bundle of $g(S^k)$ in $S^n$, for example as a consequence of \cite{massey59}*{Proof of main theorem}. Next, consider the composition $g^{-1}\circ f$, which is a well-defined o.p.~diffeomorphism of $S^k$ since $f$ and $g$ have the same image. Since $S^{k+1}$ has no exotic smooth structures, $g^{-1}\circ f$ is pseudo-isotopic to the identity, see for example~\cite{kosinski:differential_manifolds}*{Chapter~VIII}. This means that there exists  a diffeomorphism $H\colon S^k\times [0,1]\to S^k\times [0,1]$ with $H_0=g^{-1}\circ f$ and $H_1=\id_{S^k}$. By construction, the composition $T\circ H$ provides an isotopy between $f$ and $T_1$, which in turn is isotopic to $g$ via $T$.
    
    It is known that $S^{k+1}$ has a unique smooth structure for $k\in \{0,1,2, 4, 5, 11, 55, 60\}$, and for $k < 139$ these are the only dimensions (plus perhaps $k =3$) for which~$S^{k+1}$ has a unique smooth structure; see~\cite{behrens2020detecting}*{Corollary~1.3}. By Freedman's work~\cites{freedman:top_4-mfds}, the case $k=3$ is equivalent to the smooth $4$-dimensional Poincaré conjecture.
\end{remark}

As a consequence of \Cref{REM:embeddings}, we define the unknot as above, and not as an embedding of a sphere whose \emph{image} bounds an $(n-1)$-ball as a submanifold. We will call the latter notion an \emph{unknotted} sphere embedding, see \Cref{def:unknotted embedding}. The unknot defines an unknotted sphere embedding, but the converse is generally not true.
    
Of course, many invariants associated to a knot or sphere embedding, such as its exterior or invariants derived from it, do not depend on the parametrisation. On the other hand, we will need to work with our definition using concrete embeddings in order for the notions of surgery and trace (see below) along a sphere embedding~$\Sigma$ to be well-defined. For the construction in \Cref{subsec:RBG_constr_explicit,subsec:nDRBG}, we will first define embedded spheres and later parametrise them.

\subsection{Framings, surgeries and traces}\label{subsec:framings}

For $n \geq 3$ and $1 \leq k \leq n$, let $M^n$ be an $n$-manifold as above and let~$\Sigma \colon S^k \hookrightarrow M^n$ be a 
sphere embedding. A \emph{framing} of $\Sigma$ is a trivialisation of the normal bundle of $\Sigma(S^k)$, \ie an isomorphism $f$ of oriented vector bundles from $TM/T\Sigma(S^k)$ to $\Sigma(S^k)\times \R^{n-k}$. The pair $(\Sigma,f)$ is called a 
\emph{framed sphere embedding}, or, in case of codimension $n-k=2$, a \emph{framed knot}. 

A framed sphere embedding $(\Sigma \colon S^k \hookrightarrow M^n,f)$ determines the isotopy class of an embedding $\hat{\Sigma} \colon S^{k} \times D^{n-k} \hookrightarrow M$. The result of \emph{surgery on $M$ along $(\Sigma,f)$} is the oriented, connected, compact, smooth $n$-manifold
\begin{align*}
    M_f(\Sigma)\defeq (M \setminus \intr \hat{\Sigma}(S^{k} \times D^{n-k})) \cup_{\partial} 
( D^{k+1}\times S^{n-k-1})
\end{align*} 
with gluing map $\hat{\Sigma}|_{S^{k}\times S^{n-k-1}}$. This manifold is well-defined since, up to o.p.~diffeomorphism, it is independent of the choices involved. In the above description of~$M_f(\Sigma)$, there is a preferred embedding of $D^{k+1}\times S^{n-k-1}$, which we can view as a sphere embedding ~$\{0\}\times S^{n-k-1}$ with a preferred framing. If we surger along that framed sphere embedding, we recover~$M$; see \eg~\cite{gompf-stipsicz:book}*{Section 5.2}. This procedure is often called \emph{reversing the surgery}. For~$k=n-2$, we will also call it \emph{loop surgery} and $\gamma_\Sigma \defeq \{0\} \times S^1 \subseteq M_f(\Sigma)$ the \emph{dual curve} of $(\Sigma,f)$. 

We will also use the following notation. Let~$W$ be an $(n+1)$-manifold and let $(\Sigma,f)$ be a framed $k$-sphere embedding in~$\partial W$. Then we denote by
\begin{align*}
    W \cup h_{k+1}(\Sigma)
\end{align*}
the $(n+1)$-manifold that is obtained from $W$ by attaching an $(n+1)$-dimensional $(k+1)$-handle~$h_{k+1}=D^{k+1}\times D^{n-k}$ along the framed $k$-sphere~$(\Sigma,f)$. We denote the belt sphere of $h_{k+1}(\Sigma)$ by $\mathcal{S}_\Sigma$.
Note that $W \cup h_{k+1}(\Sigma)$ inherits a canonical orientation from the orientation of~$W$ and (after smoothing corners) a preferred smooth structure up to diffeomorphism; see~\cite{gompf-stipsicz:book}*{Sections 1.3 and 4.1} for details. We can and we will usually assume that handles are attached in (weakly) increasing order of index; see e.g.~\cite{gompf-stipsicz:book}*{Proposition 4.2.7}.

With this notation at hand, the \emph{trace} of a framed $k$-sphere embedding $(\Sigma,f)$ in $S^n$ is the oriented, connected, compact, smooth $(n+1)$-manifold (again well-defined up to o.p.~diffeomorphism)
\begin{align*}
    X_f(\Sigma) \defeq D^{n+1} \cup h_{k+1}(\Sigma).
\end{align*}
Note that $\partial X_f(\Sigma)=S^n_f(\Sigma)$. As the notation suggests, the oriented diffeomorphism type of the surgery or trace depends on the framed sphere embedding $(\Sigma,f)$, but it is independent of the orientation of $\Sigma$.

\begin{remark}\label{rem:framings}
    Orientable vector bundles of rank $n-k$ over $S^{k}$ are classified by $\pi_{k-1}(\operatorname{SO}(n-k))$; meanwhile, the possible framings of the trivial bundle of rank $n-k$ over $S^{k}$ are classified by~$\pi_{k}(\operatorname{SO}(n-k))$ (see \eg~\cite{gompf-stipsicz:book}*{Chapter~4}). We will mostly be interested in sphere embeddings of codimension $n-k=2$, \ie knots. As a consequence of the above classifications, knots always have a trivial normal bundle if the ambient manifold has dimension~$n\neq 4$, since $\pi_{n-3}(\operatorname{SO}(2))$ is trivial if and only if $n\neq 4$. For~$n = 4$, an embedded sphere $\Sigma(S^2)$ in $M$ has a trivial normal bundle if and only if its homological self intersection number~$\Sigma \cdot \Sigma$ vanishes. The latter is automatically fulfilled if $H_2(M)$ vanishes; for example, if $M=S^4$. Since~$\pi_{n-2}(\operatorname{SO}(2))$ is trivial if and only if $n\neq 3$, a knot $K \colon S^{n-2} \hookrightarrow M^n$ with trivial normal bundle has a \emph{unique} framing (up to isotopy) if and only if $n \neq 3$.  For~$n=3$, the framings are in a non-canonical bijection with the integers; see~\cite{gompf-stipsicz:book}*{Section~4.5} for further discussion.
\end{remark}

\textbf{Convention:} Whenever there is a unique choice of framing for a sphere embedding~$\Sigma$, for example if it is a knot in an ambient manifold $M^n$ of dimension $n \neq3$ as in \Cref{rem:framings}, we will just omit the framing $f$ from the notation of $(\Sigma,f)$, $M_f(\Sigma)$ and~$X_f(\Sigma)$.

\subsection{Tubing}
\label{subsubsec:tubing}

For $n \geq 4$, let $K_1$ and $K_2$ be knots in an $n$-manifold~$M$. Suppose that~$\beta \colon D^1 \hookrightarrow M$ is an arc whose endpoints lie on $K_1(S^{n-2})$ and $K_2(S^{n-2})$, and whose image is otherwise disjoint from the image of $K_1$ and $K_2$. Thicken $\beta$ to an embedding $h_{\beta} \colon D^1 \times D^{n-2} \hookrightarrow M$ such that the image of $\partial D^1 \times D^{n-2}$ lies on~$K_1(S^{n-2})$ and $K_2(S^{n-2})$. Now, modify $K_1$ and $K_2$ by removing the pair of balls $h_{\beta}(\partial D^1 \times D^{n-2})$ from their images and gluing in the annulus $h_{\beta}(D^1 \times \partial D^{n-2})$. This operation is called \emph{tubing~$K_1$ and~$K_2$ along $\beta$}. We denote the resulting submanifold of $M$ by~$K_1 \#_{\beta} K_2$. We will always choose the embedding $h_\beta$ such that $K_1 \#_{\beta} K_2$ has an orientation that restricts to the one of $K_1(S^{n-2})$ and $K_2(S^{n-2})$. Notice that such thickenings $h_{\beta}$, as submanifolds, are parametrised by isotopy classes of paths with fixed endpoints in~$\text{Gr}^+_{n-2}(\mathbb{R}^{n-1})$, the Grassmannian of oriented $(n-2)$-planes in $\mathbb{R}^{n-1}$, using uniqueness of tubular neighbourhoods. Since $\text{Gr}^+_{n-2}(\mathbb{R}^{n-1})$ is simply connected for $n \geq 4$, it follows that there is a unique choice of thickening~$h_{\beta}$ up to isotopy. This implies that, up to ambient isotopy, tubing is a well-defined operation between oriented submanifolds. More precisely, the isotopy class of the resulting submanifold only depends on the homotopy class of the arc~$\beta$ rel boundary.

If one considers tubing as an operation on parametrised knots, \ie if one wants to define a new embedding $S^{n-2} \hookrightarrow S^n$, then there are a priori two ways to tube~$K_1$ and $K_2$ along $\beta$. See \eg \cite{LBT}*{Remark 5.3}, where this operation of tubing is discussed for surfaces in $4$-manifolds.

Note that the above described operation is usually called \emph{band sum} in dimension~$n=3$. In this case, we have to be careful to specify the embedding~$h_{\beta}$, for which we have infinitely many choices parametrised by $\Z$.

\subsection{Meridians and the knot group}\label{subsec:meridians}

Let $K \colon S^{n-2} \hookrightarrow M^n$ be a knot in $M$ with trivial normal bundle. The \emph{meridian} of~$K$ is an oriented simple closed curve~$\mu_K$ in~$M$ that is isotopic to the oriented boundary $\{\text{pt}\} \times \partial D^2 $ of a transverse disc~$\{\text{pt}\} \times  D^2 $ in the tubular neighbourhood $\nu K \cong S^{n-2} \times D^{2}$ in $M$. Sometimes we will also denote by $\mu_K$ the meridian seen as a curve in the knot exterior $E_K$. Note that the \emph{knot group} $\pi_1(K) \defeq \pi_1(E_K)$ (but not its isomorphism type) depends on the choice of a base point~$x_0 \in M$. In general, the meridian~$\mu_K$ is not a curve based at $x_0$. So to see it as an element of~$\pi_1(K)$, we choose a path~$w$ from~$x_0$ to the base point of the curve~$\mu_K$. Different choices of $w$ lead to a conjugation, so the conjugacy class of $\mu_K$ in the knot group is well-defined. By abusing notation, we will also write~$\mu_K$ for the conjugacy class of the push-forward of the meridian of~$K$ into the knot group.

\subsection{Ribbon knots}\label{sec:ribbon}

In \Cref{subsec:RBG_constr_explicit}, we will build knots by doubling embedded ribbon discs. Let us make this more precise.

A knot $K \colon S^{n-2} \hookrightarrow S^n$ is \emph{ribbon} if it bounds an immersed disc~$D \colon D^{n-1} \looparrowright S^n$ with only ribbon singularities. This means that the components of the singular set of~$D$ are $(n-2)$-discs whose boundary $(n-3)$-spheres either lie on $D(\partial D^{n-1})$ or are disjoint from it. We call $D$ a \emph{ribbon disc}. Pushing the interior of~$D(D^{n-1})$ appropriately into the ball $D^{n+1}$ bounded by $S^n$ produces an embedded disc $D^{n-1} \hookrightarrow D^{n+1}$ with boundary $K$, which we call an \emph{embedded ribbon disc} and, by slight abuse of notation, also denote by $D$. Forming the \emph{double} of this embedded ribbon disc $D \colon D^{n-1}\hookrightarrow D^{n+1}$, \ie the union of two copies of $D$ in~$D^{n+1}$ glued along their boundaries using the identity map, gives rise to a ribbon $(n-1)$-knot
in~$S^{n+1}$; see for example~\cite{Carter_Kamada_Saito}*{Section 2.2.1}.

\subsection{Equivalence relations on knots} 
\label{sec:equivalence}

For classical knots, \ie knots in $S^3$, it is well known that (o.p.) diffeomorphisms of pairs and (ambient) isotopies are equivalent notions. One way of proving the harder direction is to use Cerf's result that the mapping class group of $S^3$ is trivial~\cite{cerf}. However, as we will see, there is a more direct way of proving that these two notions of equivalence of knots in~$S^3$ coincide. This also works in higher dimensions, where the mapping class groups of spheres are generally non-trivial~\cites{milnor_exotic_7-spheres}. We start by precisely defining the notions we consider. 

\begin{definition}\label{def:equivalence}
Let $n \geq 3$ and $1 \leq k \leq n$. Two framed $k$-sphere embeddings $(\Sigma_1,f_1)$ and $(\Sigma_2,f_2)$ in an $n$-manifold~$M^n$ are \emph{equivalent} if there exists an o.p.~diffeomorphism~$h \colon M \to M$ such that $h\circ \Sigma_1 =\Sigma_2$ and the following diagram commutes:
\[
\begin{tikzcd}
    TM/T\Sigma_1(S^k) \arrow[r, "f_1"] \arrow[d, "\text{D}\,h"] & \Sigma_1(S^k)\times \R^{n-k} \arrow[d, "h\times \text{id}"]\\
    TM/T\Sigma_2(S^k) \arrow[r, "f_2"] & \Sigma_2(S^k)\times \R^{n-k}.
\end{tikzcd}
\] 
We will call such a map $h$ an \emph{equivalence} between $(\Sigma_1, f_1)$ and $(\Sigma_2, f_2)$.
\end{definition}

Note that the term ``equivalence'' is also used slightly differently in the literature when knots are defined as submanifolds.

\begin{definition}\label{def:amb_iso}
Let $n \geq 3$ and $1 \leq k \leq n$. Two framed $k$-sphere embeddings $(\Sigma_1,f_1)$ and $(\Sigma_2,f_2)$ in an $n$-manifold~$M^n$ are \emph{ambiently isotopic} if there exists an equivalence between $(\Sigma_1,f_1)$ and $(\Sigma_2,f_2)$ which is isotopic to $\operatorname{id}_M$. 
\end{definition}

The following lemma is a standard result for unframed spheres, which we found discussed at~\cite{mathoverflow}. We provide a proof in the setting of framed sphere embeddings, as this is what we will need presently.

\begin{lemma}\label{lem:equivalence_implies_iso}
    Let $n \geq 3$ and $1 \leq k \leq n$. Two framed $k$-sphere embeddings $(\Sigma_1,f_1)$ and~$(\Sigma_2,f_2)$ in $S^n$ are equivalent if and only if they are ambiently isotopic. 
    \end{lemma}

\begin{proof}
    Observe that there is only one non-trivial implication, so we suppose that the framed sphere embeddings~$(\Sigma_1,f_1)$ and $(\Sigma_2,f_2)$ are equivalent via a diffeomorphism~$h\colon S^n\to S^n$. Think of $S^n$ as the union of two $n$-balls $D^n_{N}$ and $D^n_{S}$ glued along their boundary. Up to ambient isotopy of framed sphere embeddings, the images of both~$\Sigma_1$ and~$\Sigma_2$ can be assumed to be contained in $D^n_{S}$. By the disc theorem of Palais~\cite{palais},
    the diffeomorphism $h$ can be isotoped to restrict to the identity on $D^n_{N}$. Let $r_t\colon S^n\to S^n$, $t\in [0,1]$, be a~$1$-parameter family of rotations such that~$r_0 = \id_{S^n}$ and $r_1$ swaps $D^n_{S}$ and $D^n_{N}$. We now consider the family of framed 
    sphere embeddings~$(\widetilde{\Sigma}_t, \widetilde{f}_t)$ where $\widetilde{\Sigma}_t= h\circ r_t \circ \Sigma_1$ and the framing $\widetilde{f}_t$ is induced by the framing~$f_1$ via~$h\circ r_t$. By construction, these framed sphere embeddings are all ambiently isotopic. Moreover, $(\widetilde{\Sigma}_0, \widetilde{f}_0)= (\Sigma_2, f_2)$ and~$(\widetilde{\Sigma}_1,\widetilde{f}_1)$ is just $r_1\circ\Sigma_1$ with the framing induced by~$r_1$, since~$h$ is the identity on $D^n_{N}$. Clearly, $(\Sigma_1,f_1)$ is ambiently isotopic to $(\widetilde{\Sigma}_1, \widetilde{f}_1)$, and thus to~$(\Sigma_2,f_2)$. 
\end{proof}

For knots with a trivial normal bundle in $M^n$ with $n \ge 4$, we will say, with a slight abuse of notation, that two framed knots are \emph{isotopic} (rather than ambiently isotopic), even though the notion of isotopy was defined only in the context of sphere embeddings without a framing. Indeed, any ambient isotopy that extends a given isotopy will respect the framings, since they are unique when $n \ge 4$ (see \Cref{rem:framings}).

\subsection{Dimensions and codimensions of sphere embeddings}\label{subsec:general_dim}

The Schoenflies problem posits that, for any embedded $(n-1)$-sphere $\Sigma$ in $S^n$ (\ie of codimension~$1$), there is a diffeomorphism of pairs $(S^n, \Sigma)$ and $(S^n, S^{n-1})$, where $S^{n-1}\subset S^n$ denotes the equatorial $(n-1)$-sphere. This holds for $n \leq 3$ and~$n \geq 5$, the latter as an application of the $h$-cobordism theorem in these dimensions due to Smale~\cite{smale}; see~\cite{milnor:lectures_h-cob}*{Section~9}.\footnote{The Schoenflies problem for topological, locally flat embeddings was proved by Brown~\cite{brown} in any dimension~$n \geq 1$.} Isotopy classes of embeddings $S^{k} \hookrightarrow S^n$ of codimension $n-k=1$ are thus only interesting for $n=4$. We will not address this difficult remaining case of the Schoenflies problem here. 

On the other hand, embeddings of high codimensions $n -k \geq 3$ are well-understood. In particular, isotopy classes of embeddings $S^k\hookrightarrow S^n$ together with the connected sum form a group in codimension at least $3$. For example, any two smooth embeddings~$S^1 \hookrightarrow S^n$ are isotopic when $n \geq 4$; see \eg~\cite{hirsch}. In the piecewise linear and topological categories, Zeeman~\cite{Zeeman:unknotting} and Stallings~\cite{stallings:unknotting} proved that embeddings $S^k \hookrightarrow S^n$ of codimension~$n-k \geq 3$ are unknotted (cf.~\Cref{def:unknotted embedding}). However, in the smooth category, knotting can happen even in codimension~$n-k \geq 3$, as shown by Haefliger~\cite{haefliger:knotted-spheres}. Still, smooth embeddings $S^k \hookrightarrow S^n$ with $n-k \geq 3$ were classified by Haefliger~\cite{haefliger_diff_embeddings} and Levine \cite{levine:class_diff_knots}. Indeed, the former \cite{haefliger:plongements} showed that the group of isotopy classes of smooth $k$-sphere embeddings in $S^n$ is trivial when~$2n > 3(k+1)$, but non-trivial when~$2n = 3(k+1)$. 

This leaves the case of sphere embeddings of codimension $2$, for which no simple classification is known. In particular, the set of isotopy classes of codimension $2$ embeddings does not form a group under connected sum. The rich theory of classical knots~$S^1 \hookrightarrow S^3$ illustrates the complexity of this setting. In this article, we address some of the questions that arise when we broaden our focus to~$n\geq4$.

\section{Constructing sphere embeddings with the same trace}\label{sec:construction}

\subsection{The generalised RBG construction}\label{subsec:RBG} 

In this section, we describe a general method to construct pairs of sphere embeddings with o.p.~diffeomorphic traces. As mentioned in the introduction, this construction is motivated by similar $3$-dimensional constructions and can be thought of as a generalisation of these. In particular, the RBG link constructions from \cites{piccirillo:shake-genus,piccirillo:conway-knot} are relevant, while our notation is inspired by~\cites{manolescu-piccirillo:RBG,Tagami}. Other constructions for creating knots that share a surgery or trace in dimension $3$ can be found for example in \cites{Lickorish_sharing_surgery,Akbulut,brakes:multiple,Livingston_surgery,Motegi_surgery,Gordon-Luecke,Osoinach_annulus,Teragaito,AJLO2,AJLO,yasui:corks-concordance,MillerPiccirillo,McCoy_char_hyp,Abe_Tagami}.

\begin{definition}\label{def:RBG_mfd}
    We call an $(n+1)$-manifold $W$ an \emph{RBG manifold} if $W$ comes equipped with a handle decomposition of the form
    \begin{align*}
        W \cong D^{n+1}\cup h_k(R) \cup h_{k+1}(B) \cup h_{k+1}(G),
    \end{align*}
    for some $1\leq k\leq n$, where $R$ is a framed $(k-1)$-sphere embedding in $\partial D^{n+1}$ and~$B,\, G$ are framed 
    $k$-sphere embeddings in $\partial (D^{n+1}\cup h_k(R))$ such that there exists
    \begin{enumerate}[label=(\roman*)]
        \item an isotopy of $B(S^k)$ in $\partial (D^{n+1}\cup h_k(R))$ after which it intersects the belt sphere~$\mathcal{S}_R$ of $h_k(R)$ transversely in a single point;
        \item an isotopy of $G(S^k)$ in $\partial (D^{n+1}\cup h_k(R))$ after which it intersects the belt sphere~$\mathcal{S}_R$ of $h_k(R)$ transversely in a single point.
    \end{enumerate}
\end{definition}

Note that in \Cref{def:RBG_mfd} we do not assume that $B(S^k)$ and $G(S^k)$ \emph{simultaneously} intersect~$\mathcal{S}_R$ transversely in a single point. 

Any $(n+1)$-dimensional RBG manifold yields a pair of (not necessarily non-isotopic) $k$-sphere embeddings in $S^n$ that have o.p.~diffeomorphic traces, as follows.
By hypothesis, after an isotopy, the image of $B$ intersects $\mathcal{S}_R$ transversely in a single point. Thus the handles $h_k(R)$ and $ h_{k+1}(B)$ cancel. The same holds true if we replace $B$ by $G$. We thus obtain diffeomorphisms
\begin{align}\label{eq:f_G_and_f_B}
    \begin{split}
        f_G &\colon D^{n+1} \cup h_k(R)\cup h_{k+1}(B) \longrightarrow
        D^{n+1}, \\
        f_B &\colon D^{n+1} \cup h_k(R)\cup h_{k+1}(G) \longrightarrow 
        D^{n+1}. 
    \end{split}
\end{align}
Now, we define the framed $k$-sphere embeddings $\Sigma_B$ and $\Sigma_G$ in $\partial D^{n+1}=S^n$ as the compositions of $B$ and $G$ with the diffeomorphisms $f_B$ and $f_G$,  
\ie 
\begin{align*}
    \Sigma_B\defeq f_B\circ B
    \quad \text{and} \quad 
    \Sigma_G\defeq f_G \circ G
\end{align*}   
with framings induced by those of $B$ and $G$.
By construction, the traces of $\Sigma_B$ and~$\Sigma_G$ are both o.p.~diffeomorphic to $W$ and thus to each other, as summarised in the next proposition.

\begin{proposition}\label{thm:rbg_knots}
    To any $(n+1)$-dimensional RBG manifold $W$ as in \Cref{def:RBG_mfd}, we can assign the (well-defined)  ambient isotopy classes of two framed $k$-spheres $\Sigma_B$ and $\Sigma_G$ in $S^n$ whose traces are both o.p.~diffeomorphic to $W$ and thus to each other.
\end{proposition}

\begin{proof}
    The diffeomorphisms $f_B$ and $f_G$, and hence $\Sigma_B$ and $\Sigma_G$, a priori depend on the choice of the isotopies of $B$ and $G$ and the exact cancellation of the handles. However, we now show that, up to ambient isotopy,~$\Sigma_B$ and $\Sigma_G$ are independent of these choices. 
    We describe the argument for~$\Sigma_G$; the same works for $\Sigma_B$. Let 
    \begin{align*}
        f_G, \, f'_G \colon D^{n+1} \cup h_k(R)\cup h_{k+1}(B) \longrightarrow
        D^{n+1}   
    \end{align*}
    be two diffeomorphisms as in \eqref{eq:f_G_and_f_B} yielding sphere embeddings $\Sigma_G$ and $\Sigma_G'$. Restricting the composition $f_G \circ (f'_G)^{-1}$ to $\partial D^{n+1}$, we obtain a diffeomorphism $S^n \to S^n$ which gives rise to an equivalence between the (framed) sphere embeddings $\Sigma_G' = f_G' \circ G$ and $\Sigma_G = f_G \circ G$ (see~\Cref{def:equivalence}). Thus $\Sigma_G$ and~$\Sigma_G'$ are ambiently isotopic by \cref{lem:equivalence_implies_iso}.
\end{proof}

\begin{remark}\label{rem:surgeries}
    A similar construction as in \Cref{def:RBG_mfd} and \Cref{thm:rbg_knots}, but more general, also works to construct sphere embeddings with diffeomorphic surgeries, but a priori unrelated traces. Let $R$ be a framed $s$-sphere in $S^n$ and let $B$ and $G$ be framed $k$-spheres in the manifold obtained by surgery on~$S^n$ along $R$ such that 
    \begin{enumerate}[label=(\roman*)]
        \item surgery along $B$ yields a manifold diffeomorphic to $S^n$ via a diffeomorphism~$f_G$;
        \item surgery along $G$ yields a manifold diffeomorphic to $S^n$ via a diffeomorphism~$f_B$.
    \end{enumerate}
    Then the framed $k$-spheres $\Sigma_B=f_B(B)$ and $\Sigma_G=f_G(G)$ in $S^n$ have o.p.~diffeomorphic surgeries. This construction is more in line with the approach of~\cite{manolescu-piccirillo:RBG}. In \Cref{sec:traces_Gluck}, and in particular in \Cref{Glucktwist}, we will study the relation between surgeries and traces of $(n-2)$-sphere embeddings
    in dimensions $n \geq 4$.
\end{remark}

\subsection{Trivial cases}\label{subsec:trivial_cases}

In the construction from \Cref{subsec:RBG}, it might very well happen that the framed $k$-sphere embeddings~$\Sigma_B$ and $\Sigma_G$ in $S^n$ are actually ambiently isotopic. Indeed, this will happen if the RBG manifold~$W$ is too simple or symmetric, or for certain ranges of $k$. For example, $\Sigma_B$ and $\Sigma_G$ are certainly isotopic when they arise as embeddings~$S^k \hookrightarrow S^n$ for $2n > 3(k+1)$~\cite{haefliger:plongements}. See also the discussion in~\Cref{subsec:general_dim}. Another common situation that arises is as follows.

\begin{lemma}\label{lemma:simult_dual}
    For $n \geq 3$, let $W$ be an $(n+1)$-dimensional RBG manifold such that (possibly after an isotopy) the $k$-spheres~$B(S^k)$ and $G(S^k)$ simultaneously intersect the belt sphere $\mathcal{S}_R$ transversely in a single point. Then $\Sigma_B$ and $\Sigma_G$ are ambiently isotopic as unoriented submanifolds of $S^n$.
\end{lemma}
   
\begin{proof}
    If $B(S^k)$ and $G(S^k)$ simultaneously intersect $\mathcal{S}_R$ transversely in a single point, then we can choose an arc $\beta$ connecting this pair of intersection points which lies entirely in $\mathcal{S}_R$. Take a tubular neighbourhood of $\beta$ in $\partial W$ of the form $D^k \times [0,1] \times D^{n-k-1}$, where $[0,1] \times D^{n-k-1} \subset \mathcal{S}_R$. This is possible since $\mathcal{S}_R$ is $(n-k)$-dimensional. Consider the ``tube'' $A=\partial D^k \times [0,1] \times \{0\}$. Take a push-off of the attaching sphere $G(S^k)$ and slide it following $A$ over the handle $h_{k+1}(B)$. This yields a $k$-sphere $S$ which we consider as an unoriented submanifold. Note that $S$ is just the result of tubing push-offs of $B(S^k)$ and $G(S^k)$ along $A$. The image of $S$ under the diffeomorphism~$f_G$ which cancels $h_{k+1}(B)$ and~$h_k(R)$ (see \eqref{eq:f_G_and_f_B}) is isotopic, as an unoriented submanifold of $S^n$, to~$\Sigma_G$.
    
    We now reverse the roles of~$B$ and $G$ and apply the same argument. The push-off of~$B$ that we take matches the one which was previously involved in the tubing resulting in $S$; similarly, the slide of this sphere over $h_{k+1}(G)$ implicitly uses the same push-off of $G$ that we started with. Since we also use the same arc $\beta$, the resulting $k$-sphere is again $S$. We deduce as before that the image of $S$ under $f_B$ is isotopic, as unoriented submanifold of $S^n$, to~$\Sigma_B$.

    In summary, the diffeomorphism $f_G\circ f_B^{-1}$, restricted to $\partial D^{n+1}$, sends $f_B(S)$ to~$f_G(S)$ and thus can be used to define an equivalence of unoriented submanifolds between~$\Sigma_B$ and $\Sigma_G$, which are isotopic to $f_B(S)$ and $f_G(S)$, respectively. We can promote it to an ambient isotopy using \Cref{lem:equivalence_implies_iso}.
\end{proof}

However, there is no reason to expect that $\Sigma_B$ and $\Sigma_G$ are isotopic in general.

\subsection{Constructing RBG manifolds with \texorpdfstring{$\boldsymbol{k=n-2}$}{k=n-2}.}
\label{subsec:nDRBG}

The most interesting case in the construction of framed $k$-sphere embeddings in $S^n$ from $(n+1)$-dimensional RBG manifolds~$W$ is that of codimension~$2$, \ie when $k=n-2$; see the discussion in~\Cref{subsec:general_dim,subsec:trivial_cases}.
In this case, $W$ has a handle decomposition of the form
\begin{align}\label{eq:RBG_k=n-2}
    W = D^{n+1}\cup h_{n-2}(R) \cup h_{n-1}(B) \cup h_{n-1}(G).
\end{align}
The challenge here is to construct a sufficiently complicated RBG manifold $W$ such that the associated framed 
$k$-sphere embeddings, in this case (framed) knots $K_B\defeq \Sigma_B$ and $K_G\defeq \Sigma_G$ in $S^n$, are non-isotopic. At the same time, $W$ should be simple enough to allow us to effectively compute invariants that distinguish~$K_B$ and~$K_G$. In the following, we present a method to construct interesting RBG manifolds as in \eqref{eq:RBG_k=n-2}, where we expect~$K_B$ and $K_G$ to usually be non-isotopic in $S^n$ (see also \Cref{rem:sim_geom_dual}). In \Cref{subsec:RBG_constr_explicit}, we will use this general method to explicitly construct RBG manifolds with non-isotopic associated knots. Throughout this section, let~$n \geq 4$.

\subsubsection{Construction of \texorpdfstring{$R$}{R}}\label{General R construction}

Let $R$ be the embedding induced by the standard inclusions $S^{n-3}\hookrightarrow \R^{n-2}\hookrightarrow \R^n \hookrightarrow S^{n}$. We choose the framing for $R$ for which we obtain $D^{n+1} \cup h_{n-2}(R) \cong D^3 \times S^{n-2}$ with boundary $\partial (D^{n+1}\cup h_{n-2}(R)) \cong S^2 \times S^{n-2}$. Fix a point $z \in S^{n-2}$ and write~$\mathcal{S}_R = S^2 \times \{z\} \subseteq S^2 \times S^{n-2}$ for the belt sphere of~$h_{n-2}(R)$. 

Next, we will construct knots $B$ and $G$ in $S^2\times S^{n-2}$, so that after an isotopy, the image of each of $B$ and $G$ is \emph{geometrically dual} to $\mathcal{S}_R$, \ie after isotopies, $B(S^{n-2})$ and~$\mathcal{S}_R$, or~$G(S^{n-2})$ and~$\mathcal{S}_R$, respectively, intersect transversely in a single point. We will first define~$B$ and $G$ as embedded submanifolds of $\partial(D^{n+1} \cup h_{n-2}(R)) \cong S^2 \times S^{n-2}$, and then choose an embedding representing these submanifolds (see also \Cref{REM:embeddings}). Since there is a unique choice of framing for the spheres involved in the construction of~$B$ and $G$ (see \Cref{rem:framings}), we will usually just omit the framing from the discussion. 

For $\ell \geq 1$, choose $2\ell+2$ distinct points $w, x_1, \dots, x_{2\ell+1}$ in $S^2$ and consider the $(n-2)$-spheres $\{w\} \times S^{n-2}$ and $\{x_i\} \times S^{n-2}$, for $i \in \{1, \dots, 2\ell+1\}$, in~$S^2 \times S^{n-2}$.  
We will describe how to construct $B$ and $G$ by modifying these spheres.

\subsubsection{Construction of the knot \texorpdfstring{$B$}{B} in \texorpdfstring{$S^2 \times S^{n-2}$}{S^2 times S^{n-2}}}\label{General B construction}

To construct~$B$, we will add a local knot to $\{w\} \times S^{n-2}$; see \Cref{fig:general_B_construction} for a schematic. More precisely, let $K$ be a knot in~$S^n$. The sphere $B$ is obtained as the pairwise connected sum
\begin{align*}
    (S^2\times S^{n-2},B) \defeq (S^2\times S^{n-2}, \{w\}\times S^{n-2}) \# (S^n, K).
\end{align*}
Note that we perform the connected sum away from the spheres $\{x_i\} \times S^{n-2}$, $i \in \{1, \dots, 2\ell+1\}$, and away from $\mathcal{S}_R = S^2 \times \{z\}$.

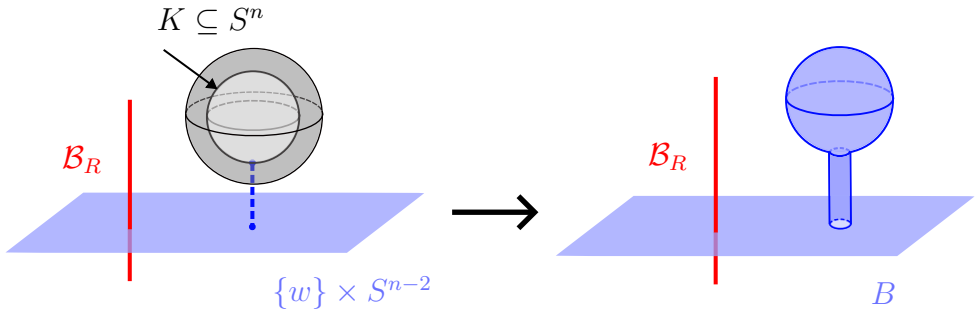
\begin{figure}[htbp] 
	\centering
    \def\svgwidth{0,85\columnwidth}
	%% Creator: Inkscape 1.4.2 (f4327f4, 2025-05-13), www.inkscape.org
%% Accompanies image file 'G_constr_arc_config_v2.pdf' (pdf, eps, ps)
\begingroup%
  \makeatletter%
  \providecommand\color[2][]{%
    \errmessage{(Inkscape) Color is used for the text in Inkscape, but the package 'color.sty' is not loaded}%
    \renewcommand\color[2][]{}%
  }%
  \providecommand\transparent[1]{%
    \errmessage{(Inkscape) Transparency is used (non-zero) for the text in Inkscape, but the package 'transparent.sty' is not loaded}%
    \renewcommand\transparent[1]{}%
  }%
  \providecommand\rotatebox[2]{#2}%
  \newcommand*\fsize{\dimexpr\f@size pt\relax}%
  \newcommand*\lineheight[1]{\fontsize{\fsize}{#1\fsize}\selectfont}%
  \ifx\svgwidth\undefined%
    \setlength{\unitlength}{427.18351589bp}%
    \ifx\svgscale\undefined%
      \relax%
    \else%
      \setlength{\unitlength}{\unitlength * \real{\svgscale}}%
    \fi%
  \else%
    \setlength{\unitlength}{\svgwidth}%
  \fi%
  \global\let\svgwidth\undefined%
  \global\let\svgscale\undefined%
  \makeatother%
  \begin{picture}(1,0.3315743)%
    \lineheight{1}%
    \setlength\tabcolsep{0pt}%
    \put(0,0){\includegraphics[width=\unitlength,page=1]{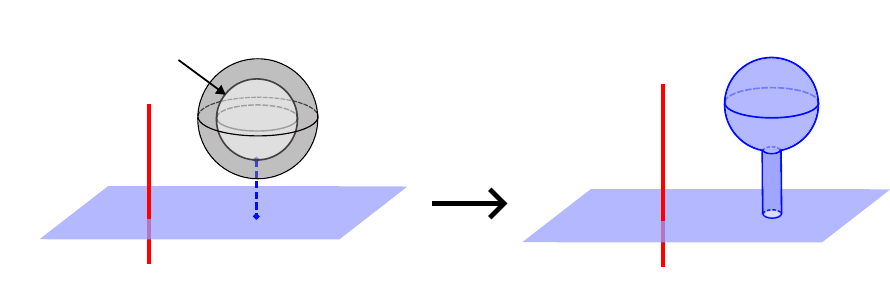}}%
    \put(0.16062532,0.2874296){\color[rgb]{0,0,0}\makebox(0,0)[lt]{\lineheight{1.25}\smash{\begin{tabular}[t]{l}$K\subseteq D^n$\end{tabular}}}}%
    \put(0.25771446,0.01413972){\color[rgb]{0.4,0.43137255,1}\makebox(0,0)[lt]{\lineheight{1.25}\smash{\begin{tabular}[t]{l}$\{w\}\times S^{n-2}$\end{tabular}}}}%
    \put(0.84512177,0.01131114){\color[rgb]{0.4,0.43137255,1}\transparent{0.98039198}\makebox(0,0)[lt]{\lineheight{1.25}\smash{\begin{tabular}[t]{l}$B$\end{tabular}}}}%
    \put(0.68452276,0.17273754){\color[rgb]{1,0,0}\makebox(0,0)[lt]{\lineheight{1.25}\smash{\begin{tabular}[t]{l}$\mathcal{S}_R$\end{tabular}}}}%
    \put(0.10893019,0.17273754){\color[rgb]{1,0,0}\makebox(0,0)[lt]{\lineheight{1.25}\smash{\begin{tabular}[t]{l}$\mathcal{S}_R$\end{tabular}}}}%
  \end{picture}%
\endgroup%

	\caption{Construction of the knot $B$ in $S^2 \times S^{n-2}$ as connected sum of a local knot $K$ and $\{w\} \times S^{n-2}$. The belt sphere $\mathcal{S}_R = S^2 \times \{z\}$ is in red. 
    }
	\label{fig:general_B_construction}
\end{figure}

\subsubsection{Construction of the knot \texorpdfstring{$G$}{G} in \texorpdfstring{$S^2 \times S^{n-2}$}{S^2 times S^{n-2}}}\label{General G construction}

To construct $G$, we begin by tubing together the spheres $\{x_i\} \times S^{n-2}$, $i \in \{1, \dots, 2\ell+1\}$, along standard arcs. More precisely, choose orientations of these spheres such that all $\{x_i\} \times S^{n-2}$ for odd~$i \in \{1, \dots, 2\ell+1\}$ have the same orientation and such that the orientation of each~$\{x_i\} \times S^{n-2}$ for even $i \in \{1, \dots, 2\ell+1\}$ differs from that of $\{x_1\} \times S^{n-2}$. 

Now, we take a collection of $2\ell$ arcs $\beta_i\colon D^1 \hookrightarrow S^2 \times S^{n-2}$, $i \in \{1, \dots, 2\ell\}$, such that, for each $i$, the two endpoints of $\beta_i$ lie on $\{x_i\} \times S^{n-2}$ and $\{x_{i+1}\} \times S^{n-2}$, respectively. These arcs can be chosen ``straight'' from $\{x_i\} \times S^{n-2}$ to $\{x_{i+1}\} \times S^{n-2}$, disjoint from each other and disjoint from $B$ and $\mathcal{S}_R$. More precisely, we can choose $2\ell$ arcs in~$S^2$ that connect the points $x_i$ and $x_{i+1}$ in $S^2$ for $i \in \{1, \dots, 2\ell\}$, and use disjoint push-offs of these to define the arcs $\beta_i$ in $S^2 \times S^{n-2}$ (see Figure \ref{fig:general_G_construction}). Tubing the spheres $\{x_i\} \times S^{n-2}$, $i \in \{1, \dots, 2\ell+1\}$, along the arcs $\beta_i$, $i \in \{1, \dots, 2\ell\}$, as explained in \Cref{subsubsec:tubing}, we obtain an $(n-2)$-sphere $\widetilde{G}$ which is disjoint from $B$ and intersects~$\mathcal{S}_R=S^2 \times \{z\}$ transversely in $2\ell+1$ points, but algebraically only once. 

\begin{figure}[htbp] 
	\centering
	\def\svgwidth{0,85\columnwidth}
	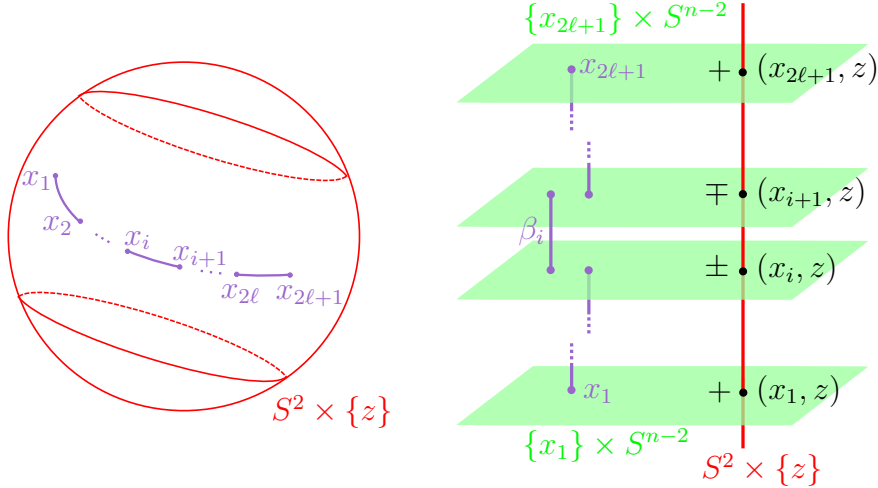
	\caption{Choice of arcs in $S^2 \times \{z\}$ that give rise to ``straight'' arcs~$\beta_1, \dots, \beta_{2\ell}$ in $S^2 \times S^{n-2}$. The knot $\widetilde{G}$ in $S^2 \times S^{n-2}$ is obtained by tubing $\bigcup \{x_i\} \times S^{n-2}$ along the arcs $\beta_i$.
    }
	\label{fig:general_G_construction}
\end{figure}

Next, we will modify the collection of arcs $\{\beta_i\}$ in $S^2 \times S^{n-2}$ into a new collection of arcs $\{\beta'_i\colon D^1 \hookrightarrow S^2 \times S^{n-2}\}$ by linking them (possibly trivially) with $B$. Explicitly, choose $2\ell$ disjoint embeddings $\gamma_i\colon S^1 \hookrightarrow S^n\setminus K$, $i \in \{1, \dots, 2\ell\}$, where $K$ is the local knot used to construct~$B$ (see Construction~\ref{General B construction}). The embeddings~$\gamma_i$, $i \in \{1, \dots, 2\ell\}$, are all contained in~$S^n$ and as such can be assumed to be disjoint from $B$ in $S^2\times S^{n-2} = S^2\times S^{n-2}\# S^n$, and also disjoint from $\{x_j\}\times S^{n-2}$, $j \in \{1, \dots, 2\ell +1\}$.
For each $i \in \{1, \dots, 2\ell\}$, modify the arc $\beta_i$ by performing a band sum with $\gamma_i$ along any arc $\delta_i$ disjoint from the $\{x_j\} \times S^{n-2}$, $j \in \{1, \dots, 2\ell +1\}$, and disjoint from $B$ to obtain an arc $\beta'_i\defeq \beta_i\#_{\delta_i} \gamma_i$.
See \Cref{F: Arcmodification}, where $\beta_i, \gamma_i$ and $\delta_i$ are schematically depicted on the left; and the band sum $\beta'_i =  \beta_i\#_{\delta_i} \gamma_i$ in the centre. We define $G$ to be the sphere obtained by tubing the spheres~$\{x_i\}\times S^{n-2}$ along the arcs~$\beta'_i$. Note that $G$ a priori depends on the choices of the loops $\gamma_i$ and the arcs~$\delta_i$. However, the next lemma shows that, up to isotopy in $S^2\times S^{n-2}$ (\ie ignoring~$B$), it does not.

\begin{figure}[htbp] 
	\centering
	\includegraphics[width=1\textwidth]{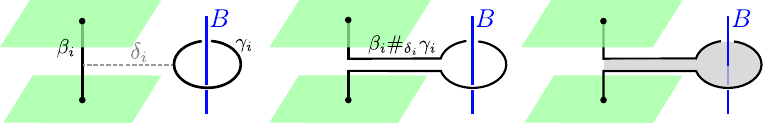}
	\caption{Left: Arcs $\beta_i$ and $\delta_i$ in $S^2 \times S^{n-2}$ together with the embedding $\gamma_i\colon S^1 \hookrightarrow S^n\setminus K$. Centre: Band sum $\beta_i\#_{\delta_i} \gamma_i$ of $\beta_i$ and $\gamma_i$ along $\delta_i$. Right: The
    region giving an isotopy between $\beta_i\#_{\delta_i} \gamma_i$ and $\beta_i$ in $S^2 \times S^{n-2}$ is shaded in grey.}
	\label{F: Arcmodification}
\end{figure}

\begin{lemma}\label{lem:GequalsGtilda}
    The spheres $G$ and $\widetilde{G}$ are both isotopic to $\{x_1\} \times S^{n-2}$ in $S^2 \times S^{n-2}$. In particular, up to isotopy in $S^2 \times S^{n-2}$, the sphere $G$ is independent of the choices of the loops~$\gamma_i$ and arcs $\delta_i$.
\end{lemma}

\begin{proof}
First, we show that $\widetilde{G}$ is isotopic to $\{x_1\} \times S^{n-2}$ in $S^2 \times S^{n-2}$. To that end, consider $\beta_{2\ell} \times S^{n-2}$. Observe that its boundary $(\partial\beta_{2\ell}) \times S^{n-2}$ is given by the disjoint union of $\{x_{2\ell}\}\times S^{n-2}$ and $\{x_{2\ell+1}\}\times S^{n-2}$, and that the complement of an open tubular neighbourhood of $\beta_{2\ell}$ in~$\beta_{2\ell}\times S^{n-2}$ is diffeomorphic to an $(n-1)$-disc~$D$ after smoothing corners. We can isotope $\widetilde{G}$ over~$D$ to the sphere obtained by tubing the $(2\ell-1)$ spheres $\{x_{1}\}\times S^{n-2}, \dots, \{x_{2\ell-1}\}\times S^{n-2}$ along the arcs $\beta_1,\ldots,\beta_{2\ell-2}$. See \Cref{Fig:Gtilde_isotopy} for a schematic. By induction we get an isotopy from~$\widetilde{G}$ to $\{x_1\}\times S^{n-2}$. 

Now, note that the embeddings $\gamma_i$, $i \in \{1, \dots, 2\ell\}$, extend to disjoint embeddings $(D^2,S^1)\hookrightarrow (S^n,S^n\setminus K)$ by general position when $n\geq 5$ and by using a small $2$-sphere linking the boundary $S^1$ to tube away self intersections when $n=4$; for details see~\cite{Normantrick}. For each $i \in \{1, \dots, 2\ell\}$, we can use these discs bounded by the~$\gamma_i$ together with the arc $\delta_i$ to define an isotopy rel boundary in the complement of the spheres $\{x_{1}\}\times S^{n-2}, \dots, \{x_{2\ell+1}\}\times S^{n-2}$ in~$S^2 \times S^{n-2}$ between the arcs $\beta_i$ and~$\beta_i\#_{\delta_i} \gamma_i$. See the right-hand side of \Cref{F: Arcmodification}, where the region giving an isotopy between the arcs~$\beta_i\#_{\delta_i} \gamma_i$ and~$\beta_i$ in $S^2 \times S^{n-2}$ is shaded. Note that, if the loops $\gamma_i$ are not homotopically trivial in $S^n \setminus K$, the shaded region in \Cref{F: Arcmodification} will intersect~$K$ and hence the described isotopy between $\widetilde{G}$ and $G$ is not supported in the exterior of $B$ (but in $S^2 \times S^{n-2}$). Since isotopies of the arcs rel boundary induce an isotopy between the tubed surfaces, the sphere $G$ is isotopic to~$\widetilde{G}$ in~$S^2\times S^{n-2}$.
\end{proof}

\begin{figure}[htbp] 
	\centering
	\def\svgwidth{0,8\columnwidth}
	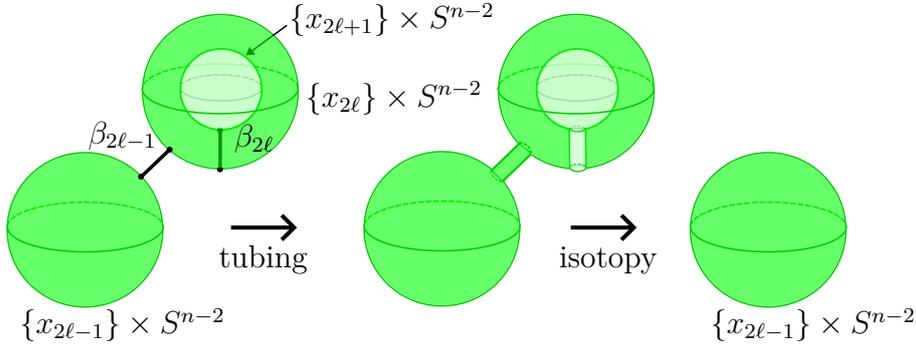
	\caption{A schematic for an isotopy between $\widetilde{G}$ and $\{x_{2\ell-1}\}\times S^{n-2}$. }\label{Fig:Gtilde_isotopy}
\end{figure}

Note that, in general, $G$ is not isotopic to $\widetilde{G}$ in the complement of $B$. See \cref{rem:sim_geom_dual} for further discussion.

We have so far defined $B$ and $G$ as embedded submanifolds of $S^2 \times S^{n-2}$, which is diffeomorphic to $\partial (D^{n+1} \cup h_{n-2}(R))$. To define the corresponding RBG manifold, we must choose a parametrisation and framing for each of $B$ and $G$, so that we can attach $(n-1)$-handles to $D^{n+1} \cup h_{n-2}(R)$ along them. Both $B$ and $G$ are isotopic to $\{\text{pt}\}\times S^{n-2}$ and hence have closed tubular neighbourhoods in $S^2 \times S^{n-2}$ that are diffeomorphic to $S^{n-2} \times D^2$. We can use any such diffeomorphism for each of~$B$ and~$G$ to fix a parametrisation and framing. From now on, by slight abuse of notation, $B$ and $G$ will denote the corresponding framed knots in~$\partial D^{n+1} \cup h_{n-2}(R)$.

\subsubsection{The construction gives an RBG manifold}

\begin{proposition}\label{prop:RBG-mfd_from_constr}
The framed knots $B$ and $G$ together with the framed $R$ from Constructions~\ref{General R construction}--\ref{General G construction} always produce an RBG manifold $W$ as in \eqref{eq:RBG_k=n-2}.
\end{proposition}

\begin{proof}
We need to show that, for any $B$ and~$G$ resulting from the above constructions, their images in $\partial (D^{n+1} \cup h_{n-2}(R))$ are geometrically dual to the belt sphere $\mathcal{S}_R$ up to isotopy. For any local knot~$K$, the sphere $B(S^{n-2})$ is by construction geometrically dual to~$\mathcal{S}_R$. Moreover, by \cref{lem:GequalsGtilda}, the embedded submanifold $G(S^{n-2})$ is isotopic to~$\{x_1\} \times S^{n-2}$, which is geometrically dual to $\mathcal{S}_R$. 
\end{proof}

\begin{remark}\label{rem:sim_geom_dual}
    In certain situations, $B(S^{n-2})$ and $G(S^{n-2})$ will not be simultaneously geometrically dual to $\mathcal{S}_R$. For example, this can happen if the embeddings $\gamma_1, \dots, \gamma_{2\ell}$ in the above construction of $G$ are chosen to represent distinct elements of $\pi_1(S^n\setminus K)$. We will make use of a particular instance of this situation in the next section. 

    Note that for the specific RBG manifolds constructed in \Cref{subsec:RBG_constr_explicit}, $G$ will not be isotopic to $\widetilde{G}$ in the complement of $B$. Thus $G$, when considered as a knot in the complement of $B$, does indeed depend on the choices of $\gamma_i$ and~$\delta_i$. So the isotopy class of $G$ in $S^2\times S^{n-2}$ is independent of the choices made, while the isotopy class of the $2$-component link $B\cup G$ in $S^2\times S^{n-2}$ depends on these choices.
\end{remark}

\subsection{An explicit infinite family of RBG manifolds}\label{subsec:RBG_constr_explicit} 

In this section, 
using the construction from \Cref{subsec:nDRBG}, for every~$n \geq 4$ and $m \geq 1$, we will build a pair of knots $B_m$ and $G_m$ in $\partial (D^{n+1} \cup h_{n-2}(R))$ that give rise to an RBG manifold $W$ as in \eqref{eq:RBG_k=n-2}. By \Cref{thm:rbg_knots}, we will obtain knots~$K_{B_{m}}$ and $K_{G_{m}}$ in~$S^n$ with o.p.~diffeomorphic traces.  In \cref{sec:distinguishing}, we will show that $K_{B_m}$ and $K_{G_m}$ are indeed non-isotopic. 

\subsubsection{Construction of \texorpdfstring{$R$}{R}}

Fix $n \geq 4$ and $m \geq 1$. As in Construction~\ref{General R construction}, let us denote by $R$ an embedding $S^{n-3} \hookrightarrow S^n$ such that $D^{n+1}\cup h_{n-2}(R) \cong D^3 \times S^{n-2}$, and by~$\mathcal{S}_R = S^2 \times \{z\} \subseteq S^2 \times S^{n-2}$ the belt sphere of~$h_{n-2}(R)$ for some~$z \in S^{n-2}$.

\subsubsection{Construction of \texorpdfstring{$B_m$}{B_m}}\label{K_m_construction} 
We set~$\ell = 1$, so $2\ell+1 =3$. In the notation from \Cref{subsec:nDRBG}, we will construct~$B_m$ and~$G_m$ from the $(n-2)$-spheres $\{w\} \times S^{n-2}$ and $ \{x_i\} \times S^{n-2} \subseteq S^2 \times S^{n-2}$, $i \in \{1, 2, 3\}$, for four distinct points $w, x_1, x_2, x_3\in S^2$.

The construction of $B_m$ is based on the choice of a local knot~$K_m$ in~$S^n$. We will build $K_m$ by doubling an $(n-2)$-dimensional embedded ribbon disc $D^\mathcal{B}_m$ in $D^n$ (see \Cref{sec:ribbon}). First, we will explain the construction of the immersed ribbon disc~$D^\mathcal{B}_m$. The reader is invited to consult the left-hand side of \Cref{fig:Akbulut_trick} throughout, which illustrates the situation for~$n=4$. For now, ignore the dashed green part of this figure and the top red circle $\mathcal{R}$ and just view the shaded blue part as a subset of~$S^{n-1}=S^3$. We will explain the whole figure in \Cref{subsec:Kirby_diagram}.

\begin{figure}[htbp] 
	\centering
	\centering
    \includegraphics[width=1\textwidth]{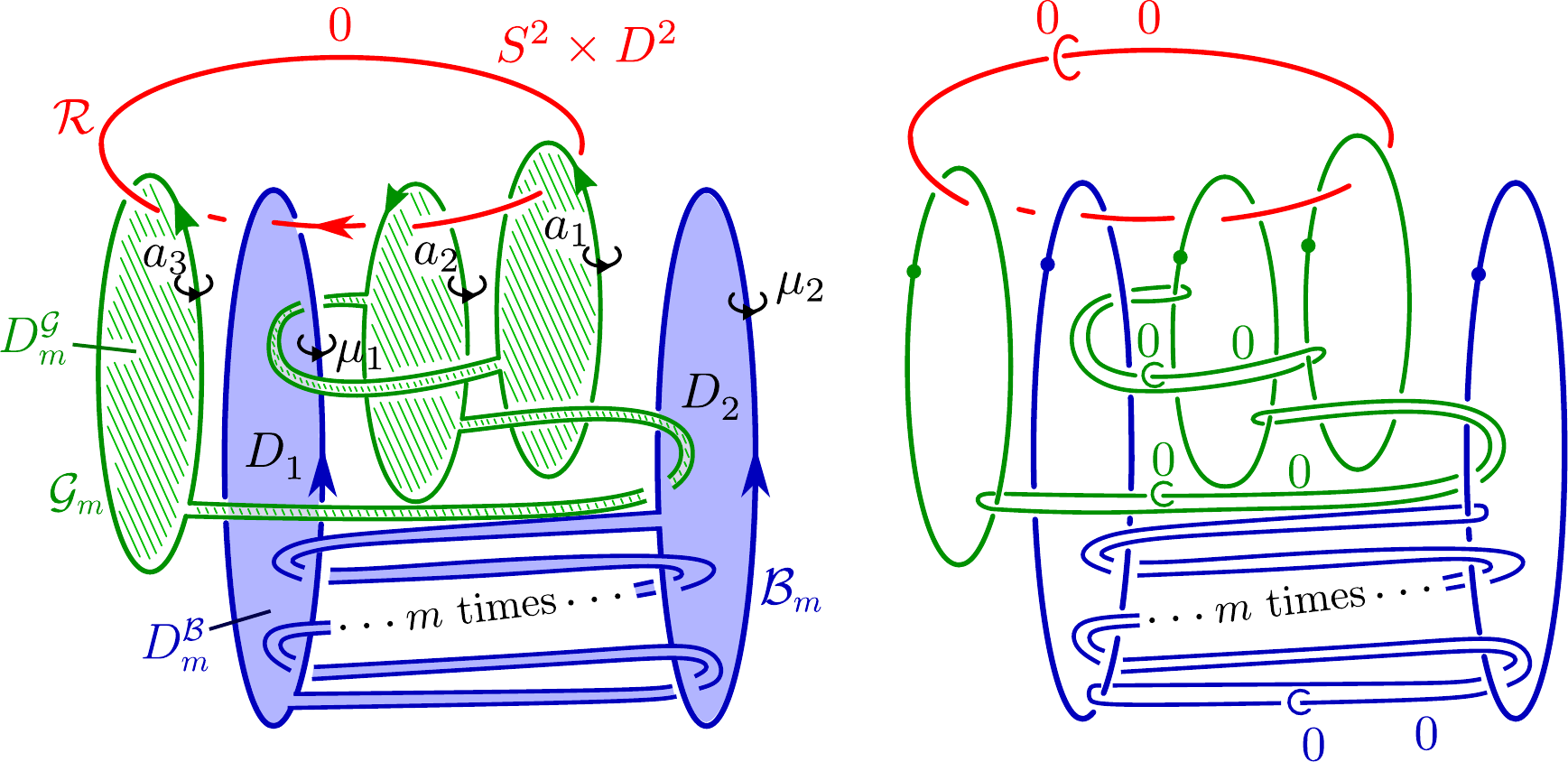}
	\caption{On the left: The two discs $D^\mathcal{G}_m$ bounded by $\mathcal{G}_m$ (dashed in green) and~$D^\mathcal{B}_m$ bounded by $\mathcal{B}_m$ (shaded in blue) in $S^2\times D^2$ for~$n=4$. Here, we simplify the notation from the text to $D_i \defeq D^{n-2}_i$ for $i \in \{1,2\}$. On the right: Kirby diagram for $S^2\times S^2\setminus (B_m\cup G_m)$.}
    \label{fig:Akbulut_trick}
\end{figure}

Choose disjoint, oriented, smoothly embedded $(n-2)$-balls $D^{n-2}_{i}$ in $S^{n-1}$ for $i\in \{1,2\}$, and denote the (oriented) meridians of $\partial D^{n-2}_{i} \cong S^{n-3}$ by~$\mu_{i}$. To construct~$D^\mathcal{B}_m$, choose an arc~$\alpha_\mathcal{B}$ whose interior lies in the complement of $\partial D^{n-2}_1 \cup \partial D^{n-2}_2$ as follows. If~$m=1$, $\alpha_\mathcal{B}$ starts at~$\partial D^{n-2}_{1}$, wraps once around~$\partial D^{n-2}_{2}$ following $\mu_{2}$, then wraps once around~$\partial D^{n-2}_{1}$ following~$\mu_{1}$, and finally ends at $\partial D^{n-2}_{2}$. For general $m$, the arc~$\alpha_\mathcal{B}$ also starts at~$\partial D^{n-2}_{1}$ and ends at~$\partial D^{n-2}_{2}$, and in between it wraps in total~$m$ times alternatingly first around~$\partial D^{n-2}_{2}$ and then around $\partial D^{n-2}_{1}$. For example, for~$m=3$, the arc follows $\mu_{2},\mu_{1},\mu_{2},\mu_{1},\mu_{2},\mu_{1}$ in this order. Recall that in dimension $n \ge 4$ homotopy rel boundary implies isotopy rel boundary for embedded arcs, hence the arc $\alpha_\mathcal{B}$ as described above is uniquely determined up to isotopy rel boundary, and it is chosen such that it transversely intersects $D^{n-2}_{1}$ and~$D^{n-2}_{2}$ exactly~$m$ times each.

Now, tube $\partial D^{n-2}_{1}$ and $\partial D^{n-2}_{2}$ along $\alpha_\mathcal{B}$ to obtain a sphere $\mathcal{B}_m$ (see \Cref{subsubsec:tubing}). Here, for~$n >4$, in the tubing procedure there is a unique choice of thickening when $\alpha_\mathcal{B}$ is thickened to~$h_{\alpha_\mathcal{B}} \colon D^1 \times D^{n-3} \hookrightarrow S^{n-1}$ and~$\mathcal{B}_m$ is well-defined up to ambient isotopy. For~$n=4$, \ie when $\alpha_\mathcal{B}$ and $h_{\alpha_\mathcal{B}} \colon D^1 \times D^1 \hookrightarrow S^3$ lie in $S^3$, we thicken by using the blackboard framing, as shown in \Cref{fig:Akbulut_trick} (left) in blue. The two balls $D^{n-2}_{1}$ and $D^{n-2}_{2}$, together with $h_{\alpha_\mathcal{B}}(D^1 \times D^{n-3})$ attached to them along $h_{\alpha_\mathcal{B}}(S^0 \times D^{n-3})$, give rise to the desired immersed disc $D^\mathcal{B}_m$ in~$S^{n-1}$, with~$2m$ ribbon singularities and boundary~$\mathcal{B}_m$. The ribbon intersections disappear when the interior of $D^\mathcal{B}_m$ is appropriately pushed into~$D^n$. We obtain an embedded disc in~$D^n$ with boundary~$\mathcal{B}_m$ in $S^{n-1}$, which, by slight abuse of notation, we will also denote by~$D^\mathcal{B}_m$.

By doubling the pair $(D^n, D^\mathcal{B}_m)$, we obtain the ribbon $(n-2)$-knot $K_m$ in $S^n$ (see \Cref{sec:ribbon}).\footnote{For $n=4$ and $m=1$, we observe that $\mathcal{B}_m = \mathcal{B}_1 = \partial D^\mathcal{B}_1$ is the knot~$3_1 \# -3_1$, and the double~$K_1$ of $D^\mathcal{B}_1$ in $S^4$ is the spun trefoil, see \eg \cite{suciu:inf-many-ribbon-knots}*{Figure~2}.} Finally, by Construction \ref{General B construction} we define 
\begin{align*}
    (S^2\times S^{n-2},B_m) \defeq (S^2\times S^{n-2}, \{w\}\times S^{n-2}) \# (S^n, K_m).
\end{align*}

\subsubsection{Construction of \texorpdfstring{$G_m$}{G_m}}\label{G_m_construction}
As in Construction \ref{General G construction}, we start with~$\widetilde{G}_m$, which is obtained by tubing the spheres $\{x_i\} \times S^{n-2}$, $i \in \{1, 2,3\}$, along ``straight'' arcs~$\beta_1$ and~$\beta_2$. We now have to choose two disjoint embeddings $\gamma_i\colon S^1 \hookrightarrow S^n\setminus K_m$, $i \in \{1, 2\}$. By slightly abusing notation, we denote by $\mu_1$ and $\mu_2$ the images of the meridians $\mu_1$ and $\mu_2$ of $\partial D^{n-2}_{1}$ and $\partial D^{n-2}_{2}$ in $S^n\setminus K_m$. Then we choose $\gamma_1$ and $\gamma_2$ such that $\gamma_1(S^1)$ and $\gamma_2(S^1)$ represent $\mu_1^{-1}$ and $\mu_2^{-1}$ in $\pi_1(S^n \setminus K_m)$ up to conjugacy, respectively (see also \Cref{subsec:meridians}). We define~$G_m$ to be the sphere obtained by tubing the spheres~$\bigcup_{i=1}^3\{x_i\}\times S^{n-2}$ along the arcs~$\beta'_i=\beta_i\#_{\delta_i} \gamma_i$, $i \in \{1, 2\}$.

This finishes the construction of~$B_m$ and~$G_m$. 

\subsubsection{The construction gives an RBG manifold}
By \Cref{prop:RBG-mfd_from_constr}, we obtain an RBG manifold $W_m = D^{n+1}\cup h_{n-2}(R) \cup h_{n-1}(B_m) \cup h_{n-1}(G_m)$. Therefore, by \Cref{thm:rbg_knots}, we can summarise the work of this subsection as follows.

\begin{proposition}\label{prop:summ_constr}
    Let $n \geq 4$. For every $m \geq 1$, the pair of framed knots $K_{B_{m}}$ and~$K_{G_m}$ in $S^n$ have o.p.~diffeomorphic traces.
\end{proposition}

We will use these pairs in \Cref{sec:distinguishing} to prove \Cref{theorem:non-isotopic} as a direct application of \Cref{theorem:non-isotopic_compl}.

\subsection{Kirby diagrams in dimension 4}\label{subsec:Kirby_diagram}

In dimension $n=4$, we can describe~$B_m$ and~$G_m$ explicitly in Kirby diagrams as follows. Consider the Kirby diagram of~$S^2\times D^2$ given by the $0$-framed red unknot $\mathcal{R}$ shown in \Cref{fig:Akbulut_trick} on the left. The knots~$\mathcal{B}_m$ and~$\mathcal{G}_m$ in $\partial(S^2\times D^2)$ bound ribbon discs obtained by taking band sums of the two blue shaded and three green dashed discs in \Cref{fig:Akbulut_trick}, respectively. Here, the disc bounded by $\mathcal{B}_m$ is the ribbon disc constructed in Construction \ref{K_m_construction}. By doubling, the three dashed green discs in \Cref{fig:Akbulut_trick} will produce the three green spheres $\{x_i\} \times S^{2} \subseteq S^2 \times S^{2}$, $i \in \{1, 2, 3\}$, from Construction~\ref{G_m_construction}. We push the interiors of the blue and green discs bounded by $\mathcal{B}_m$ and $\mathcal{G}_m$ into the interior of~$S^2\times D^2$, which gives disjointly and properly embedded $2$-discs in $S^2\times D^2$. By slight abuse of notation, we denote these by~$D^\mathcal{B}_m$ and $D^\mathcal{G}_m$. Forming a relative double of~$(S^2\times D^2, D^\mathcal{B}_m \cup D^\mathcal{G}_m)$, we obtain two disjointly embedded $2$-spheres in~$S^2\times S^2$, which are the $2$-spheres~$B_m$ and~$G_m$ constructed above. Roughly speaking, the doubles of the green and blue shaded discs in Figure \ref{fig:Akbulut_trick} correspond to the initial $2$-spheres in Constructions \ref{K_m_construction} and \ref{G_m_construction}, while the tubes connecting them are given by the doubles of the ribbon bands connecting such discs. In this perspective, the $2$-sphere $S^2 \times \{z\}$ in~$S^2\times S^2$, which is given by gluing the core of the red $2$-handle in \Cref{fig:Akbulut_trick} to a pushed-in Seifert disc of the red unknot $\mathcal{R}$, corresponds to the belt sphere of the $5$-dimensional $2$-handle $h_2(R)$ attached to $D^5$ according to \Cref{def:RBG_mfd} of an RBG manifold.

On the right of \Cref{fig:Akbulut_trick} we present a Kirby diagram of $(S^2\times S^2)\setminus (B_m\cup G_m)$ which is obtained as follows. First, we double $S^2\times D^2$, which corresponds to adding a $0$-framed meridian to the red knot $\mathcal{R}$, see for example~\cite{gompf-stipsicz:book}*{Example~4.6.3}. By construction, the dashed green ribbon disc $D^\mathcal{G}_m$ admits a handle decomposition with three $0$-handles and two $1$-handles. Thus, its double, \ie the green sphere $G_m$, admits a handle decomposition with three $0$-handles, four $1$-handles (two of which are duals of the $1$-handles of~$D^\mathcal{G}_m$), and three $2$-handles. Analogously, we get a handle decomposition for the blue sphere~$B_m$ consisting of two $0$-handles, two $1$-handles, and two $2$-handles. Now, there is a standard argument for constructing a Kirby diagram of the complement of an embedded surface, see~\cite{gompf-stipsicz:book}*{Section~6.2}. For every $2$-dimensional $0$-handle of~$B_m\cup G_m$ we obtain a $4$-dimensional $1$-handle of its complement. Likewise, for each $2$-dimensional $1$-handle we obtain a $4$-dimensional $2$-handle. The complement of the ribbon discs is constructed exactly as described in~\cite{gompf-stipsicz:book}*{p.~212}. In particular, the attaching regions of the $4$-dimensional $2$-handles that correspond to dual $1$-handles in the handle decomposition of the ribbon spheres~$B_m$ and $G_m$ are given by $0$-framed meridians of the attaching spheres of the $2$-handles of the complement of the ribbon discs (see \cite{gompf-stipsicz:book}*{Example 4.6.3}). Note that, as usual, by~\cite{laudenbachpoenaru} the $3$-handles of~$(S^2\times S^2)\setminus (B_m\cup G_m)$ corresponding to the $2$-handles of~$B_m\cup G_m$ can be omitted in our Kirby diagram (see also~\cite{gompf-stipsicz:book}*{Section~4.4}).

We will use the discussion above and in particular the Kirby diagram from  \Cref{fig:Akbulut_trick} in \Cref{prop:fund_groups_expl} to establish fundamental group presentations for $S^2 \times S^2 \setminus (B_m \cup G_m)$.

\begin{remark}\label{rem:doubling}
    Note that we could also construct $B_m$ and $G_m$ for general $n \geq 4$ by a relative double construction as described above for $n=4$. For later use, we point out that this can be done by crossing the construction described in dimension $4$ with~$D^{n-4}$. More precisely, to get the spheres $B_m$ and $G_m$ in $S^2\times S^{n-2}$, we can take the relative double of $((S^2\times D^2)\times D^{n-4}, (D^\mathcal{B}_m \cup D^\mathcal{G}_m) \times D^{n-4})$.   This description makes it evident that the latter pair deformation retracts onto~$(S^2\times D^2, D^\mathcal{B}_m\cup D^\mathcal{G}_m)$, and this retraction preserves the meridians of the boundaries of the ribbon discs. Note that one could depict the handle decomposition for $S^2\times S^{n-2} \setminus (B_m \cup G_m) $ for every~$n \geq 4$ in a diagram similar to Figure 8 in \cite{kim}.
\end{remark}

Let us end this section with the following remarks on our construction.
In the original construction for~$n=3$, an RBG link consists of a $3$-component link $R \cup B \cup G$ in~$S^3$. This translates to an RBG $4$-manifold~$W$ with~$k=1$ by viewing the link as a Kirby diagram in which the unknotted component $R$ is a dotted $1$-handle and the components $B$ and $G$ represent $0$-framed $2$-handles. There is an analogous construction in the higher-dimensional setting given by taking a~$2$-component link $B\cup G$ of $(n-2)$-spheres in $S^n$ together with a circle $\mathcal{R}$ linking both components once. When $n\geq 4$, the circle always bounds a disc. We require that both~$B$ and~$G$ are individually isotopic in the complement of $\mathcal{R}$ to spheres that intersect the disc bounded by $\mathcal{R}$ geometrically once. Then $\mathcal{R}\cup B\cup G$ can be understood as a \emph{generalised RBG link} in $S^n$. The corresponding RBG manifold is obtained by pushing the disc bounded by~$\mathcal{R}$ into~$D^{n+1}$, removing a neighbourhood of it and then attaching $(n-1)$-handles along $B$ and $G$ as before.

\section{Distinguishing knots with the same trace}
\label{sec:distinguishing}

In this section, we first provide a brief overview of relevant knot invariants. Then, in light of the examples constructed in \Cref{sec:construction}, we define a new knot invariant in terms of the conjugacy classes of the meridians of knots. We will use this new invariant to finish the proof of \Cref{theorem:non-isotopic} in \Cref{subsec:proof_thm_non-iso}.

\subsection{Knot invariants}\label{subsec:knot-invariants}

Let us first consider the exterior $E_K =S^n \setminus \overset{\circ}{\nu} K$ of a knot~$K$ in~$S^n$ for $n \geq 3$. By Alexander duality, the exterior $E_K$ has the same homology as~$S^1$. In particular, the homology groups of $E_K$ are independent of the knot type~$K$.
However, a powerful invariant is given by the fundamental group of $E_K$, from which one can extract manageable homological invariants, such as the \emph{Alexander invariants}.

In \Cref{subsec:Fox_calculus}, we briefly recall the definitions of the most important of these invariants: the Alexander module and the Alexander polynomial. We will distinguish our infinitely many pairs of knots $(K_{B_m}, K_{G_m})_{m \geq 1}$ from each other using the Alexander polynomial (see \Cref{prop:Alex_polys}).

The following proposition shows that, for $n \geq 4$, knots $K_1, K_2$ in $M^n$ with homotopy equivalent surgeries~$M(K_1) \simeq M(K_2)$, have isomorphic knot groups~$\pi_1(K_1) \cong \pi_1(K_2)$. In particular, if two knots in~$S^n$ have o.p.~diffeomorphic surgeries, then they cannot be distinguished by their Alexander polynomials or modules. 

\begin{proposition}
\label{proposition:fundamental_group}
    For $n \geq 4$, let $K$ be a knot with trivial normal bundle in an $n$-manifold $M$, possibly with $\partial M \neq \emptyset$.
    Then the inclusion $E_{K}\hookrightarrow M(K)$ induces an isomorphism on the fundamental groups, which identifies the conjugacy class of the meridian $\mu_K$ with that of the dual curve of $K$ in~$M(K)$ for some orientation of the dual curve. In particular, if two such knots have homotopy equivalent surgeries, then their knot groups are isomorphic.
\end{proposition} 

\begin{proof}
This follows from the Seifert--van Kampen theorem applied to the decomposition $M(K) = E_K \cup_{S^{n-2}\times S^1} D^{n-1} \times S^1$, noting that the inclusion $S^{n-2}\times S^1 \hookrightarrow D^{n-1}\times S^1$ induces an isomorphism on the level of $\pi_1$ when~$n\geq 4$. Via this inclusion, the dual curve $\gamma_K = \{0\} \times S^{1}$ of $K$ is identified (for some orientation) with the element of $\pi_1(K)$ represented by $\mu_K$, which is well-defined up to conjugacy (see also \Cref{subsec:meridians}).
\end{proof}

\begin{remark}
    Although the fundamental group $\pi_1 \defeq \pi_1(K_i)\cong  \pi_1(S^n\setminus K_i)$, $i \in \{1,2\}$, of the complement cannot distinguish between knots $K_1$ and $K_2$ with a shared surgery, the second homotopy groups~$\pi_2$  of their complements could differ. We could compare these by considering each $\pi_2$ as a (left) $\mathbb{Z}[\pi_1]$-module. This method has been successfully employed to distinguish knots with isomorphic knot groups, \eg in~\cites{gordon,suciu:inf-many-ribbon-knots}. However, sometimes this is not enough. Plotnick--Suciu~\cite{plotnick_suciu} gave examples of $2$-knot complements with identical first and second homotopy groups (the latter as $\Z[\pi_1]$-module), whose homotopy types they distinguished using $k$-invariants. Many questions concerning the homotopy type of the complement of $2$-knots remain open; see, for example, \cite{lomonaco}*{Section XII}. Note that, in contrast, classical $1$-knots have aspherical complements~\cite{Papakyriakopoulos:dehnslemma_asphericity}. 

    For an overview on other knot invariants that could potentially distinguish knots with the same surgery, we refer to the excellent survey article by Conway~\cite{conway2022invariants2knots}. Many of these invariants are typically only computed for fibred knots, such as twist-spun knots, where the knot exterior has a simple structure. In our setting, our knot exteriors do not have a simple geometric structure, so these approaches are difficult to apply. Instead, we use a more elementary invariant: the conjugacy class of the meridians in the fundamental group. We will explain this next.
\end{remark}

\subsection{Counting representations}\label{subsec:counting_reps}

For each $m \geq 1$, the knots~$K_{B_m}$ and $K_{G_m}$ constructed in \Cref{prop:summ_constr} have o.p.~diffeomorphic traces and thus surgeries. \Cref{proposition:fundamental_group} tells us that their knot groups are isomorphic, so it is not possible to distinguish between~$K_{B_m}$ and $K_{G_m}$ using abelian knot invariants derived from the knot groups for a fixed $m \geq 1$. In \Cref{subsec:explicit_reps}, we will compute a presentation of $\pi_1(K_{B_m})\cong \pi_1(K_{G_{m}})$ in which two distinct generators represent the conjugacy classes of the meridians of~$K_{B_m}$ and~$K_{G_m}$, respectively. This observation will be crucial in distinguishing these knots. Indeed, we will look at representations of the knot groups into a specific finite group~$A$ which map the elements representing the conjugacy classes of the meridians of $K_{B_m}$ and~$K_{G_m}$ to specific elements of~$A$. The number of such representations will differ for~$K_{B_m}$ and~$K_{G_m}$. 

Let us now explain this in detail. First, we introduce some notation. 

\begin{definition}\label{def:inv_reps}
For finitely presented groups $H$ and $A$, and elements $h \in H$, $\sigma \in A$, we denote by $\operatorname{Rep}(H,A,h,\sigma)\defeq \{ \varphi \in \hm(H,A) \mid \varphi(h) = \sigma \}$ the set of homomorphisms from $H$ to $A$ that map $h$ to $\sigma$. Moreover, $\# \operatorname{Rep}(H,A,h,\sigma)$ will denote the number of such homomorphisms, which is possibly infinite. 
\end{definition}
Note that $\# \operatorname{Rep}(H,A,h,\sigma)$ is invariant under conjugation of $h \in H$ and $\sigma \in A$, \ie
\begin{align*}
	\# \operatorname{Rep}(H,A,h,\sigma) = \# \operatorname{Rep}(H,A,ghg^{-1},\tau\sigma\tau^{-1}) \qquad \text{for any } g \in H, \tau \in A. 
\end{align*}
In fact, a bijection is obtained by associating to $\varphi\in \operatorname{Rep}(H,A,h,\sigma)$ the homomorphism $c_\tau\varphi c^{-1}_g$, where $c_g$ and $c_\tau$ denote the conjugations by $g$ and $\tau$ respectively.
\begin{definition}\label{def:rep_knot_inv}
    Let $K$ be a knot in $S^n$ with meridian $\mu_K$. 
    For a finitely presented group $A$ and an element~$\sigma \in A$, we define
    \begin{align*}
        \mathcal{N}_K(A,\sigma) \defeq \# \operatorname{Rep}(\pi_1(K),A,\mu_K,\sigma).
    \end{align*}
\end{definition}

Recall the following definition.

\begin{definition}\label{def:homo_equiv_rel_bound}
Let $X, Y$ be $n$-manifolds with possibly non-empty boundary. We say that $X$ and $Y$ are \emph{homotopy equivalent rel boundary} if the pairs $(X,\partial X)$ and~$(Y,\partial Y)$ are homotopy equivalent, \ie if there is a homotopy equivalence $f\colon X\to Y$ restricting to a homotopy equivalence $f\vert_{\partial X}\colon \partial X\to \partial Y$.
\end{definition}

The following proposition shows that $\mathcal{N}_K(A,\sigma)$ defines a knot invariant. 

\begin{proposition}\label{prop:counting_invariant}
    For $n \geq 4$, suppose that two knots $K_1$ and $K_2$ in $S^n$ have homotopy equivalent exteriors rel boundary (which happens \eg if $K_1$ and $K_2$ are ambiently isotopic). Let $A$ be a finitely presented group and let $\sigma \in A$ be an element which is conjugate to its inverse~$\sigma^{-1}$ in~$A$. Then $\mathcal{N}_{K_1}(A,\sigma) = \mathcal{N}_{K_2}(A,\sigma)$.
\end{proposition}

\begin{proof}
    Suppose that $(E_{K_1}, \partial \nu K_1)$ and~$(E_{K_2}, \partial \nu K_2)$ are homotopy equivalent via a map $h \colon E_{K_1} \to E_{K_2}$. In particular, $h$ induces a homotopy equivalence between the boundaries $\partial \nu K_1 \simeq \partial \nu K_2$. Since the fundamental groups of $\partial \nu K_1 \simeq \partial \nu K_2$ are infinite cyclic generated by the meridians, the induced isomorphism $h_\ast \colon \pi_1 (K_{1}) \to \pi_1 (K_{2})$ maps the conjugacy class of $\mu_{K_1}$ either to the conjugacy class of $\mu_{K_2}$ or to the conjugacy class of $\mu_{K_{2}}^{-1}$. In the first case, $h_\ast$ induces a bijection between
    \begin{align*}
        \operatorname{Rep}(\pi_1(K_1),A,\mu_{K_1},\sigma) \qquad \text{and} \qquad \operatorname{Rep}(\pi_1(K_2),A,\mu_{K_2},\sigma).
    \end{align*}
    In the second case, there is a bijection between $\operatorname{Rep}(\pi_1(K_1),A,\mu_{K_1},\sigma)$ and $$\operatorname{Rep}(\pi_1(K_2),A,\mu_{K_2}^{-1},\sigma)=\operatorname{Rep}(\pi_1(K_2),A,\mu_{K_2},\sigma^{-1}).$$ The latter set is in bijection with $\operatorname{Rep}(\pi_1(K_2),A,\mu_{K_2},\sigma)$, because~$\sigma $ and its inverse are conjugate in $A$. In both cases, we obtain $\mathcal{N}_{K_1}(A,\sigma) = \mathcal{N}_{K_2}(A,\sigma)$. 
\end{proof}

\begin{remark}\label{rem:counting_invariant_dim3}
    The statement of Proposition \ref{prop:counting_invariant} also holds for knots $K_1, K_2$ in~$S^3$, although the proof is slightly different. According to \cites{waldhausen1968irreducible}, if $(E_{K_1}, \partial \nu K_1)$ and~$(E_{K_2}, \partial \nu K_2)$ are homotopy equivalent, then $K_1$ and $K_2$ have homeomorphic exteriors. By \cite{mca1989knots}, knots in $S^3$ with  homeomorphic exteriors are ambiently isotopic as unoriented knots up to taking mirror images. We thus find a homeomorphism of pairs $h\colon (S^3,K_1)\to(S^3,K_2)$ which induces an isomorphism $h_*\colon \pi_1(K_1)\to\pi_1(K_2)$ mapping the conjugacy class of~$\mu_{K_1}$ either to $\mu_{K_2}$ or $\mu_{K_2}^{-1}$. We conclude that $\mathcal{N}_{K_1}(A,\sigma) = \mathcal{N}_{K_2}(A,\sigma)$ as above.
\end{remark}

\begin{remark}\label{rem:Suciu_meridians}
    Suciu used the idea of distinguishing ribbon discs with the same exterior by their meridians in~\cite{suciu:inf-many-ribbon-knots}, where he constructed infinitely many distinct (embedded) ribbon $n$-discs with the same exterior for $n\geq 3$. 
\end{remark}

Before we turn to an explicit calculation of the above defined invariant $\mathcal{N}_{K}(A,\sigma)$ for our knots~$K\in \{K_{B_m},K_{G_m}\}$, let us prove the following statement on knots~$K$ coming from fairly general RBG manifolds. Recall that the meridian $\mu_K$ is well-defined up to conjugacy in $\pi_1(K)$ and that we also denote this conjugacy class by~$\mu_K$. 

\begin{proposition}\label{prop:fund_groups}
    Let $W$ be an $(n+1)$-dimensional RBG manifold and suppose that $D^{n+1} \cup h_{n-2}(R) \cong D^3 \times S^{n-2}$ as in Construction~\ref{General R construction}, so 
    $B$ and~$G$ are  $(n-2)$-spheres in $\partial(D^{n+1} \cup h_{n-2}(R)) \cong S^2 \times S^{n-2}$. Then, for the associated knots $K_B$ and~$K_G$ from \Cref{thm:rbg_knots}, there are isomorphisms of pairs
    \begin{align}
        \begin{split}
            (\pi_1(K_{B}),  
            \ \mu_{K_{B}}) &\cong (\pi_1(S^2 \times S^{n-2} \setminus  (\overset{\circ}{\nu} B \cup \overset{\circ}{\nu} G)), \ \mu_{B}), \\ 
            (\pi_1(K_G), 
            \ \mu_{K_{G}}) &\cong (\pi_1(S^2 \times S^{n-2} \setminus  (\overset{\circ}{\nu} B \cup \overset{\circ}{\nu} G)), \ \mu_{G}).
            \label{eq:fund_gps_KB_KG}
        \end{split}
    \end{align}
\end{proposition} 

\begin{proof}
    By definition, $K_{B} = f_{B}\circ B$ and~$K_{G} = f_{G} \circ G$ for diffeomorphisms $f_G$ and~$f_B$ as in \eqref{eq:f_G_and_f_B}. The surgery on $S^2\times S^{n-2}\setminus \overset{\circ}{\nu} B$ along $G$ is diffeomorphic, via the restriction of $f_B$, to $S^n\setminus \overset{\circ}{\nu} K_B$. Moreover this diffeomorphism maps $\mu_B$ to $\mu_{K_B}$. Hence 
    \begin{align*}
        (\pi_1(S^n \setminus \overset{\circ}{\nu} K_{B}), \ \mu_{K_{B}}) &\cong (\pi_1((S^2 \times S^{n-2} \setminus \overset{\circ}{\nu} B) (G)), \ \mu_{B})\\&\cong (\pi_1(S^2 \times S^{n-2} \setminus (\overset{\circ}{\nu} G\cup \overset{\circ}{\nu} B)), \ \mu_B),
    \end{align*}
    where the second isomorphism follows from \Cref{proposition:fundamental_group}. In the same way, one proves the result for $K_G$.
\end{proof}

\subsection{Explicit calculations for \texorpdfstring{$\boldsymbol{K_{B_m}}$}{KBm} and \texorpdfstring{$\boldsymbol{K_{G_m}}$}{KGm}}\label{subsec:explicit_reps} 
We first determine a presentation for the knot groups $ \pi_1(K_{B_m}) \cong \pi_1(K_{G_m})$.

\begin{proposition}\label{prop:fund_groups_expl}
    Let $F_m \defeq \pi_1(S^2 \times S^{n-2} \setminus (\overset{\circ}{\nu} B_m\cup ~\overset{\circ}{\nu} G_m))$. Then, for every~$m \geq 1$, we have $F_m \cong \pi_1(K_{B_m}) \cong \pi_1(K_{G_m})$ and a presentation for $F_m$ is given by 
    \begin{align*}
        F_m &\cong \langle x,y,a \mid (yx)^my(yx)^{-m}x^{-1}=1,\, (x^{-1}a x) a^{-1} x^{-1}( y a y^{-1} )=1\rangle,  
    \end{align*}
    where both $x$ and $y$ represent the conjugacy class of the meridian of $K_{B_m}$, while the generator $a$ represents the conjugacy class of the meridian of $K_{G_m}$.
\end{proposition}

\begin{proof}
    Throughout the proof, fix $m \geq 1$. The group isomorphisms $F_m \cong \pi_1(K_{B_m}) \cong \pi_1(K_{G_m})$ follow directly from \Cref{prop:fund_groups}. Note that under these isomorphisms, the meridians~$\mu_{K_{B_m}}$ and $\mu_{K_{G_m}}$ are identified with meridians of the spheres $B_m$ and~$G_m$. We will now first compute a presentation for $F_m$ in dimension $n=4$ 
    and then reduce the general case~$n \geq 5$ to this computation.

    For the case $n=4$, we will rely on the discussion from \Cref{subsec:Kirby_diagram}. Recall that~$S^2 \times D^2$ has a handle decomposition that consists of one $0$- and one $2$-handle. We denote this $2$-handle by $h_2$. Consider the ribbon discs $D^b_m$ (shaded in blue) and~$D^g_m$ (dashed in green) in $S^2 \times D^2$ on the left-hand side of \Cref{fig:Akbulut_trick} from \Cref{subsec:RBG_constr_explicit}. The green disc $D^g_m$ has a handle decomposition with three $0$-handles and two $1$-handles; the blue disc $D^b_m$ has a handle decomposition with two $0$-handles $D^2_1$ and $D^2_2$ and one $1$-handle. This provides us with a handle decomposition of the complement of $D^b_m \cup D^g_m$ in $S^2 \times D^2$, which consists of one $0$-handle, five $1$-handles and four $2$-handles (see \eg~\cite{gompf-stipsicz:book}*{Proposition~6.2.1}), where the fourth $2$-handle corresponds to $h_2$. This handle decomposition can be described in a Kirby diagram by removing the small $0$-framed unknots in the Kirby diagram on the right-hand side of \Cref{fig:Akbulut_trick}. From this handle decomposition, we can directly compute a presentation of $\pi_1(S^2 \times D^2 \setminus( D^b_m \cup D^g_m))$, as we do next.
    
    Denote the meridians of the three $0$-handles that make up the green disc $D^g_m$ in \Cref{fig:Akbulut_trick} (left) by $a_3, a_2, a_1$ from left to right. Recall that we denote the meridians of~$\partial D^2_1$ and $\partial D^2_2$ from the construction of~$D^b_m$ by $\mu_1,\mu_2$. The push-forwards of all these five meridians represent generators of $\pi_1(S^2 \times D^2 \setminus(D^b_m \cup D^g_m))$. We can read off the following relations from the $2$-handle attachments in $S^2 \times D^2 \setminus (D^b_m \cup D^g_m)$ corresponding to the $2$-dimensional $1$-handles of $D^b_m \cup D^g_m$, and the $2$-handle~$h_2$:
    \begin{align*}
        \mu_2 a_2^{-1} \mu_2^{-1} a_3 &= 1  \quad (D^g_m)\quad \Rightarrow \quad a_3 = \mu_2 a_2 \mu_2^{-1},\\
        a_1^{-1} \mu_1^{-1} a_2 \mu_1&= 1 \quad (D^g_m) \quad \Rightarrow \quad a_1 = \mu_1^{-1} a_2 \mu_1, \\
        a_1a_2^{-1} \mu_1^{-1}a_3  &=1 \quad (h_2) \quad \Rightarrow \quad (\mu_1^{-1} a_2 \mu_1)a_2^{-1} \mu_1^{-1}(\mu_2 a_2 \mu_2^{-1})=1,\\
        (\mu_{2}\mu_{1})^m\mu_{2}(\mu_{2}\mu_{1})^{-m}\mu_{1}^{-1}&=1 \quad (D^b_m).
    \end{align*}
    Setting $x= \mu_1$, $y = \mu_2$ and $a = a_2$ for better readability, we obtain the presentation
    \begin{align*}
    \pi_1(&S^2 \times D^2 \setminus( D^b_m \cup D^g_m)) \cong \\
        &\langle x, y, a \mid (yx)^my(yx)^{-m}x^{-1} 
        = (x^{-1}a x) a^{-1} x^{-1}( y a y^{-1} ) = 1 \rangle.
    \end{align*}
    Since the spheres~$B_m$ and $G_m$ in $S^2 \times S^2$ are obtained as relative doubles of $D^b_m \cup D^g_m$ in $S^2 \times D^2$ (see \Cref{subsec:Kirby_diagram}), we have 
    \begin{align*}
    \pi_1(S^2 \times S^2\setminus &(B_m \cup G_m) ) \cong \pi_1(S^2 \times D^2 \setminus( D^b_m \cup D^g_m)) \\
        &\cong\langle x, y, a \mid (yx)^my(yx)^{-m}x^{-1} 
        =  (x^{-1}a x) a^{-1} x^{-1}( y a y^{-1} ) = 1 \rangle,
    \end{align*}
    where the first isomorphism can be obtained by a standard Seifert--van Kampen argument. Under these isomorphisms, $x$ and $y$ represent meridians of~$\partial D^b_m$ and thus~$B_m$, while $a$ represents a meridian of $\partial D^g_m$ and thus $G_m$. By \Cref{prop:fund_groups}, this yields the claim for $n=4$. Alternatively, we could have also read off the fundamental group presentation directly from the Kirby diagram for $S^2 \times S^2 \setminus (B_m \cup G_m)$ on the right-hand side of~\Cref{fig:Akbulut_trick}. In that case, we have additional $2$-handles coming from the relative doubling process described in \Cref{subsec:Kirby_diagram}, but the relations coming from these additional $2$-handles are trivial. 
    
    For $n \geq 5$, as pointed out in \Cref{rem:doubling} we can obtain $B_m$ and $G_m$ also as doubles of $(n-2)$-dimensional (properly embedded) ribbon discs in $S^{2} \times D^{n-2}$ by crossing the construction in dimension~$4$ with~$D^{n-4}$. The deformation retraction obtained in this way induces, by restriction, a deformation retraction of the complement of these discs in $S^{2} \times D^{n-2}$ onto~$S^2 \times D^2 \setminus( D^b_m \cup D^g_m)$, preserving the meridians. By applying Seifert--van Kampen to the relative double, we get the desired result.
\end{proof}

We will now compute $\mathcal{N}_{K}(A,\sigma)$ for $K \in \{K_{B_m}, K_{G_m}\}$ using \Cref{prop:fund_groups_expl} and a specific finite group $A$. The alternating group of degree~$5$, which we denote by~$A_5$, has order $60$ and the property that all its elements are conjugate to their inverses, so it is particularly suitable for an application of \Cref{prop:counting_invariant}.

\begin{proposition}
\label{proposition:A_5_reps}
    Consider the $5$-cycle $\sigma= (1 5 4 3 2)$ in~$A_5$. Let $m \geq 1$ be an integer with $m \equiv 1 \pmod{60}$. Then
    \begin{align*}
        \mathcal{N}_{K_{B_m}}(A_5,\sigma) = 6 \qquad \text{and} \qquad \mathcal{N}_{K_{G_m}}(A_5,\sigma)=1. 
    \end{align*}
\end{proposition}

\begin{proof}
    By virtue of \Cref{prop:fund_groups_expl}, we have to count the number of homomorphisms from $F_m$ to~$A_5$ that map~$x$ or~$a$ to $\sigma=(1 5 4 3 2) \in A_5$, respectively.  For ${m}=1$, this is an easy task using the function \emph{GroupHomomorphismByImages} in GAP via Sage~\cites{GAP4,sagemath} to check if a given assignment on the generators $x,y,a$ of $F_m$ extends to a group homomorphism. We simply apply this function to all possible maps $F_m \to A_5$ by assigning elements of $A_5$ to $x,y,a$. We obtain $\mathcal{N}_{K_{B_1}}(A_5,\sigma) = 6$ and~$\mathcal{N}_{K_{G_1}}(A_5,\sigma)=1$. See \Cref{sec:appendix_sage} for our Sage source code.
    
    We claim that this determines the number of representations also for any other~$m$ congruent to $1$ modulo $60$. To see this, note that the second relation in the presentation of~$F_m$ given in \Cref{prop:fund_groups_expl} does not depend on $m$. Consider the first relation and suppose that $m= 60 m'+1$ for some $m' \in \Z$. Since the order of $A_5$ is $60$, for every $\sigma_1, \sigma_2, \sigma_3 \in A_5$, the word $(\sigma_2\sigma_1)^m\sigma_2(\sigma_2\sigma_1)^{-m}\sigma_1^{-1}=(\sigma_2\sigma_1)^{60m^\prime +1}\sigma_2(\sigma_2\sigma_1)^{-(60m^\prime +1)}\sigma_1^{-1}$ is the trivial word in $A_5$ if and only if  $(\sigma_2\sigma_1)^1\sigma_2(\sigma_2\sigma_1)^{-1}\sigma_1^{-1}$ is. Therefore the assignment $\varphi \colon x \mapsto \sigma_1, y \mapsto \sigma_2, a \mapsto \sigma_3$ defines a group homomorphism $F_m \to A_5$ if and only if it defines a group homomorphism $F_1 \to A_5$. We obtain 
    \begin{align*}
        \mathcal{N}_{K_{B_m}}(A_5,\sigma) &= \mathcal{N}_{K_{B_1}}(A_5,\sigma)
         = 6 \qquad \text{and}\\
        \mathcal{N}_{K_{G_m}}(A_5,\sigma) &= \mathcal{N}_{K_{G_1}}(A_5,\sigma) =1 \qquad \text{for all } m \equiv 1 \pmod{60}.
       \quad \,
        \qedhere 
    \end{align*}
\end{proof}

\begin{corollary}\label{cor:compl_not_homot_rel_boundary}
    Let $m \geq 1$ be an integer with $m \equiv 1 \pmod{60}$. Then the exteriors of~$K_{B_{m}}$ and $K_{G_{m}}$ in $S^n$ are not homotopy equivalent rel boundary. 
\end{corollary}

\begin{proof}
    Notice that $\sigma=(1 5 4 3 2)\in A_5$ is conjugate to its inverse. If the exteriors of~$K_{B_{m}}$ and~$K_{G_{m}}$ were homotopy equivalent rel boundary for some $m \equiv 1 \pmod{60}$, then by \Cref{prop:counting_invariant} we would have $\mathcal{N}_{K_{B_m}}(A_5,\sigma) = \mathcal{N}_{K_{G_m}}(A_5,\sigma)$. But \Cref{proposition:A_5_reps} shows 
    that these numbers differ, a contradiction. 
\end{proof}

\subsection{Fox calculus to distinguish the pairs}\label{subsec:Fox_calculus}

In this subsection, we recall the definition of the Alexander polynomial and compute it for our knots $K_{B_m}$ and $K_{G_m}$ in order to distinguish the pairs $\{(K_{B_{m}}, K_{G_{m}})\}_{m \geq 1}$ from one another.

Given a knot $K \colon S^{n-2} \hookrightarrow S^n$ and its exterior $E_K$, the kernel of the Hurewicz homomorphism $\pi\defeq\pi_1(E_K) \twoheadrightarrow H_1(E_K) \cong \Z$ determines the \emph{infinite cyclic cover}~$\widetilde{E}_K \to E_K$. The fundamental group of this cover is the commutator subgroup $\pi^\prime \defeq [\pi,\pi]$ and its deck transformation group is $\pi/\pi^\prime \cong H_1(E_K) \cong \Z$, which induces an action of the group ring $\Lambda \defeq \Z[t,t^{-1}] \cong \Z[H_1(E_K)]$ on the homology groups~$H_\ast(\widetilde{E}_K)$. The \emph{Alexander module} of $K$ is defined as the (left) $\Lambda$-module~$H_1(\widetilde{E}_K)$, and its order
is the \emph{Alexander polynomial} $\Delta_K(t) \in \Lambda$ of $K$, which is well-defined up to multiplication by units in $\Lambda$ (see \eg \cite{Rolfsen}*{Chapter~7} or \cite{friedl2024surveyfoundationsfourmanifoldtheory}*{Section~15.1}). Note that, as groups, $H_1(\widetilde{E}_K) \cong \pi^\prime/[\pi^\prime,\pi^\prime]$. A presentation matrix for $H_1(\widetilde{E}_K)$, and thus~$\Delta_K(t)$, can be computed from a presentation of $\pi$ using the free differential calculus due to Fox \cites{fox:free_diff_calc,fox_crowell}.

\begin{proposition}\label{prop:Alex_polys}
    For every $m \geq 1$, the Alexander polynomials of $K_{B_m}$ and $K_{G_m}$, up to multiplication by units in $\Z[t,t^{-1}]$, are
    \begin{align*}
        \Delta_{K_{B_m}}(t)  =         \Delta_{K_{G_m}}(t)  =        \sum_{k=0}^{2m}(-1)^{k}t^{k-2}. 
    \end{align*}
\end{proposition}

\begin{proof}
    Fix $m \geq 1$. By \Cref{prop:fund_groups_expl}, both knot groups $\pi_1( K_{B_m})\cong \pi_1(K_{G_{m}})$ have a presentation~$F_m$ with generators $x$, $y$ and $a$ and relations 
    \begin{align*}
        r_1 = x^{-1} (yx)^m y (yx)^{-m} \qquad \text{and} \qquad 
        r_2 = (x^{-1}a x) a^{-1} x^{-1}( y a y^{-1} ).  
    \end{align*}
    Given a presentation $
\langle x_1,\ldots,x_n \mid r_1,\ldots,r_m \rangle
$ of a knot group, one obtains a presentation of the unreduced Alexander module from the matrix \(V=(v_{ij})\), where~\(v_{ij}\) is the image of the Fox derivative \(\partial r_i/\partial x_j\) under the substitution \(x_k \mapsto t\) for all~\(k=1,\ldots,n\). The Alexander polynomial of the knot can then be recovered as the greatest common divisor of the $(n-1) \times (n-1)$ minors of $V$.
    Computing the Fox derivatives yields
    \begin{align*}
    \frac{\partial r_1}{\partial x} &=-x^{-1}+x^{-1}\left(\left(\sum_{k=0}^{m-1} (yx)^k \right) y + (yx)^m y \left(\sum_{k=0}^{m-1} (yx)^{-k} \right) \left(-x^{-1}\right)\right),\\
    \frac{\partial r_1}{\partial y} &= x^{-1}\left(\sum_{k=0}^{m} (yx)^k -(yx)^my (yx)^{-1}\sum_{k=0}^{m-1} (yx)^{-k} \right),\\ 
    \frac{\partial r_1}{\partial a}&=0,\qquad \\
    \frac{\partial r_2}{\partial x}&=-x^{-1}+ x^{-1}a-x^{-1}axa^{-1}x^{-1},\\
    \frac{\partial r_2}{\partial y}&=x^{-1}axa^{-1}x^{-1}\left(1-yay^{-1}\right), \qquad \\
    \frac{\partial r_2}{\partial a}&=x^{-1}\left(1+ax\left(-a^{-1}+a^{-1}x^{-1}y\right)\right).
    \end{align*}
    Substituting $t$ for $x$, $y$ and $a$ gives us
    \[
    V=
        \begin{pdmatrix}
           \frac{\partial r_1}{\partial x} & \frac{\partial r_1}{\partial y} & \frac{\partial r_1}{\partial a} \\
            \frac{\partial r_2}{\partial x} & \frac{\partial r_2}{\partial y} & \frac{\partial r_2}{\partial a}
        \end{pdmatrix}_{\{x=y=a=t\}}=
        \begin{pdmatrix}
            \sum_{k=-1}^{2m-1}(-1)^{k}t^{k} & -\sum_{k=-1}^{2m-1}(-1)^{k}t^{k} & 0 \\
            1-2t^{-1} & t^{-1}-1 & t^{-1}
    \end{pdmatrix}.
    \]
    Computing the determinant of a $(2\times 2)$-submatrix, we obtain the claim.
\end{proof}

\subsection{Proof of \texorpdfstring{\Cref{theorem:non-isotopic}}{Theorem 1.1}}
\label{subsec:proof_thm_non-iso}

We now  prove the following theorem, which directly implies \Cref{theorem:non-isotopic}.

\begin{theorem}\label{theorem:non-isotopic_compl}
    For every $n\geq 4$, there exist families 
    $(K_B^\ell)_{\ell \geq 1}$ and~
    $(K_G^\ell)_{\ell \geq 1}$ of knots in $S^n$ such that
    \begin{enumerate}[label=(\roman*)]
        \item the exteriors of $K_B^\ell$ and $K_B^{\ell'}$
           are homotopy equivalent if and only if 
        $\ell=\ell'$,\label{thm:i}
        \item the exteriors of $K_G^\ell$
        and $K_G^{\ell'}$ are homotopy equivalent if and only if $\ell=\ell'$,\label{thm:ii}
        \item for all $\ell$, the exteriors of 
        $K_B^\ell$ and $K_G^\ell$ are not homotopy equivalent rel boundary (see \Cref{def:homo_equiv_rel_bound}), but \label{thm:iii}
        \item for all $\ell$, the traces $X(K_B^\ell)$ and 
        $X(K_G^\ell)$ are o.p.~diffeomorphic. \label{thm:iv}
    \end{enumerate}
In particular, for each $\ell \geq 1$, the two knots~$K_B^\ell$ and $K_G^\ell$ are not ambiently isotopic, but have o.p.~diffeomorphic traces. 
\end{theorem}

\begin{proof}
    Fix $n \geq 4$. For every $m \geq 1$, the knots $K_{B_m}$ and~$K_{G_m}$ in $S^n$ have o.p.~diffeomorphic traces (see \Cref{prop:summ_constr}), so \ref{thm:iv} holds. By \Cref{prop:Alex_polys}, the Alexander polynomials $\Delta_{K_{B_m}} = \Delta_{K_{G_m}}$ and $\Delta_{K_{B_{m^\prime}}} = \Delta_{K_{G_{m^\prime}}}$ have distinct breadths (\ie differences between highest and lowest non-trivial powers) when $m \neq m^\prime$. This shows \ref{thm:i} and~\ref{thm:ii}. Now suppose that $m \equiv 1 \pmod{60}$. By \Cref{cor:compl_not_homot_rel_boundary}, the exteriors of~$K_{B_{m}}$ and~$K_{G_{m}}$ in $S^n$ are not homotopy equivalent rel boundary. Setting $K_B^\ell\defeq K_{B_{m}}$
    and~$K_G^\ell\defeq K_{G_{m}}$ for $m= 60 \ell+1$, $\ell \geq 1$, yields the desired infinite families satisfying~\ref{thm:i}--\ref{thm:iv}.
\end{proof}

\subsection{On a higher-dimensional light bulb theorem for links}\label{sec:LBT}
The classical light bulb theorem states that any 1-knot in $S^1 \times  S^2$ intersecting $\{*\} \times  S^2$ in a single point is isotopic to $S^1 \times \{*\}$. The following analogous result for knots in~$S^2 \times S^{n-2}$ for~$n \geq 5$, under an additional condition stated below, was proven by Litherland~\cite{litherland:LBT}. 
Denote~by $$\iota\colon S^2\times S^{n-2}  \hookrightarrow \partial D^3 \times D^{n-1} \hookrightarrow \partial D^{n+2} = S^{n+1}$$ the standard inclusions. 
Litherland showed that if $S$ is a smoothly embedded $(n-2)$-sphere in $S^2 \times S^{n-2}$ that is homologous (equivalently homotopic) to $\{\text{pt}\} \times S^{n-2}$ and intersects~$S^2 \times \{\text{pt}\}$ transversely in one point, and additionally~$\iota (S) \subseteq S^{n+1}$ bounds a smoothly embedded $(n-1)$-ball in $S^{n+1}$, then $S$ is isotopic to $\{\text{pt}\} \times S^{n-2}$. Note that the condition that $\iota(S)$ bounds a smooth ball in $S^{n+1}$ is automatically satisfied when $n=4$, since all 2-knots are slice~\cites{kervaire_french}. More recently, the $4$-dimensional light bulb theorem was proven by Gabai~\cite{LBT}. It states that if $S$ is a smoothly embedded $2$-sphere in $S^2 \times S^2$ that is homologous (equivalently homotopic) to~$\{\text{pt}\} \times S^2$ and intersects~$S^2 \times \{\text{pt}\}$ transversely in one point, then $S$ is isotopic to~$\{\text{pt}\} \times S^2$ via an isotopy fixing $S^2 \times \{\text{pt}\}$ pointwise.

In contrast to the above results for knots, the proof of \Cref{theorem:non-isotopic_compl} in the previous sections yields the following statement for links. 

\begin{corollary}\label{cor:LBT}
    Let $n \geq 4$. There is a $2$-component link $L$ in $S^2 \times S^{n-2}$ such that 
    \begin{enumerate}[label=(\roman*)]
        \item the image of each component of $L$ is isotopic (as an oriented submanifold) in~$S^2 \times S^{n-2}$ to $\{z\} \times S^{n-2}$ for $z \in S^2$, but \label{item:one}
        \item $L$ is not isotopic to the link $\{z_1, z_2\} \times S^{n-2}$ for $z_1 \neq z_2 \in S^2$. \label{item:two}
    \end{enumerate}
\end{corollary}

\begin{proof}
    Consider an RBG manifold for which the associated knots $K_B$ and $K_G$ are non-isotopic. Examples of such manifolds are given in the proof of \Cref{theorem:non-isotopic_compl}. This proof shows that $B(S^{n-2})$ and $G(S^{n-2})$ are both isotopic to $\{\operatorname{pt}\} \times S^{n-2}$ in~$S^2 \times S^{n-2}$; see in particular \Cref{lem:GequalsGtilda,prop:RBG-mfd_from_constr}. However, according to \Cref{lemma:simult_dual}, we cannot simultaneously isotope~$B(S^{n-2})$ and $G(S^{n-2})$ to $\{\operatorname{pt}\} \times S^{n-2}$. Therefore, the link $B(S^{n-2}) \cup G(S^{n-2})$ in $S^2 \times S^{n-2}$ fulfils the desired properties. 
\end{proof}

A light bulb theorem for links of spheres in a 4-manifold whose fundamental group has no $2$-torsion was proved in \cite{LBT}*{Section 10} under the assumptions that each component admits a geometrically dual sphere that is disjoint from the other components of the link and that these geometrically dual spheres are pairwise disjoint. \Cref{cor:LBT} above demonstrates that this is a necessary condition for a higher-dimensional light bulb theorem for links to hold. 

\begin{remark}\label{rem:LBT}
We present an alternative proof, which also applies in ambient dimension three. We thank the anonymous referees for their significant help in developing this alternative argument for $n \geq 5$. Let $\Sigma_0$ denote a standard sphere $\{x_1\} \times S^{n-2}$ in $S^2 \times S^{n-2}$, and let $K$ be a non-trivial knot in $S^n$. Denote by $\Sigma_1$ the connected sum of a second standard sphere $\{x_2\} \times S^{n-2}$ in $S^2 \times S^{n-2}$, with $K$. We consider the link $L:=\Sigma_0 \sqcup \Sigma_1\subseteq S^2\times S^{n-2}$. By \Cref{lemma:simult_dual}, these links are not interesting from the point of view of RBG manifolds, because they produce isotopic associated knots in~$S^n$. But we can apply them to prove \Cref{cor:LBT}, as follows. 

Observe that the fundamental group of the exterior of $L$ is isomorphic to $\pi_1(K)$. When $n=3$ or $4$, the link $L$ satisfies \ref{item:one} in \Cref{cor:LBT} by the classical and Gabai's~\cite{LBT} light bulb theorem respectively, regardless of $K$. We can ensure that $L$ satisfies \ref{item:two} in \Cref{cor:LBT} by choosing $K$ to have $\pi_1(K)\not\cong \Z$.

What about larger $n$? According to \cite{litherland:LBT}*{Remark~2}, for $n\geq 5$, the link $L$ satisfies~\ref{item:one} in \Cref{cor:LBT} if and only if $K$ lies in the kernel of the map $\mathcal{C}^{n,n-2}\to \mathcal{C}^{n+1,n-2}$, where $\mathcal{C}^{j,k}$ denotes the concordance group of $k$-knots in $S^j$.
Take a knot $J$ with~$\pi_1(J) \not\cong  \Z$ (e.g.~obtained by iteratively 0-twist spinning a nontrivial $1$-knot), and set $K\defeq J\#-J$. Then $\pi_1(K) \not\cong  \Z$ and $K$ is slice, so $K$ lies in the kernel of the map $\mathcal{C}^{n,n-2}\to \mathcal{C}^{n+1,n-2}$. Now, the same argument as above applies and re-proves Corollary 4.15.
\end{remark}

As noted above, for the RBG manifolds constructed in \Cref{subsec:RBG_constr_explicit}, the spheres $B_m(S^{n-2})$ and~$G_m(S^{n-2})$ cannot be simultaneously geometrically dual to the belt sphere $\mathcal{S}_R$. Note that we can also recover this fact from the presentation of the knot groups~$F_m$ computed in \Cref{prop:fund_groups_expl}. By \Cref{lemma:simult_dual}, if $B(S^{n-2})$ and~$G(S^{n-2})$ admitted a simultaneous geometric dual, then the elements~$x, a$ representing $\mu_{K_{B_m}}$ and~$\mu_{K_{G_m}}$ would be conjugate in~$F_m$. However, for every $m \equiv 1 \pmod{60}$, the map 
\begin{align*}
    \varphi \colon F_m \to A_5, \quad 
    x \mapsto (15432), \quad 
    y \mapsto (12453),\quad 
    a \mapsto (245)
\end{align*}
defines a homomorphism which does not map $x$ and $a$ to conjugate elements in $A_5$.

\section{Surgeries, traces and Gluck twists}
\label{sec:traces_Gluck}

In this section, we will leave the realm of RBG manifolds and explore the relationship between surgeries, traces and Gluck twists. This will lay  the groundwork for constructing, in \Cref{sec:inf_family}, for every $n \geq 5$, infinitely many non-isotopic knots in $S^n$ that all have o.p.~diffeomorphic traces. 

Throughout this section we assume that $n \geq 4$.

\subsection{Gluck twists}\label{subsec:gluck}
We start by discussing Gluck twists.

\begin{definition}\label{def:gluck_twist}
    Let $K \colon S^{n-2} \hookrightarrow S^n$ be a knot in $S^n$, and fix an identification of~$\nu K$ with $ D^2 \times S^{n-2}$. The \emph{Gluck twist of~$S^n$ along $K$} is the oriented, connected, closed, smooth $n$-manifold
    \begin{align*}
        S^n_K=(S^n \setminus \overset{\circ}{\nu} K) &\cup_{\psi} (D^2 \times S^{n-2}), \qquad \text{where}\\
        \psi \colon S^1 \times S^{n-2} \cong \partial (S^n \setminus \overset{\circ}{\nu} K) &\to \partial D^2 \times S^{n-2}=S^1 \times S^{n-2},\\
        (e^{it},x) &\mapsto (e^{it}, \rho_t(x)),
    \end{align*}
    and $\rho_t$ denotes the rotation of the $(n-2)$-sphere $\{e^{it}\} \times S^{n-2}$ about its standard diameter by angle $t$; see \cite{Gluck} for the case $n=4$. The embedding of~$\{1\} \times S^{n-2}$ defines a knot $T(K)\colon S^{n-2} \hookrightarrow S^n_K$, which we call the \emph{Gluck twist of $K$}.
\end{definition}

Notice that Gluck twisting $S^n_K$ along $T(K)$, similarly to the operation in \Cref{def:gluck_twist}, yields the original knot $K$, \ie $T(T(K))=K$. In dimension $n \ge 4$, for any $(n-2)$-knot, there is at most one other $(n-2)$-knot (up to reparametrisation) with homeomorphic complement. If $K$ and $J$ are non-isotopic such knots, then~$K$ has to be isotopic to the Gluck twist $T(J)$ of $J$. This has been proved by Gluck~\cite{Gluck} in dimension $n=4$, and by Browder and Lashof--Shaneson~\cites{browder, lashofshaneson} in dimension~$n\ge 5$. Since $K$ and $T(K)$ have diffeomorphic exteriors $S^n \setminus \overset{\circ}{\nu} K \cong S^n_K \setminus (D^2 \times S^{n-2}) $ and the Gluck twist map~$\psi$ extends over $S^1 \times D^{n-1}$, it follows that~$K$ and $T(K)$ also share the same surgery, \ie $S^n_K(T(K)) \cong S^n(K)$. On the other hand, their traces might differ, as we will discuss below in \Cref{criterion}.

\begin{remark}\label{rem:Gluck_manifold}
    For any knot $K$ in $S^n$, using the Seifert--van Kampen theorem and the Mayer--Vietoris sequence, one can easily see that the Gluck twist $S^n_K$ along $K$ is a simply connected homology sphere. By a standard argument using the Hurewicz and Whitehead theorems, it is a homotopy $n$-sphere. As a consequence of the topological Poincaré conjecture, $S^n_K$ is hence homeomorphic to~$S^n$ \cites{freedman:top_4-mfds, smale}. Nevertheless, as a smooth $n$-manifold, $S^n_K$ could possibly be a non-standard $n$-sphere.
\end{remark}

Next, we present a condition under which the Gluck twist of~$S^n$ along a knot $K$ is diffeomorphic to the standard smooth $n$-sphere in dimension $n \geq 5$.

\begin{proposition}\label{prop:std}
    For $n \geq 5$, let $K \colon S^{n-2} \hookrightarrow S^n$ be a knot in $S^n$. Suppose that $K$ extends to an embedding $D^{n-1} \hookrightarrow D^{n+1}$. Then $S^n_K$ is diffeomorphic to~$S^n$.
\end{proposition}

\begin{remark}\label{rem:slice}
    When $S^{n-1}$ has a unique smooth structure (see \Cref{REM:embeddings}), the assumption on the knot $K$ in \Cref{prop:std} is equivalent to asking $K$ to be \emph{slice}, \ie to asking $K(S^{n-2}) \subseteq S^n = \partial D^{n+1}$ to bound an $(n-1)$-dimensional smoothly and properly embedded disc in~$D^{n+1}$.
    For even $n$, all knots in $S^n$ are slice \cites{kervaire_french, kervaire:knot_cob}.
    In particular, for each such knot, there is a reparametrisation of its embedding for which we can apply the proposition. 
    For $n \in \{6, 62\}$, which are the even $n \geq 5$ such that $S^{n-1}$ is known to have a unique smooth structure (see \Cref{REM:embeddings}), we obtain that $S^n_K$ is diffeomorphic to~$S^n$ for \emph{every} knot $K$ in~$S^n$.
\end{remark}

\begin{proof}[{Proof of \Cref{prop:std}}]
We will construct an $h$-cobordism from $S^n$ to $S^n_K$. To that end, consider the $(n+1)$-manifold $W$ obtained by first attaching an $(n-1)$-handle to~$D^{n+1}$ along $K$ and then attaching a $2$-handle along the dual curve $\gamma_K$ in the surgery~$S^n(K)$ of $S^n$ along $K$.
Note that, since $\gamma_K$ is the belt sphere of the $(n-1)$-handle, it naturally inherits a framing. We do not use this framing, but the \emph{opposite} framing for the latter handle attachment. Here we are using $\pi_1(\operatorname{SO}(n-1))\cong \Z/2\Z$.

The boundary of $W$ is indeed $S^n_K$. In fact, $\partial W$ is the result of first performing surgery on $S^n$ along $K$ and then performing a loop surgery on the resulting manifold along the dual curve $\gamma_K$, using the opposite framing. By the definition of surgery (see \Cref{subsec:framings}), concatenating the two operations produces a smooth $n$-manifold given by gluing the complement of $K$ to a copy of $S^{n-2}\times D^2$. Choosing the framing naturally inherited by $\gamma_K$ would give back $S^n$, while the opposite framing yields the Gluck twist $S^n_K$.

The Seifert--van Kampen theorem and a Mayer--Vietoris sequence computation imply that $W$ is a simply connected $(n+1)$-manifold, $H_k(W)\cong \Z$ for~$k \in\{ 0,2,n-1\}$ and~$H_k(W)=0$ otherwise. Notice that a generator of $H_{n-1}(W)$ is represented by a smoothly embedded $(n-1)$-sphere $S$ given by the union of the core of the $(n-1)$-handle and a slice disc in $D^{n+1}$ for $K$.
By \Cref{rem:framings}, $S$ has a trivial normal bundle in $W$ and its framing is unique. We now claim that the manifold
\[B\defeq (W \setminus \nu S) \cup_\partial (D^n \times S^1)\]
obtained by surgery on $W$ along $S$ is a homotopy $(n+1)$-ball, whose boundary is $S^n_K$. First notice that $\partial B= \partial W$ and therefore it coincides with $S^n_K$. A Mayer--Vietoris argument shows that $H_k(W\setminus\nu S)\cong\Z$ for $k \in \{0,n-1,n\}$ and all other homology groups are trivial, where generators for $H_{n-1}(W\setminus\nu S) \cong H_{n-1}(W)$ and $H_n(W\setminus\nu S) \cong H_n(\partial \nu S)$ can be represented by a parallel copy of $S$ and by $\partial\nu S$, respectively. Since $B$ is obtained by gluing~$D^n\times S^1$ onto $\partial\nu S$, it is easy to check that all the homology groups $H_k(B)$ for $k > 0$ are trivial. Moreover, by the Seifert--van Kampen theorem, $B$ is simply connected and hence a homotopy $(n+1)$-ball as claimed.
By removing a small ball in the interior of~$B$ we obtain an $h$-cobordism between $S^n$ and $S^n_K$, and the conclusion follows from the smooth $h$-cobordism theorem~\cite{smale}.
\end{proof}

\subsection{Surgery diffeomorphisms and traces}\label{subsec:surgerydiffeos}

We now prove \Cref{Glucktwist}, which we restate for the convenience of the reader. 

\begin{reptheorem}{Glucktwist}
    For $n \geq 5$, let $K_1$ and $K_2$ be knots in $S^n$ and let $\varphi \colon S^n(K_1) \to S^n(K_2)$ be an o.p.~diffeomorphism between their surgeries. Then either there is an o.p. diffeomorphism $X(K_1) \cong X(K_2)$ or the Gluck twist~$T(K_1)$ is a knot in the standard smooth $n$-sphere with $X(T(K_1)) \cong X(K_2)$.
\end{reptheorem}

\begin{remark}
    Note that \Cref{Glucktwist} does not hold in dimension $n=3$~\cite{manolescu-piccirillo:RBG}; compare also~\cites{Akbulut,Boyer1}. In fact, the proof of \Cref{Glucktwist} is based on the same idea as that presented in~\cite{Akbulut}. 
\end{remark}

\begin{proof}[{Proof of \Cref{Glucktwist}}]
For $i=1,2$, we denote by $\gamma_i \subseteq S^n(K_i)$ the dual curve of $K_i$ such that the knot~$K_i$ can be recovered by loop surgery on $S^n(K_i)$ along $\gamma_i$ with some framing, which we denote by $f_i$.\footnote{See \Cref{subsec:framings} for the definitions of ``dual curve'' and ``loop surgery''.} Since $X(K_2)$ is simply-connected and of dimension~$n+1 \geq 5$, there is a properly embedded $2$-disc $D$ in $X(K_2)$ bounded by~$\varphi(\gamma_1)$. We can write
\[X(K_2)=(X(K_2) \setminus \overset{\circ}{\nu} D)\cup\nu D.\]
Using standard arguments as in \Cref{rem:Gluck_manifold}, one can then check that $V \defeq X(K_2) \setminus \overset{\circ}{\nu} D$ is simply-connected and that its boundary
\(\partial V =\partial(X(K_2) \setminus \overset{\circ}{\nu} D)\) is a homotopy $n$-sphere. More precisely, $\partial V\cong \varphi(S^{n}(K_{1})\setminus \overset{\circ}{\nu}\gamma_{1}) \cup_{\partial} S^{n-2}\times D^{2}$ is diffeomorphic either to $S^{n}$, when the push-forward via $\varphi$ of $f_{1}$ extends over $D$, or to the Gluck twist $S^{n}_{K_{1}}$, when the push-forward of $f_{1}$ does not extend. Moreover, one can show that $H_i(V) \cong 0$ for $i>0$. This follows from a Mayer--Vietoris argument on the pair~$(V, \nu D)$, noting for $i=n-2$ that $V \cap \nu D \cong D^2 \times S^{n-2}$ is included in $V$ via a map which factors through the homology sphere $\partial V$. Since $V$ is simply connected and has the homology of a point, by the Hurewicz and Whitehead theorems, it follows that $V$ is contractible.
Hence, by the smooth $h$-cobordism theorem~\cite{smale} applied to $V$ with a small ball from its interior removed, $V$ is a smooth $(n+1)$-ball.

This last step is where we are using the assumption~$n \geq 5$. Now, $\nu D$ is attached to~$V=X(K_2) \setminus \overset{\circ}{\nu} D$ as an $(n-1)$-handle, with attaching sphere an~$(n-2)$-knot $A$ in~$ \partial V$ which is the core of the loop surgery on $\varphi(S^n(K_1))$ along~$\varphi(\gamma_1)$. See \cref{fig:extending_surgery_diffeo} for a pictorial description of our setting. We will now distinguish two cases.

\begin{figure}[htbp] 
	\centering
    \includegraphics[width=0.7 \textwidth]{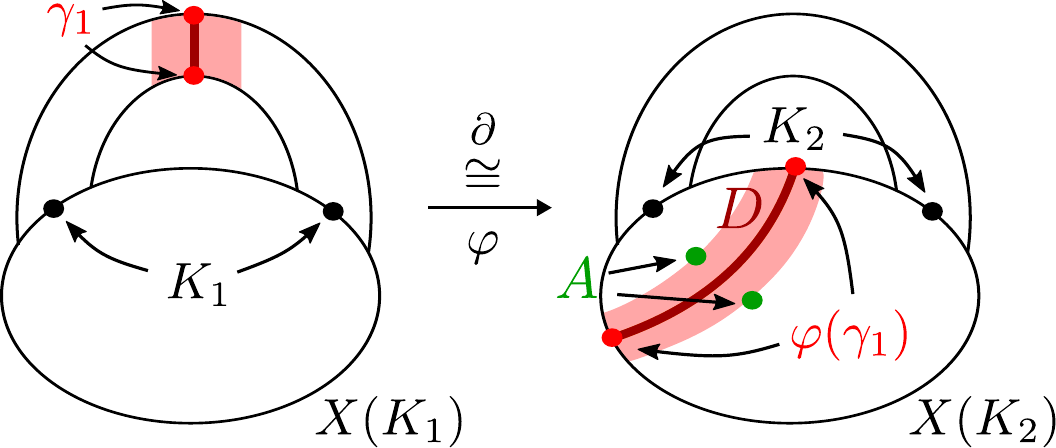}
    \caption{Extending a surgery diffeomorphism to the traces.}
    \label{fig:extending_surgery_diffeo}
\end{figure}

First, suppose that the push-forward $\varphi_\ast(f_1)$ of the framing $f_1$ on~$\varphi(\gamma_1)$ extends over~$D$. In this case, $A$ is equivalent to the core of the loop surgery on $S^n(K_1)$ along~$\gamma_1$ with framing $f_{1}$. This means that the attaching sphere of $\nu D$ is equivalent to $K_1$, which yields an identification
\begin{align}\label{eq:traces_h-cob}
    X(K_2) = V \cup_A \nu D \cong D^{n+1} \cup_{K_1=S^{n-2} \times \{0\}} (D^{n-1} \times D^2) = X(K_1).
\end{align}
On the other hand, it can also happen that the framing $\varphi_*(f_1)$ does not extend over~$D$. Then we use the fact that surgery on $S^n(K_1)$ along $\gamma_1$ with framing not equivalent to~$f_1$ yields the pair $(S^n_{K_1}, T(K_1))$.
As noted above, in this case $S^n_{K_1}$ is diffeomorphic to the boundary of the $(n+1)$-ball $V$, so it is diffeomorphic to $S^n$. Moreover, this diffeomorphism identifies $A$ and $T(K_1)$. This implies that \[X(K_2) = V \cup_A \nu D \cong D^{n+1} \cup_{T(K_1)=S^{n-2} \times \{0\}} (D^{n-1} \times D^2)=X(T(K_1)).\qedhere \]
\end{proof}

The same proof yields a similar result in the topological category for~$n=4$. The reason why we do not obtain a statement in the smooth setting for $n=4$ is the failure of the smooth $h$-cobordism theorem in this dimension, which follows from the work of Donaldson~\cite{Donaldson}. 
In particular, the homotopy $4$-sphere that we obtain is homeomorphic but not necessarily diffeomorphic to $S^4$.

Before stating our result in the topological category for $n=4$, let us remark that the Gluck twist~$T(K)$ of a knot $K$ is a smooth $2$-knot in the possibly exotic $4$-sphere~$S^4_{K}$ (see \Cref{rem:Gluck_manifold}). For such knots we define the \textit{trace} $X(T(K))$ of $T(K)$ by attaching a $3$-handle along $T(K)\times \{1\}$ to $S^4_{K}\times [0,1]$ and then gluing $D^4$ to the resulting manifold via a homeomorphism between $\partial D^4$ and $S^4_{K}\times \{0\}$. It follows from the Alexander trick that this manifold is well-defined up to o.p.~homeomorphism.

\begin{theorem}\label{Glucktwisttop}
    Let $K_1$ and $K_2$ be $2$-knots in $S^4$ and let $\varphi \colon S^{4}(K_{1}) \to S^{4}(K_{2})$ be an o.p.~diffeomorphism between their surgeries. Then either there is an o.p.~diffeomorphism $X(K_1) \cong X(K_2)$ or there is an o.p.~homeomorphism $X(T(K_1)) \approx X(K_2)$. If the Gluck twist $S^4_{K_1}$ is a standard smooth $4$-sphere, then we can conclude (as in the case $n \geq 5$) that either $X(K_1) \cong X(K_2)$ or $X(T(K_1))\cong X(K_2)$.
\end{theorem}

\begin{proof}
The proof works similarly as in the case $n \geq 5$. In particular, if the push-forward $\varphi_*(f_1)$ of the framing $f_1$ of $\gamma_{K_1}$ extends over $D$, then we can still conclude that~$V$ is a smooth $5$-disc since $\partial V$ is a standard smooth $4$-sphere and we can hence apply \cite{kosinski:differential_manifolds}*{Chapter~VIII, Corollary~4.7}. If $\varphi_*(f_1)$ does not extend over~$D$, then we can only conclude that $V$ is a topological $5$-disc,\footnote{Since $\partial V=S^4_{K_1}$ might not be diffeomorphic to $S^4$, we cannot apply \cite{kosinski:differential_manifolds}*{Chapter~VIII, Corollary~4.7} to conclude that $V$ is a standard $D^5$.} and hence there is a homeomorphism $X(T(K_1)) \approx X(K_2)$. However, since~$\partial V$ is identified with the Gluck twist $S^4_{K_1}$, in the case in which $S^4_{K_1}$ is a standard $S^4$ we can again apply~\cite{kosinski:differential_manifolds}*{Chapter~VIII, Corollary~4.7} to conclude that~$V$ is a standard smooth $5$-disc and find a diffeomorphism $X(T(K_1)) \cong X(K_2)$. 
\end{proof}

\begin{remark}\label{rmk:extension}
  In the cases in which $S^{n+1}$ has a unique smooth structure (see \Cref{REM:embeddings}), the diffeomorphisms $X(K_1) \cong X(K_2)$ and $X(T(K_1))\cong X(K_2)$ can be constructed as smooth extensions of the map $\varphi\colon S^n(K_1) \to S^n(K_2)$ in \Cref{Glucktwist,Glucktwisttop}. This follows from the fact that one can extend the identification induced by~$\varphi$ of $(S^{n},K_{1} ) $ or $ (S^{n},T(K_{1}) )$ with $(\partial V,A)$ by noting that every self-diffeomorphism of the $n$-sphere extends to the $(n+1)$-ball in this case; see \Cref{lem:extending_sphere_diffeo_to_disc_diffeo} below.
  Any self-homeomorphism of the $n$-sphere extends to the $(n+1)$-ball by the Alexander trick. Therefore we can always extend $\varphi$ to a homeomorphism~$X(K_1) \approx X(K_2)$ or~$X(T(K_1))\approx X(K_2)$.
\end{remark}

\begin{lemma}[{See \eg \cite{kosinski:differential_manifolds}*{Chapter~VIII, Corollary~5.6}.}]\label{lem:extending_sphere_diffeo_to_disc_diffeo}
    If the $(n+1)$-dimensional smooth Poincaré conjecture is true, then any self-diffeomorphism of the $n$-sphere extends to a diffeomorphism  of the $(n+1)$-disc. \qed
\end{lemma}

\subsection{Extending surgery diffeomorphisms}

By \cref{Glucktwisttop} and \cref{rmk:extension}, any surgery diffeomorphism $\varphi \colon S^4(K_1) \to S^4(K_2)$ for $2$-knots~$K_1$ and $K_2$ in $S^4$ extends either to a trace diffeomorphism $X(K_1)\to X(K_2)$ or to a homeomorphism $X(T(K_1))\to X(K_2)$. The following theorem provides the complete answer as to which of these two cases occurs, depending on the algebraic topology of the traces.

\begin{reptheorem}{extension_intro}
    Let $K_1$ and $K_2$ be $2$-knots in $S^4$. An o.p.~surgery diffeomorphism \(\varphi\colon S^4(K_1) \to S^4(K_2)\) extends to an o.p.~trace diffeomorphism \(\Phi\colon X(K_1) \to X(K_2)\) if and only if the closed \(5\)-manifold \(X(K_1) \cup_{\varphi} - X(K_2)\) is spin.
\end{reptheorem}

Both the trace and the surgery along a 2-knot $K$ are spin, since 
\begin{align*}
    H^2(X(K);\Z/2)=0=H^2(S^4(K);\Z/2).
\end{align*} Recall that isomorphism classes of spin structures on a CW complex $X$ are in bijection with $H^1(X;\Z/2)$ (see e.g.~\cite{gompf-stipsicz:book}*{Section~1.4}). Since $H^1(X(K_i);\Z/2)=0$ for each $i \in \{1,2\}$, there is a unique spin structure $\mathfrak{s}_i$ on $S^4(K_i)$, up to isomorphism, that extends over $X(K_i)$. Therefore, \Cref{extension_intro} says that an o.p.~surgery diffeomorphism~$\varphi\colon S^4(K_1)\to S^4(K_2)$ extends to an o.p.~trace diffeomorphism $\Phi\colon X(K_1)\to X(K_2)$ if and only if~$\varphi^*(\mathfrak{s}_1)=\mathfrak{s}_2$.

\begin{remark}\label{rem:sth}
    \Cref{extension_intro} is analogous to the corresponding result for \(1\)-knots referred to in \cite{manolescu-piccirillo:RBG}*{Theorem 3.7}, which is deduced as a special case of Boyer's classification of simply-connected \(4\)-manifolds with boundary \cite{Boyer1}. The aforementioned result concerns the extension of homeomorphisms of the $0$-surgeries of $1$-knots to homeomorphisms of the $0$-framed traces. In the case of $2$-knots, however, we prove a statement in the smooth category. On the other hand, we note that \Cref{Glucktwist,extension_intro} also hold true in the topological category, \ie for locally flat knots in $S^n$. More care must be taken in this category when defining and working with all the relevant terms. We will not do this here, but just refer the reader to the excellent survey by Friedl--Nagel--Orson--Powell~\cite{friedl2024surveyfoundationsfourmanifoldtheory}.
\end{remark}

\begin{proof}[Proof of \cref{extension_intro}]
    We follow the proof indicated by \cite{manolescu-piccirillo:RBG}*{Theorem 3.7}, adapting it to one dimension up, and using Barden's classification of simply-connected \(5\)-manifolds~\cite{BardenClassify}.

    The ``only if'' argument is the same as in the classical case. Assume that \(\varphi\) extends to a diffeomorphism \(\Phi\). From \(\Phi\) we can build a diffeomorphism \[F\colon X(K_1) \cup_{\id}-X(K_1) \to X(K_1) \cup_{\varphi} - X(K_2).\] 
    Since \(X(K_1)\) is spin, also $X(K_1) \cup_{\id}-X(K_1)$ is spin, and thus \(X(K_1) \cup_{\varphi} - X(K_2)\) must be spin as well.

    For the ``if'' direction, we wish to use the assumption that \(Z \defeq X(K_1) \cup_{\varphi} - X(K_2)\) is spin to construct an $h$-cobordism between the traces. First, notice that \(Z\) is a closed, simply-connected \(5\)-manifold. Furthermore, \(Z\) is spin and has \(H_2(Z) \cong \Z\), so by \cite{BardenClassify}*{Theorem 2.3}, we obtain a diffeomorphism \(Z\cong S^2\times S^3\). Thus \(Z\) bounds a \(6\)-manifold \(W\) diffeomorphic to~$D^3 \times S^3$, which provides us with a cobordism between \(X(K_1)\) and \(X(K_2)\). We will show that it is an $h$-cobordism. Indeed, by the Hurewicz and Whitehead theorems, it suffices to show that the inclusions \(X(K_i) \hookrightarrow
    W\) induce isomorphisms on each homology group. Observe that the only homology groups we need to consider are $H_3(X(K_i))$, $i \in \{1,2\}$, and~$H_3(W)$, which are all isomorphic to $\Z$. Indeed, the inclusions $X(K_i) \hookrightarrow Z$ and $Z \hookrightarrow W$ induce isomorphisms on the level of $H_3$. Hence~\(W\) gives rise to an $h$-cobordism. By the \(6\)-dimensional $h$-cobordism theorem~\cite{smale}, we obtain a diffeomorphism \(X(K_1) \to X(K_2)\) extending~\(\varphi\).
\end{proof}

\begin{remark}\label{rem:extension_higher}
    For the ``if'' direction in the proof of \Cref{extension_intro}, we used Barden's classification of simply-connected $5$-manifolds~\cite{BardenClassify}. We record the following alternative direct approach, which could potentially be used to generalise \Cref{extension_intro} to all dimensions $n~\geq~4$. First, use Whitehead's theorem and the fact that $Z=X(K_1) \cup_{\varphi} - X(K_2)$ is spin to construct a homotopy equivalence of $Z$ to~$S^2\times S^{n-1}$. Then use surgery theory of simply-connected manifolds to upgrade this homotopy equivalence to a diffeomorphism, using the fact that the structure set of \(S^2\times S^{n-1}\) for \(n\geq 5\) has been determined in the smooth category \cite{crowley2010smoothstructuresetsp}. The rest of the proof would then proceed exactly as the proof of \Cref{extension_intro}.
\end{remark}

We observe that \Cref{extension_intro} implies the following corollary, which is in sharp contrast to the situation in dimension $n=3$ (see for example~\cites{Akbulut1,Akbulut2,Akbulut3}).

\begin{corollary}\label{cor:exotic_traces}
    Let $K_1$ and $K_2$ be $2$-knots in $S^4$. If there exists an o.p.~trace homeomorphism $X(K_1) \to X(K_2)$ that restricts to a diffeomorphism of the surgeries, then the traces $X(K_1)$ and $X(K_2)$ are o.p.~diffeomorphic.    
\end{corollary}

\begin{proof}
    Let $\Phi\colon X(K_1) \to X(K_2)$ be a homeomorphism such that its restriction to the surgeries $\varphi\colon S^4(K_1) \to S^4(K_2)$ is a diffeomorphism. Then we can build the smooth $5$-manifold $X(K_1) \cup_{\varphi} - X(K_2)$. By the same argument as in the proof of \Cref{extension_intro}, this manifold is spin and thus $\varphi$ extends to a diffeomorphism of the traces by \Cref{extension_intro}. Note that this does not imply that the original map $\Phi$ is a diffeomorphism.
\end{proof}

\begin{remark}\label{rem:extension}
The ``only if'' direction of \Cref{extension_intro} is true more generally for any $n \geq 3$, \ie given knots $K_1, K_2$ in $S^n$, if a surgery diffeomorphism \(\varphi\colon S^n(K_1) \to S^n(K_2)\) extends to a trace diffeomorphism $X(K_1) \to X(K_2)$, then the closed $(n+1)$-manifold \(X(K_1) \cup_{\varphi} - X(K_2)\) is spin. For $n=3$, this was proven in \cite{manolescu-piccirillo:RBG}, whereas for~$n \geq5$ the proof works exactly the same as the one for $n=4$ given in the proof of \Cref{extension_intro}. We believe that the ``if'' direction of \Cref{extension_intro} is true in all dimensions, too, and thus also \Cref{cor:exotic_traces}; see Question~\ref{ques:extensions} and the discussion there. 
\end{remark}

To finish this section, we record the following direct consequence of \Cref{Glucktwisttop,extension_intro}.

\begin{corollary}\label{criterion}
    Let $K$ be a $2$-knot in $S^4$ such that the Gluck twist $S^4_K$ is a standard smooth~$S^4$. Then $X(K) \cong X(T(K))$ if and only if there exists an o.p.~diffeomorphism $\varphi \colon S^4(K) \to S^4(T(K))$ such that the manifold $X(K) \cup_\varphi -X(T(K))$ is spin.\qed
\end{corollary}

\section{Infinitely many knots with the same trace in dimensions \texorpdfstring{$\geq 5$}{greater than or equal to 5}}\label{sec:inf_family}

In this section, we prove \cref{infinite_example_intro}, restated below, which demonstrates for $n \geq 5$ the existence of infinitely many non-isotopic knots in $S^n$ with o.p.~diffeomorphic traces. This will be deduced by combining Plotnick's work in \cite{Plotnick1983InfinitelyMD} with \Cref{Glucktwist}. 

\begin{reptheorem}{infinite_example_intro}
    Let $n \geq 5$. There is an infinite collection $\{X_i\}_{i \geq 1}$ of pairwise non-diffeomorphic smooth $(n+1)$-manifolds such that, for each integer~$i \geq 1$, there exists an infinite collection~$\{K_j^i \}_{j \geq 1}$ of pairwise non-isotopic knots in $S^n$ whose trace is o.p.~diffeomorphic to $X_i$, \ie
    \begin{align*}
        X(K_j^i) \cong X_i \qquad \text{for all integers } \quad i,j \geq 1.
    \end{align*}
    In particular, for each $i \geq 1$, this results in an infinite family of pairwise non-isotopic knots in $S^n$ with o.p.~diffeomorphic traces. 
\end{reptheorem}

\begin{proof}
In \cite{Plotnick1983InfinitelyMD}*{Theorem~1}, Plotnick showed, for $n \geq 4$, the existence of an infinite collection $\{D_i \}_{i \ge 1}$ of pairwise inequivalent smoothly embedded discs~$D^{n-1}$ in~$D^{n+1} $ with the same exterior. \Cref{infinite_example_intro} will follow by showing that the $(n-2)$-knots $\partial D_i \subseteq S^n= \partial D^{n+1}$ are pairwise inequivalent and that, up to Gluck twisting, they have the same trace. This also works for $n=4$ until we invoke the smooth $h$-cobordism theorem; for now, we will therefore just assume that $n \geq 4$. We will clarify when we start assuming $n \geq 5$ below, then discuss how the argument diverges for $n=4$ in the upcoming proof of \Cref{Plot4}.

There are infinitely many Brieskorn spheres~$\Sigma(p,q,r)$ that bound a contractible smooth $4$-manifold; see \eg \cite{MR634760}. For each such Brieskorn sphere, Plotnick produced an $(n+1)$-manifold $X(p,q,r)$ and an infinite family of pairwise inequivalent (and thus non-isotopic) knots in $S^n$ having~$X(p,q,r)$ as their trace. The details are as follows.
 
Let $\Sigma(p,q,r)$ be a Brieskorn sphere that bounds a contractible smooth $4$-manifold~$W$. For $n \geq 5$, we define the $n$-manifold \[V\defeq (W\times S^{n-4}) \cup h_{n-3}(\{\text{pt}\} \times S^{n-4}),\]
where the $(n-3)$-handle $h_{n-3}$ is attached to $W \times S^{n-4}$ along $\{\text{pt}\} \times S^{n-4}$. For~$n=4$, we define $V \coloneqq W$. We obtain, for every $n \ge 4$, a contractible $n$-manifold~$V$, whose boundary $\Sigma^{n-1} \defeq \partial V$ is a homology $(n-1)$-sphere which for $n=4$, is equal to~$\Sigma(p,q,r)$; while for $n \ge 5$ it is obtained from $\partial W \times S^{n-4}$ by surgery on $\{\text{pt}\} \times S^{n-4}$. For $n> 5$, $\partial W \times S^{n-4}$ has the same fundamental group as $\partial W = \Sigma (p,q,r)$ and surgery on an $(n-4)$-sphere does not affect the fundamental group. When $n=5$, $\pi_1(\partial W \times S^1)= \pi_1(\partial W) \times \Z$, and performing surgery on $\{\text{pt}\} \times S^1$ kills the generator of the $\Z$-factor. Therefore, for every $n \ge 4$, the fundamental group of $\Sigma^{n-1}$ coincides with $ \pi_1(\Sigma(p,q,r))$.

In ~\cite{Plotnick1983InfinitelyMD}, Plotnick showed that there exists an infinite collection $A^{p,q,r}$ of normal generators of $\pi_1(\Sigma^{n-1})$ that are \emph{algebraically distinct modulo the centre $Z(\pi_1(\Sigma^{n-1}))$}. 
This means that, for any pair of distinct elements $\alpha,\beta \in A^{p,q,r}$, there is no automorphism of~$\pi_1(\Sigma^{n-1})$ sending $\alpha$ to $(\beta c)^{\pm 1}$, for any $c\in Z(\pi_1(\Sigma^{n-1}))$. Set $t\defeq [\{\text{pt}\} \times S^1] \in \pi_1(\Sigma^{n-1} \times S^1)$ and perform loop surgery on $\Sigma^{n-1} \times S^1$, with framing $f$, along a simple loop representing an element in $$\{\alpha t \in \pi_1(\Sigma^{n-1} \times S^1) \mid \alpha \in A^{p,q,r}\}.$$ We call the resulting manifold $S^\alpha_f$. Since $\alpha t$ normally generates $\pi_1(\Sigma^{n-1} \times S^1)$, $S^\alpha_f$ is simply connected, and a Mayer--Vietoris argument shows that it is a homology sphere. Note that $S^\alpha_f$ is the boundary of the manifold $B^\alpha_f$ obtained by gluing a $2$-handle $h_2^\alpha = D^2 \times D^{n-1}$ to~$V \times S^1$ along the simple loop representing $\alpha t$ with framing $f$. Since $\pi_1(V \times S^1)\cong \Z$ is generated by $t$, $B^\alpha_f$ is simply connected. A Mayer--Vietoris argument shows that $B^\alpha_f$ is a homology ball. As a consequence of the Whitehead and Hurewicz theorems, $B^\alpha_f$ is a contractible manifold, and its boundary~$S^\alpha_f$ is a homotopy sphere. 

Let us now assume that $n \geq 5$. Then, by the $h$-cobordism theorem~\cite{smale}, $S^\alpha_f$ is a smooth $n$-sphere. The cores $\{\operatorname{pt}\}\times S^{n-2}\subseteq D^2 \times \partial D^{n-1}$ of the surgeries thus give an infinite family of knots 
\begin{align*}
    \mathcal{K}_{p,q,r}=\{K_{\alpha}\colon S^{n-2}\hookrightarrow S^n \}_{\alpha \in A^{p,q,r}}
\end{align*} 
which we claim are pairwise inequivalent, and hence non-isotopic
by \Cref{lem:equivalence_implies_iso}. 

To prove this claim, note that any equivalence between $K_{\alpha}$ and $K_{\beta}$ would induce an automorphism $\varphi$ of the group~$\pi_1(\Sigma^{n-1}\times S^1)$ sending $\alpha t$ to $\beta t$ (up to conjugacy). Since~$\Sigma^{n-1}$ is a homology sphere, $\pi_1(\Sigma^{n-1})$ is the commutator subgroup of ${\pi_1(\Sigma^{n-1} \times S^1)}$, and it is hence fixed by any automorphism of $\pi_1(\Sigma^{n-1} \times S^1) $. The automorphism~$\varphi$ thus produces, by restriction, an automorphism of~$\pi_1(\Sigma^{n-1})$. Since~$\varphi$ induces the identity on the abelianisation, it must send $t$ to $ct$ for some central element $c\in \pi_1(\Sigma^{n-1})$, so that $\varphi(\alpha) = \beta c$, thus contradicting the assumptions on $A^{p,q,r}$. Therefore, the infinite family of knots in $\mathcal{K}_{p,q,r}$ indeed consists of pairwise non-isotopic knots.

By construction, the surgery on any of these knots is o.p.~diffeomorphic to~$\Sigma^{n-1} \times S^1$.  Since performing surgery on a loop with distinct framings yields knots that are related by a Gluck twist, \Cref{Glucktwist} implies that the collection $\mathcal{K}_{p,q,r}$ can be constructed so that all of its elements have diffeomorphic traces, by choosing the correct framing~$f$ of the simple loop that represents~$\alpha t$ in the above construction. We denote this shared trace by $X(p,q,r)$.

Now note that, if $\{p,q,r\} \neq \{p^\prime, q^\prime, r^\prime\}$, then the traces $X(p,q,r)$ and $X(p^\prime,q^\prime,r^\prime)$ can be distinguished by the fundamental groups of their boundaries, namely~$\pi_1(\Sigma(p,q,r)) \times \Z$ and $\pi_1(\Sigma(p^\prime,q^\prime,r^\prime)) \times \Z$, respectively. In fact, their commutator subgroups are~$\pi_1(\Sigma(p,q,r))$ and $\pi_1(\Sigma(p^\prime,q^\prime,r^\prime))$, respectively, and they differ by the discussion at the end of \cite{milnor19753}*{Section~3}, where the fundamental group of the Brieskorn sphere is the group denoted by $[\Gamma,\Gamma]$. This implies that the trace $X(p,q,r)$ of the knots in~$\mathcal{K}_{p,q,r}$ is not diffeomorphic to the trace $X(p^\prime,q^\prime,r^\prime)$ of the knots in the family $\mathcal{K}_{p^\prime,q^\prime,r^\prime}$ (which is constructed analogously, but using another Brieskorn sphere~$\Sigma(p^\prime,q^\prime,r^\prime)$ as a starting point).
\end{proof}

\begin{remark}
    Plotnick's construction \cite{Plotnick1983InfinitelyMD} produces infinite families of \emph{slice} knots in $S^n$ which are pairwise inequivalent. Indeed, gluing the $2$-handle~$h_2^{\alpha}$ to~$V \times S^1$ along a representative of $\alpha t$ yields a contractible $(n+1)$-manifold bounded by a homotopy $n$-sphere. By the smooth $h$-cobordism theorem, this is a smooth $(n+1)$-disc. A slice disc for the knot $K_{\alpha}$ is hence given by the cocore of $h^2_{\alpha}$. In particular, such slice discs correspond to the `knotted ball pairs' constructed by Plotnick in~\cite{Plotnick1983InfinitelyMD}*{Theorem~1}.
\end{remark}

A similar argument yields \Cref{infinite_example_intro} in the topological category for $n=4$, again following Plotnick's work. 

\begin{theorem}\label{Plot4}
 There is an infinite collection $\{X_i\}_{i \geq 1}$ of pairwise non-homeomorphic topological $5$-manifolds such that, for each integer $i \geq 1$, there exists an infinite collection $\{K_j^i \}_{j \geq 1}$ of pairwise non-isotopic locally flat $2$-knots in $S^4$ whose trace is o.p.~homeomorphic to $X_i$, \ie
    \begin{align*}
        X(K_j^i) \approx X_i \qquad \text{for all integers} \quad i,j \geq 1.
    \end{align*}
    In particular, for each $i \geq 1$, this results in an infinite family of pairwise non-isotopic locally flat $2$-knots in $S^4$ with o.p.~homeomorphic traces. 
\end{theorem}

\begin{proof}
    This follows from noting that the construction used in the proof of \cref{infinite_example_intro} works in the topological category for $n=4$. In particular, construct the $4$-manifold~$\Sigma^3\times S^1$ as in the proof of \cref{infinite_example_intro} starting from a Brieskorn sphere~$\Sigma(p,q,r)$ that bounds a contractible smooth $4$-manifold. Performing surgery on $\Sigma^3 \times S^1$ along any curve representing $\alpha t$ for $\alpha \in A^{p,q,r}$ as above yields a smooth knot $K_\alpha$ in a smooth homotopy $4$-sphere $S^4_{\alpha}$ which bounds a contractible $5$-manifold. The topological $h$-cobordism theorem \cite{freedman:top_4-mfds} then implies that each $S^4_{\alpha}$ is homeomorphic to $S^4$.\footnote{We are unable to prove the result in the smooth category precisely because $S^4_\alpha$ may not be diffeomorphic to $S^4$.}
    Hence in this way we produce an infinite family of locally flat knots in $S^4$, which are pairwise non-isotopic for the same reasons as described in the proof of \Cref{infinite_example_intro}. Exactly as we explained for Gluck twists of knots before \Cref{Glucktwisttop}, we can define the (topological) trace of $K_\alpha$. By construction, for any pair~$\alpha,\beta \in A^{p,q,r}$, the manifolds $S^4_{\alpha}(K_\alpha)$ and $S^4_{\beta}(K_\beta)$ are o.p.~diffeomorphic. This implies, by a slight variation of the proofs of \Cref{Glucktwist} and \Cref{Glucktwisttop}, that~$X(K_\beta)$ is o.p.~homeomorphic either to $X(K_\alpha)$ or to~$X(T(K_\alpha))$. Note that, even if the trace is now only a topological manifold, $S^4_{\beta}\times [0,1]$ union the $3$-handle contained in the trace has a smooth structure and is simply connected, so we can use the theory of smooth manifolds to find the embedded disc with tubular neighbourhood needed to start the proof of \Cref{Glucktwist} (see \Cref{subsec:surgerydiffeos}). Then we proceed as in the proof of \Cref{Glucktwisttop}. We conclude as in the proof of \Cref{infinite_example_intro}.
\end{proof}

Finally, we can also construct arbitrarily large families of non-isotopic knots in $S^4$ with o.p.~diffeomorphic traces.

\begin{theorem}\label{rem:fin_many_2-knots_same_trace}
For every $N>0$, there exists a family of more than $N$ pairwise non-isotopic $2$-knots in $S^4$ with o.p.~diffeomorphic traces.
\end{theorem}

\begin{proof}
The construction of these families is essentially given in \cite{plotnick1983homotopy}*{Section~3}, to which we refer the reader for more details. Fix $p,q,r$ positive, pairwise coprime integers such that $\Sigma(p,q,r) \not \cong \Sigma(2,3,5)$. Recall that $\Sigma(p,q,r) $ is the $r$-fold branched cover of $S^3$ branched along the torus knot $T(p,q)$, and that the branched covering automorphisms form a group which is canonically identified with $\Z/r\Z$ (see \eg the beginning of \cite{plotnick1983homotopy}*{Section 1}). Let $\sigma$ denote the canonical generator of this group corresponding to $1 \in \Z/r\Z$. It is proved in \cites{zeeman1965twisting} that the $r$-twist spun knot of~$T(p,q)$ is a fibred knot, with fibre~$\Sigma(p,q,r) \setminus \intr B^3$, and monodromy given by the restriction of~$\sigma$. For~$k>0$ and $\operatorname{gcd}(k,r)=1$, the $k$-branched cover of $S^4$ along this knot defines a knot $K_{p,q}^{r,k}$ in~$S^4$ whose exterior fibres over $S^1$ with fibre $\Sigma(p,q,r) \setminus \intr B^3$ and monodromy given by the restriction of $\sigma^k$ \cites{pao1978non,plotnick1983homotopy}. It is shown in~\cite{plotnick1983homotopy}*{Proposition~3.5} that if $ k \not \equiv\pm k' (\text{mod } r)$, then the knots~$K_{p,q}^{r,k}$ and~$K_{p,q}^{r,k'} $ are not isotopic. Moreover, $\sigma$ is isotopic to the identity (see for example the proof of Lemma 1.1 in \cite{milnor19753} and \cite{orlik2006seifert}*{p.~143}), hence $S^4(K_{p,q}^{r,k})\cong \Sigma(p,q,r) \times S^1$. By permuting $p,q,r$ and varying~$k$ accordingly, one obtains a finite family of knots with the same surgery, which are not isotopic by \cite{plotnick1983homotopy}*{Theorem~3.7}. Up to taking Gluck twists, which in this case are again smooth spheres in $S^4$ by \cite{pao1978non}*{Theorem~5}, these knots have o.p.~diffeomorphic traces by \Cref{Glucktwisttop}. Taking $p,q,r$ arbitrarily large, this gives arbitrarily large families of knots with the same trace.
\end{proof}

Note that \Cref{infinite_example_intro} (for $n \geq 5$) and \Cref{rem:fin_many_2-knots_same_trace} (for $n=4$) together provide an alternative proof of \Cref{theorem:non-isotopic}.

\section{Detecting the unknot by its surgery in higher dimensions}
\label{sec:property_R}

In the previous sections, we have seen that it is possible to construct many non-isotopic knots in $S^n$ with the same trace, and therefore the same surgery. In this section, we will show that, in contrast to these results, the unknot in $S^n$ is detected by its surgery, and consequently by its trace. Throughout this section, we will assume that $n \geq 4$.

The same reasons that make it more difficult to distinguish non-isotopic knots with the same trace in higher dimensions simplify the proof that, in higher dimensions, the unknot is detected by its surgery (and hence by its trace), compared to analogous results in dimension~$3$. 
We restate \Cref{theorem:unknot} for the convenience of the reader. 

\begin{reptheorem}{theorem:unknot}
    For $n\geq 4$, let $K$ be a knot in~$S^n$. Suppose that its surgery~$S^n(K)$ is (possibly orientation-reversingly) diffeomorphic to the surgery~$S^n(U) \cong S^{n-1} \times S^1$ of the unknot $U$. Then $K$ is isotopic to $U$. 
\end{reptheorem}
We will prove \Cref{theorem:unknot} as an application of a more general result regarding \emph{unknotted} sphere embeddings of~$S^{n-2}$ in~$S^n$.

\begin{definition}\label{def:unknotted embedding}
A knot $K\colon S^{n-2}\hookrightarrow S^n$ is \emph{unknotted} if its image bounds a ball.   
\end{definition}

Recall, by contrast, that the unknot is (the isotopy class) of the embedding of $S^{n-2}\hookrightarrow S^n$ that extends to an embedding $D^{n-1}\hookrightarrow S^n$ (see \Cref{subsec:detection,sec:background_knots}).
\begin{theorem}
\label{thm: property R new}
    Let $K_1$ and $K_2$ be knots in~$S^n$, where $K_2$ is unknotted. If there exists a (possibly orientation-reversing) diffeomorphism between the surgeries~$S^n(K_1)$ and~$S^n(K_2)$, then $K_1$ is isotopic to $K_2$.
\end{theorem}

\begin{proof}[{Proof of \Cref{theorem:unknot}}]
    By specialising the statement of \Cref{thm: property R new} to the case where~$K_2$ is the unknot $U$, we obtain the statement of \Cref{theorem:unknot}.
\end{proof}

In order to prove \Cref{thm: property R new}, we start by proving \Cref{lemma:surg_exterior} below. Note that the diffeomorphism type of the exterior of a knot $K$ does not depend on the choice of the parametrisation of $K$, but only on its image. In particular, the exterior of any unknotted knot in $S^n$ is diffeomorphic to $D^{n-1}\times S^1$.

\begin{lemma}\label{lemma:surg_exterior}
Let $K_1$ be a knot in $S^n$ such that $S^n(K_1) \cong S^n(K_2)$, where~$K_2$ is unknotted. Then the exterior of $K_1$ is diffeomorphic to $D^{n-1}\times S^1$.    
\end{lemma}

\begin{proof}
Let $\gamma_{K_1}$ and $\gamma_{K_2}$ denote the dual curves of $K_1$ and $K_2$ in the respective surgeries. Since $E_{K_2} \cong D^{n-1}\times S^1$, the fundamental group of $S^n(K_2)$ is isomorphic to~$\Z$, generated by~$\gamma_{K_2}$, by \Cref{proposition:fundamental_group}. If there is a diffeomorphism $\varphi \colon S^n(K_1) \to S^n(K_2)$, then $S^n(K_1)$ also has fundamental group isomorphic to $\Z$; again by \Cref{proposition:fundamental_group}, the same holds for the exterior~$E_{K_1}$. This implies that~$\pi_1(E_{K_1})\cong \Z$ is generated by $\mu_{K_1}$ and, by \Cref{proposition:fundamental_group} once again, we deduce that~$\pi_1(S^n(K_1))$ is generated by a conjugate of the dual curve $\gamma_{K_1}$ for some orientation. The diffeomorphism $\varphi$ thus maps the curve~$\gamma_{K_1}$ to a curve that (up to orientation) is freely homotopic to $\gamma_{K_2}$, since both curves represent the unique generator of $\pi_1(S^n(K_1)) \cong \pi_1(S^n(K_2)) \cong \Z$ up to sign. Since $n\geq 4$, we obtain that~$\varphi(\gamma_{K_1})$ and~$\gamma_{K_2}$ are isotopic up to orientation. By the isotopy extension theorem (see \eg \cite{hirsch}*{Chapter~8}), there is an ambient isotopy of $S^n(K_2)$ taking $\varphi(\gamma_{K_1})$ to~$\gamma_2$, still up to orientation. This yields a diffeomorphism $\psi\colon S^n(K_2)\to S^n(K_2)$ such that $\psi(\varphi(\gamma_1))=\gamma_2$, up to orientation. Now the composition $\psi\circ \varphi$, restricted to the exterior of $\gamma_1$, gives a diffeomorphism from the exterior $E_{K_1}$ of $\gamma_{K_1}$ to the exterior~$E_{K_2}$ of $\gamma_{K_2}$, where~$E_{K_2}\cong D^{n-1}\times S^1$.
\end{proof}

Next, we prove that, under the same hypotheses, $K_1$ also defines an unknotted sphere embedding. 

\begin{proposition}\label{prop: lemma_surg_unk}
Let $K_1$ be a knot in $S^n$ such that $S^n(K_1) \cong S^n(K_2)$, where $K_2$ is unknotted. Then $K_1$ is also unknotted.
\end{proposition}

\begin{proof}
We show that $K_1$ and $K_2$ are isotopic up to reparametrisation of the embedding, which implies the desired result, since unknottedness only depends on the image of the sphere embedding. By \Cref{lemma:surg_exterior}, the exterior of $K_1$ is diffeomorphic to~$D^{n-1}\times S^1$. For~$n\geq 5$, it was shown by \cites{shaneson1968embeddings,Levine,Levine2,Trotter} that if the exterior of a knot in $S^n$ is homotopy equivalent to a circle, then it is unknotted. This finishes the proof for $n\geq 5$, so let $n=4$ for the rest of the proof. In \cite{Gluck} it was proven that if the 2-knots $K_1$ and $K_2$ in $S^4$ have diffeomorphic exteriors, then, up to reparametrisation, either they are isotopic, or $K_1$ is isotopic to the Gluck twist~$T(K_2)$ of $K_2$. So the image of $K_1$ can be identified with that of either $K_2$ or $T(K_2)$. Since~$K_2$ is unknotted and the map used to define the Gluck twist (see \Cref{def:gluck_twist}) extends to a diffeomorphism of~$D^{3}\times S^1$, we have that $T(K_2)$ is unknotted as well. Thus, in either case, $K_1$ is unknotted.
\end{proof}

\begin{proof}[{Proof of \Cref{thm: property R new}}]
    Suppose that $S^n(K_1)$ is diffeomorphic to $S^n(K_2)$, where $K_2$ is unknotted. If the diffeomorphism is not o.p., then we can compose it with an orientation-reversing diffeomorphism of $S^n(K_2)$, which is the product of a homotopy sphere and $S^1$, to obtain an o.p.~diffeomorphism between $S^n(K_1)$ and $S^n(K_2)$. By \Cref{prop: lemma_surg_unk}, both $K_1$ and $K_2$ are unknotted and the result follows from the fact that two unknotted sphere embeddings have diffeomorphic surgeries if and only if they are isotopic \cite{shaneson1968embeddings}*{Theorem~1.3}. 
\end{proof}

The above proof of \Cref{prop: lemma_surg_unk}, which shows that being unknotted is detected by the surgery along a knot, mainly uses the fact that the knot group of an unknotted sphere is $\Z$ and thus has exactly two normal generators. In fact, we can give a partial generalisation of this detection result as follows. Recall from \Cref{subsec:framings} that the dual curve $\gamma_K$ in the surgery of $S^n$ along a knot $K$ comes equipped with a preferred framing.

\begin{lemma}\label{lemma7}
    Let $K_1$ and $K_2$ be knots in $S^n$ with dual curves $\gamma_{K_1}$ and $\gamma_{K_2}$ in the respective surgeries. Suppose that $\varphi\colon S^n(K_1)\to S^n(K_2)$ is an o.p.~diffeomorphism between surgeries such that $\varphi_\ast(\gamma_{K_1})$ and $\gamma_{K_2}$ are conjugate in~$\pi_1(S^n(K_2))$. Then either $K_2$ is isotopic to $K_1$ or the Gluck twist~$T(K_2)$ is a knot in the standard smooth~$S^n$ which is isotopic to $K_1$. The former happens if and only if $\varphi(\gamma_{K_1})$ is isotopic to~$\gamma_{K_2}$ as a framed loop.
\end{lemma}

\begin{proof}
Since $\varphi_\ast(\gamma_{K_1})$ and $\gamma_{K_2}$ are conjugate in $\pi_1(S^n(K_2))$, we can assume, after isotopy, that $\varphi(\gamma_{K_1})=\gamma_{K_2}$. Then, as explained in the proofs of \Cref{Glucktwist} and \Cref{Glucktwisttop}, the two possible conclusions depend on whether $\varphi(\gamma_{K_1})$ and $\gamma_{K_2}$ agree as framed loops or not. In the first case, by doing loop surgery, we get an equivalence, and hence an isotopy by \Cref{lem:equivalence_implies_iso}, between~$K_1$ and $K_2$; in the second case, we get an o.p.~diffeomorphism between $S^n$ and the Gluck twist of $S^n$ along $K_2$, mapping~$K_1$ to~$T(K_2)$, and we conclude again by \Cref{lem:equivalence_implies_iso}.
\end{proof}

\begin{proposition}\label{gl}
    Let $K_1$ and $K_2$ be knots
    in $S^n$ with dual curves $\gamma_{K_1}$ and $\gamma_{K_2}$ in the respective surgeries. Let $\Phi\colon X(K_1)\to X(K_2)$ be an o.p.~diffeomorphism between their traces such that~$\Phi_*(\gamma_{K_1})$ and $\gamma_{K_2}$ are conjugate in $\pi_1(S^n(K_2))$. Then $K_1$ is isotopic to~$K_2$.
    \end{proposition}

\begin{proof}
    Apply \cref{lemma7} to $(\Phi\vert_{\partial X(K_1)})^{-1}$. Then either $K_1$ is isotopic to $K_2$, in which case we are done, or the Gluck twist~$T(K_1)$ is a knot in the standard smooth $S^n$ which is isotopic to $K_2$. Thus, we assume that $\Phi$ is an o.p.~diffeomorphism between~$X(K_1)$ and $ X\coloneqq X(T(K_1)) $. Since the dual curves $\Phi_*(\gamma_{K_1})$ and~$\gamma \coloneqq \gamma_{T(K_1)}$ are conjugate in $\pi_1(\partial X)$, we can suppose that $\Phi_*(\gamma_{K_1})=\gamma_{T(K_1)}=\gamma$. Notice that $\gamma$ inherits two framings $f$ and $T(f)$, from $\gamma_{K_1} $ and $\gamma_{T(K_1)}$ respectively. In the following, we show that $f$ and $T(f)$ coincide. In light of \cref{lemma7}, this will imply that $K_1$ is isotopic to $T(K_1)$ (and hence to $K_2$ as well).
    
    Using the identification of $X$ with $X(K_1)$ and $X(T(K_1))$, we can decompose~$X$ in two different ways by gluing an $(n+1)$-dimensional $2$-handle to~$ \partial X \times [0,1]$ along~$\gamma \times \{1\}$ with framing~$f$ and~$T(f)$, respectively, and a single $(n+1)$-handle. Let $D_1 $ and~$ D_2 $ be the cores of such $2$-handles, and consider the $2$-sphere $S \coloneqq D_1 \cup D_2 \subset  DX $, where $DX \coloneqq X \cup_{\id_{\partial X}} -X$ is the double of $X$, and~$D_1$ and $D_2$ lie in different copies of $X$. 
    Since $f$ and $T(f)$ are the unique framings of $\gamma$ that extend to $D_1$ and $D_2$ respectively, $f$ and $T(f)$ coincide if and only if $\nu S$ is the trivial $(n-1)$-bundle over~$S$. Since~$H^2(X;\mathbb{Z}/2\mathbb{Z})=0$, we know that $X$ is spin. In particular, $DX$ is also spin. This implies that
    \begin{align}\label{eq:w_2}
        0=w_2(T(DX)|_S) = w_2(\nu S)+w_1(\nu S)\cup w_1(TS) + w_2(TS)=w_2(\nu S),
    \end{align}
    since by the definition of the normal bundle, \(T(DX)|_S\) is given by a sum of the normal and tangent bundles of \(S\), and thus we can apply the Whitney Product Theorem axiom for Stiefel--Whitney classes (see \cite{MilnorStasheff1974}*{Chapter~\(4\)}). We also use 
    the fact that since \(S\) is an oriented surface, \(w_1(TS)=w_2(TS)=0\). Thus~\eqref{eq:w_2} implies \(w_2(\nu S)=0\) and, as a consequence, $\nu S$ is the trivial $(n-1)$-bundle over~$S^2$. In fact, for $n \ge 4$, the $(n-1)$-vector bundles over the $2$-sphere are classified by $\pi_1(SO(n-1)) \cong \Z_2$ (see \Cref{rem:framings}) and the two possible vector bundles are distinguished by their second Stiefel--Whitney class. Therefore, the two framings $f$ and~$T(f)$ of~$\gamma$ must coincide, which implies that the two knots $K_1$ and $T(K_1)$ are isotopic in~$S^n$.
\end{proof}

\section{Further questions}
\label{sec:questions}

We end this article by posing some questions that naturally arise from the results presented in this article.

\subsection{Detecting knots by their surgeries or traces}
In \cref{theorem:unknot}, we showed that the surgery along a knot detects the unknot for $n \geq 4$. One might ask whether this property holds for other knot types. 

\begin{question}\label{quest:otherknots}
    For $n\geq 4$, are there knots in $S^n$ (other than unknotted ones) that are detected by their surgeries or traces? 
    Is there an infinite family? 
\end{question}

Note that in dimension $n = 3$, it is well known that \Cref{quest:otherknots} can be answered affirmatively~\cites{gabai:property_R,Ghiggini,BS:5_2,BS:nearly_fibred}. Indeed, every knot is detected by infinitely many of its surgeries~\cites{lackenby} (see also \cites{KS_unique_surgery}). We think that the weaker version for the traces would be interesting as well; see also \eg \cites{baldwin-sivek:traces-Lspace,baker-kegel-motegi:no-trace-detects} for work on the case when $n=3$.

\subsection{Traces vs surgeries}

In general, it would be interesting to understand the difference between equivalent traces and equivalent surgeries. In \Cref{theorem:non-isotopic,infinite_example_intro}, we constructed families of pairwise non-isotopic knots with diffeomorphic traces. The following question remains unresolved.

\begin{question}\label{quest:diffeo_surg_not_traces}
    For $n \geq 4$, are there any pairs of knots in $S^n$ with diffeomorphic surgeries, but whose traces are not diffeomorphic? 
\end{question}

Maybe a construction as described in \Cref{rem:surgeries} could help to answer \Cref{quest:diffeo_surg_not_traces}. Again, \Cref{quest:diffeo_surg_not_traces} is known to have a positive answer in dimension $n=3$, see for example~\cites{Akbulut1,Akbulut2,Akbulut3,yasui:corks-concordance,manolescu-piccirillo:RBG}. A related question is the following.

\begin{question}\label{ques:extension}
    For $n \geq 4$, do there exist knots $K_1$ and $K_2$ in $S^n$ with a surgery diffeomorphism $\varphi\colon S^n(K_1)\rightarrow S^n(K_2)$ that does not extend to a trace diffeomorphism? 
\end{question}

According to \Cref{extension_intro} and \Cref{rem:extension}, it would be enough to construct a pair of knots such that $X(K_1) \cup_{\varphi} - X(K_2)$ is not spin in order to answer \Cref{ques:extension} positively. If, in addition, the mapping class group of the surgery $S^n(K_1)\cong S^n(K_2)$ is trivial, then we would also get a positive answer to \Cref{quest:diffeo_surg_not_traces}.

\subsection{Homeomorphisms vs diffeomorphisms} Manifolds of dimension four or higher may be homeomorphic but not diffeomorphic.
In dimension $n=3$, there exist knots that have homeomorphic but non-diffeomorphic traces, see for example~\cites{Akbulut1,Akbulut2,Akbulut3,yasui:corks-concordance}. In this section, we consider whether this can occur in higher dimensions.

When $K_1$ is an $(n-2)$-knot in $S^n$, where $n\ge 8$ and  $ \MCG({S^{n-2}}) \neq 1$, we can reparametrise $K_1$ using a non-trivial element of $\MCG({S^{n-2}}) $ to obtain a new knot~$K_2$ (see the discussion in \Cref{REM:embeddings}). The traces $X(K_1)$ and $X(K_2)$ will always be homeomorphic by the Alexander trick, but they will in general not be diffeomorphic. For example, if $K_1$ is the unknot and $K_2$ is a reparametrisation of~$K_1$, then we have $X(K_1) \cong S^{n-1} \times D^2$, whereas $ X(K_2) \cong \widetilde{S} \times D^2$, where $\widetilde{S}$ is an exotic~$S^{n-1}$. However, the two manifolds are not diffeomorphic, as can be seen by using the Product Structure Theorem (see Theorem~5.1 and Remark 2 thereafter in~\cite{KirbySiebenmann}*{Essay I, Theorem~5.1, p.~31}), assuming that $S$ is at least 5-dimensional. Thus it is natural to consider whether there are knots~$K_1$ and~$K_2$ such that for every possible reparametrisation of the knots, their traces are homeomorphic but not diffeomorphic.

\begin{question}\label{ques:home_vs_diffeo}
    Let $n=4$, or $n\geq 5$ and odd. Do there exist knots in $S^n$ 
    whose traces are o.p.~homeomorphic but not o.p.~diffeomorphic,
    up to any reparametrisation of the knot? 
\end{question}

The following proposition shows that slice knots in $S^n$ for $n\geq 6$ cannot be used to answer the above question. We thank the referee for pointing out this argument.
Recall that all even-dimensional knots are slice~\cites{kervaire_french,kervaire:knot_cob}, which explains the dimension restrictions in \cref{ques:home_vs_diffeo}.

\begin{proposition}\label{rem:slice_knots} 
Let $n\geq 6$ and let $K$ and $K'$ be slice $(n-2)$-knots in $S^n$. If $X(K)$ is o.p.~homeomorphic to $X(K')$, then there exists a reparametrisation of $K$, which we denote by $\overline{K}$, such that the traces $X(\overline{K})$ and $X(K')$ are o.p.~diffeomorphic.
\end{proposition}

\begin{proof}
Since $K$ is slice, either generator of the homotopy group $\pi_{n-1}(X(K))\cong \Z$ can be represented by a smoothly embedded, possibly exotic, $(n-1)$-sphere $S$, formed by gluing a slice disc in $D^{n+1}$ to the core of the attached handle. Note that the inclusion $S \hookrightarrow X(K)$ is a homotopy equivalence. Thus, by \cite{KirbySiebenmann}*{Essay IV, Classification Theorem~10.1, p.~194}, since $n+1\geq 6$, the smooth structure on the trace $X(K)$ is determined up to concordance, and therefore o.p.~diffeomorphism,  by its restriction to an open tubular neighbourhood of $S$. By the Product Structure Theorem \cite{KirbySiebenmann}*{Essay I, Theorem 5.1, p.~31}, the smooth structure on $X(K)$ is hence determined by the smooth structure on~$S$, since $n-1\geq 5$. Therefore, if $K'$ is a slice knot whose trace is o.p.~homeomorphic to that of $K$, and $S' \subseteq X(K')$ is the smoothly embedded homotopy sphere obtained by gluing the slice disc of $K'$ to the core of the $(n-2)$-handle, then the only obstruction to $X(K)$ and $X(K')$ being o.p.~diffeomorphic is the class of $S \# -S'$ in $\Theta_{n-1}$, the group of homotopy spheres, up to o.p.~diffeomorphism.

The elements of $\Theta_{n-1}$, for $n-1\geq 5$ are known to be \emph{twisted spheres}, i.e. $(n-1)$-spheres obtained by gluing two $(n-1)$-discs along their boundary by a non-trivial element in the mapping class group of $S^{n-2}$ (see, e.g.~\cite{kosinski:differential_manifolds}*{Chapter~ VIII}). Therefore altering the parametrisation of $K$ by composing with the appropriate element of the mapping class group of $S^{n-2}$, we can modify the smooth structure on~$S\subseteq X(K)$ so that it agrees with that on $S' \subseteq X(K')$.
\end{proof}

For $n=4$, \Cref{cor:exotic_traces} shows that if there exists a trace homeomorphism that restricts to a diffeomorphism of the surgeries, then the traces are already diffeomorphic. So to answer \Cref{ques:home_vs_diffeo} positively, one needs to construct homeomorphic traces where no trace homeomorphism restricts to a surgery diffeomorphism. The proof of \Cref{cor:exotic_traces} depends on \Cref{extension_intro}, which we only proved in dimension $n=4$. See \Cref{rem:extension_higher} for comments on dimensions $n \geq 5$. 

The following question asks whether the ``if'' direction of \Cref{extension_intro} is true for~$n\geq 5$. As just mentioned, our proof of \Cref{extension_intro} needs the assumption $n=4$. The converse direction to \Cref{ques:extensions} is true in all dimensions, as discussed in \Cref{rem:extension}.

\begin{question}\label{ques:extensions}
    For $n \geq 5$, let $K_1$ and $K_2$ be knots in $S^n$ and let $\varphi\colon S^n(K_1) \to S^n(K_2)$ be a surgery diffeomorphism such that $X(K_1) \cup_{\varphi} - X(K_2)$ is spin. Does $\varphi$ extend to a trace diffeomorphism $\Phi\colon X(K_1) \to X(K_2)$?
\end{question}

\subsection{Gluck twists}

Gluck twisting a knot in $S^n$ yields a smooth $n$-manifold $S^n_K$ which is always homeomorphic to $S^n$. In dimension $4$, it is known for many knots $K$ that $S^4_K$ is even diffeomorphic to $S^4$, but in general it remains unclear if $S^4_K$ is always diffeomorphic to the standard $4$-sphere (see \cref{rem:Gluck_manifold}). In higher dimensions, we could not find such statements in the literature. In \Cref{prop:std} we showed that for knots $K$ in~$S^n$ with $n\geq 5$, the Gluck twist manifold $S^n_K$ is diffeomorphic to $S^n$ if~$K$ extends to an embedding $D^{n-1} \hookrightarrow D^{n+1}$.
In general, however, this seems to be open (see also \Cref{rem:slice}).

\begin{question}
    For $n \geq 4$, given a knot $K$ in $S^n$, is the Gluck twist $S^n_K$ of $S^n$ along~$K$ always diffeomorphic to the standard smooth $n$-sphere?
\end{question}

\subsection{Other codimensions}

In this article, we have concentrated on knots in codimension $2$. The RBG construction works in any codimension and we can thus ask many of the above questions for sphere embeddings in any codimension. 

\begin{question}\label{ques:other_codimensions}
    Do there exist non-isotopic $k$-sphere embeddings in $S^n$ that have diffeomorphic traces or surgeries for $k\neq n-2$? Can the Haefliger spheres~\cite{haefliger:knotted-spheres} be realised by the RBG construction?
\end{question}

\appendix

\section{Sage source code}\label{sec:appendix_sage}

Here we provide the Sage source code~\cites{GAP4,sagemath} used in \Cref{subsec:explicit_reps}.

\subsection{Counting representations}

The following function \emph{repr} takes as input a finitely presented group $H$, a finite group $A$, an element $q \in A$ and an integer $m \geq 0$, and computes the number of homomorphisms from $H$ to $A$ which map the $m$-th generator of $H$ (with labels starting from $0$) to $A$.
\begin{lstlisting}[language=Python]
from itertools import product
from sage.all import libgap, ConjugacyClass

def repr(H,A,q,m):
    R = 0			# R = counter of reps
    L = A.list()		# L = list of elements of A
    # T encodes the set of maps from the generators of H 
    # (minus the m-th one) to elements of A
    T = product(L,repeat=len(H.gens())-1)
    F = list(T)			# convert T into a list F
    for f in F:				
        I = list(f)     
        # add q such that the m-th generator is mapped to q
        I.insert(m,q)	
        # check if assignement defined by I defines a homom.
        if (libgap.fail!=libgap.GroupHomomorphismByImages
            (H,A,H.gens(),I)):
            R = R+1              # increase the counter R
    return R	
\end{lstlisting}

\subsection{Example} Here is an example of the usage of the function \emph{repr} to count representations of the group $F_1$ into $A_5$ as needed in the proof of \Cref{proposition:A_5_reps}.

\begin{lstlisting}[language=Python]
F.<x,y,a> = FreeGroup()
H = F/[x^-1*y*x*y*x^-1*y^-1, x^-1*a*x*a^-1*x^-1*y*a*y^-1]
A = AlternatingGroup(5)
q = A[1]                # -> q = (1,5,4,3,2)
R = repr(H,A,q,0)      # -> R = 6
S = repr(H,A,q,2)      # -> S = 1
\end{lstlisting}

\bibliographystyle{alpha}
\bibliography{Biblio}

@incollection {rolfsen:PL-I-equivalence,
    AUTHOR = {Rolfsen, Dale},
     TITLE = {Piecewise-linear {$I$}-equivalence of links},
 BOOKTITLE = {Low-dimensional topology ({C}helwood {G}ate, 1982)},
    SERIES = {London Math. Soc. Lecture Note Ser.},
    VOLUME = {95},
     PAGES = {161--178},
 PUBLISHER = {Cambridge Univ. Press, Cambridge},
      YEAR = {1985},
      ISBN = {0-521-26982-2},
   MRCLASS = {57Q45},
  MRNUMBER = {827301},
       DOI = {10.1017/CBO9780511662744.006},
       URL = {https://doi.org/10.1017/CBO9780511662744.006},
}

@article {BS:5_2,
    AUTHOR = {Baldwin, John A. and Sivek, Steven},
     TITLE = {Characterizing slopes for {$5_2$}},
   JOURNAL = {J. Lond. Math. Soc. (2)},
  FJOURNAL = {Journal of the London Mathematical Society. Second Series},
    VOLUME = {109},
      YEAR = {2024},
    NUMBER = {6},
     PAGES = {Paper No. e12951, 64},
      ISSN = {0024-6107,1469-7750},
   MRCLASS = {57K30 (57K18 57R58)},
  MRNUMBER = {4760447},
MRREVIEWER = {Kimihiko\ Motegi},
       DOI = {10.1112/jlms.12951},
       URL = {https://doi.org/10.1112/jlms.12951},
}

@article {BS:nearly_fibred, 
    AUTHOR = {Baldwin, John A. and Sivek, Steven},
     TITLE = {Zero-surgery characterizes infinitely many knots},
   JOURNAL = {Math. Res. Lett.},
  FJOURNAL = {Mathematical Research Letters},
    VOLUME = {31},
      YEAR = {2024},
    NUMBER = {6},
     PAGES = {1639--1653},
      ISSN = {1073-2780,1945-001X},
   MRCLASS = {57K10},
  MRNUMBER = {4862352},
       DOI = {10.4310/mrl.250210230206},
       URL = {https://doi.org/10.4310/mrl.250210230206},
}

@article {browder,
    AUTHOR = {Browder, William},
     TITLE = {Diffeomorphisms of {$1$}-connected manifolds},
   JOURNAL = {Trans. Amer. Math. Soc.},
  FJOURNAL = {Transactions of the American Mathematical Society},
    VOLUME = {128},
      YEAR = {1967},
     PAGES = {155--163},
      ISSN = {0002-9947,1088-6850},
   MRCLASS = {57.20},
  MRNUMBER = {212816},
MRREVIEWER = {Edward\ M.\ Brown},
       DOI = {10.2307/1994525},
       URL = {https://doi.org/10.2307/1994525},
}

@book{MilnorStasheff1974,
  author    = {John W. Milnor and James D. Stasheff},
  title     = {Characteristic Classes},
  series    = {Annals of Mathematics Studies},
  volume    = {76},
  publisher = {Princeton University Press},
  address   = {Princeton, NJ},
  year      = {1974},
  isbn      = {9780691081229}
}

@ARTICLE{yasui:corks-concordance,
       author = {{Yasui}, Kouichi},
        title = "{Corks, exotic 4-manifolds and knot concordance}",
         year = 2015,
note = {arXiv:1505.02551},
}

@article {Akbulut,
    AUTHOR = {Akbulut, Selman},
     TITLE = {On {$2$}-dimensional homology classes of {$4$}-manifolds},
   JOURNAL = {Math. Proc. Cambridge Philos. Soc.},
  FJOURNAL = {Mathematical Proceedings of the Cambridge Philosophical
              Society},
    VOLUME = {82},
      YEAR = {1977},
    NUMBER = {1},
     PAGES = {99--106},
      ISSN = {0305-0041,1469-8064},
   MRCLASS = {57D95 (57C45)},
  MRNUMBER = {433476},
MRREVIEWER = {Mauricio\ Guti\'errez},
       DOI = {10.1017/S0305004100053718},
       URL = {https://doi.org/10.1017/S0305004100053718},
}

@article{mca1989knots,
 author = {Gordon, C. McA. and Luecke, J.},
 title = {Knots are determined by their complements},
 fjournal = {Journal of the American Mathematical Society},
 journal = {J. Am. Math. Soc.},
 issn = {0894-0347},
 volume = {2},
 number = {2},
 pages = {371--415},
 year = {1989},
 doi = {10.2307/1990979},
 zbMATH = {4111600},
 Zbl = {0678.57005}
}

@article{waldhausen1968irreducible,
 author = {Waldhausen, Friedhelm},
 title = {On irreducible 3-manifolds which are sufficiently large},
 fjournal = {Annals of Mathematics. Second Series},
 journal = {Ann. of Math. (2)},
 issn = {0003-486X},
 volume = {87},
 pages = {56--88},
 year = {1968},
 doi = {10.2307/1970594},
 url = {pub.uni-bielefeld.de/record/1782185},
 zbMATH = {3253055},
 Zbl = {0157.30603}
}

@article {piccirillo:conway-knot,
    AUTHOR = {Piccirillo, Lisa},
     TITLE = {The {C}onway knot is not slice},
   JOURNAL = {Ann. of Math. (2)},
  FJOURNAL = {Annals of Mathematics. Second Series},
    VOLUME = {191},
      YEAR = {2020},
    NUMBER = {2},
     PAGES = {581--591},
      ISSN = {0003-486X,1939-8980},
   MRCLASS = {57K10 (57R65)},
  MRNUMBER = {4076631},
MRREVIEWER = {Laurence\ R.\ Taylor},
       DOI = {10.4007/annals.2020.191.2.5},
       URL = {https://doi.org/10.4007/annals.2020.191.2.5},
}

@article {manolescu-piccirillo:RBG,
    author = {Manolescu, C. and Piccirillo, L.},
     title = {From zero surgeries to candidates for exotic definite 4-manifolds},
   journal = {J. Lond. Math. Soc. (2)},
  fjournal = {Journal of the London Mathematical Society. Second Series},
    volume = {108},
      year = {2023},
    number = {5},
     pages = {2001--2036},
      ISSN = {0024-6107,1469-7750},
   mrclass = {57K40 (57K10 57R60)},
  mrnumber = {4668522},
       doi = {10.1112/jlms.12800},
       url = {https://doi.org/10.1112/jlms.12800},
}

@book{gompf-stipsicz:book,
	  author = {Gompf, R. E. and Stipsicz, A. I. },
 title = {4-manifolds and {Kirby} calculus},
 fseries = {Graduate Studies in Mathematics},
 series = {Grad. Stud. Math.},
 issn = {1065-7339},
 volume = {20},
 isbn = {0-8218-0994-6},
 year = {1999},
 publisher = {Providence, RI: American Mathematical Society},
 zbMATH = {1356915},
 Zbl = {0933.57020}
}

@ARTICLE{baldwin-sivek:traces-Lspace,
       author = {Baldwin, J.~A. and Sivek, S.},
        title = "{L-spaces and knot traces}",
         year = 2025,
note = {arXiv:2501.00914},
}

@article {piccirillo:shake-genus,
    AUTHOR = {Piccirillo, L.},
     TITLE = {Shake genus and slice genus},
   JOURNAL = {Geom. Topol.},
  FJOURNAL = {Geometry \& Topology},
    VOLUME = {23},
      YEAR = {2019},
    NUMBER = {5},
     PAGES = {2665--2684},
      ISSN = {1465-3060,1364-0380},
   MRCLASS = {57K10 (57R65)},
  MRNUMBER = {4019900},
MRREVIEWER = {Laurence\ R.\ Taylor},
       DOI = {10.2140/gt.2019.23.2665},
       URL = {https://doi.org/10.2140/gt.2019.23.2665},
}

@article {brakes:multiple,
    AUTHOR = {Brakes, W. R.},
     TITLE = {Manifolds with multiple knot-surgery descriptions},
   JOURNAL = {Math. Proc. Cambridge Philos. Soc.},
  FJOURNAL = {Mathematical Proceedings of the Cambridge Philosophical
              Society},
    VOLUME = {87},
      YEAR = {1980},
    NUMBER = {3},
     PAGES = {443--448},
      ISSN = {0305-0041,1469-8064},
   MRCLASS = {57M25},
  MRNUMBER = {556924},
MRREVIEWER = {Wilbur\ Whitten},
       DOI = {10.1017/S0305004100056875},
       URL = {https://doi.org/10.1017/S0305004100056875},
}

@ARTICLE{baker-kegel-motegi:no-trace-detects,
       author = {{Baker}, Kenneth L. and {Kegel}, Marc and {Motegi}, Kimihiko},
        title = "{Knots not detected by any trace}",
         year = 2025,
    note = {arXiv:2503.14172},
}

@article {gabai:property_R,
    AUTHOR = {Gabai, David},
     TITLE = {Foliations and the topology of {$3$}-manifolds. {III}},
   JOURNAL = {J. Differential Geom.},
  FJOURNAL = {Journal of Differential Geometry},
    VOLUME = {26},
      YEAR = {1987},
    NUMBER = {3},
     PAGES = {479--536},
      ISSN = {0022-040X,1945-743X},
   MRCLASS = {57N10 (57R30)},
  MRNUMBER = {910018},
MRREVIEWER = {Jean-Pierre\ Otal},
       URL = {http://projecteuclid.org/euclid.jdg/1214441488},
}

@article {suciu:inf-many-ribbon-knots,
    AUTHOR = {Suciu, Alexander I.},
     TITLE = {Infinitely many ribbon knots with the same fundamental group},
   JOURNAL = {Math. Proc. Cambridge Philos. Soc.},
  FJOURNAL = {Mathematical Proceedings of the Cambridge Philosophical
              Society},
    VOLUME = {98},
      YEAR = {1985},
    NUMBER = {3},
     PAGES = {481--492},
      ISSN = {0305-0041,1469-8064},
   MRCLASS = {57Q45 (57M05)},
  MRNUMBER = {803607},
MRREVIEWER = {Don\v co\ Dimovski},
       DOI = {10.1017/S0305004100063684},
       URL = {https://doi.org/10.1017/S0305004100063684},
}

@article {haefliger:knotted-spheres,
    AUTHOR = {Haefliger, Andr\'e},
     TITLE = {Knotted {$(4k-1)$}-spheres in {$6k$}-space},
   JOURNAL = {Ann. of Math. (2)},
  FJOURNAL = {Annals of Mathematics. Second Series},
    VOLUME = {75},
      YEAR = {1962},
     PAGES = {452--466},
      ISSN = {0003-486X},
   MRCLASS = {57.20},
  MRNUMBER = {145539},
MRREVIEWER = {M.\ A.\ Kervaire},
       DOI = {10.2307/1970208},
       URL = {https://doi.org/10.2307/1970208},
}

@book{hirsch,
 author = {Hirsch, Morris W.},
 title = {Differential topology},
 fseries = {Graduate Texts in Mathematics},
 series = {Grad. Texts Math.},
 issn = {0072-5285},
 volume = {33},
 year = {1976},
 publisher = {Springer, Cham},
 doi = {10.1007/978-1-4684-9449-5},
 zbMATH = {3555096},
 Zbl = {0356.57001}
}

@article{Normantrick,
    AUTHOR = {Norman, R. A.},
     TITLE = {Dehn's lemma for certain {$4$}-manifolds},
   JOURNAL = {Invent. Math.},
  FJOURNAL = {Inventiones Mathematicae},
    VOLUME = {7},
      YEAR = {1969},
     PAGES = {143--147},
      ISSN = {0020-9910,1432-1297},
   MRCLASS = {57.10 (55.00)},
  MRNUMBER = {246309},
MRREVIEWER = {H.\ Terasaka},
       DOI = {10.1007/BF01389797},
       URL = {https://doi.org/10.1007/BF01389797},
}

@article {Akbulut1,
    AUTHOR = {Akbulut, Selman},
     TITLE = {A fake compact contractible {$4$}-manifold},
   JOURNAL = {J. Differential Geom.},
  FJOURNAL = {Journal of Differential Geometry},
    VOLUME = {33},
      YEAR = {1991},
    NUMBER = {2},
     PAGES = {335--356},
      ISSN = {0022-040X,1945-743X},
   MRCLASS = {57N13 (57R55)},
  MRNUMBER = {1094459},
MRREVIEWER = {Laurence\ R.\ Taylor},
       URL = {http://projecteuclid.org/euclid.jdg/1214446320},
}

@article {Akbulut2,
    AUTHOR = {Akbulut, Selman},
     TITLE = {An exotic {$4$}-manifold},
   JOURNAL = {J. Differential Geom.},
  FJOURNAL = {Journal of Differential Geometry},
    VOLUME = {33},
      YEAR = {1991},
    NUMBER = {2},
     PAGES = {357--361},
      ISSN = {0022-040X,1945-743X},
   MRCLASS = {57N13 (57R55)},
  MRNUMBER = {1094460},
MRREVIEWER = {Laurence\ R.\ Taylor},
       URL = {http://projecteuclid.org/euclid.jdg/1214446321},
}

@article {Akbulut3,
    AUTHOR = {Akbulut, Selman and Matveyev, Rostislav},
     TITLE = {Exotic structures and adjunction inequality},
   JOURNAL = {Turkish J. Math.},
  FJOURNAL = {Turkish Journal of Mathematics},
    VOLUME = {21},
      YEAR = {1997},
    NUMBER = {1},
     PAGES = {47--53},
      ISSN = {1300-0098,1303-6149},
   MRCLASS = {57R57 (57N13)},
  MRNUMBER = {1456158},
MRREVIEWER = {Paolo\ Lisca},
}

@misc{KS_unique_surgery,
      title={Unique Surgery Descriptions along Knots}, 
      author={Marc Kegel and Misha Schmalian},
      year={2025},
note = {arXiv:2508.18521},
}

@article {Ghiggini,
    AUTHOR = {Ghiggini, Paolo},
     TITLE = {Knot {F}loer homology detects genus-one fibred knots},
   JOURNAL = {Amer. J. Math.},
  FJOURNAL = {American Journal of Mathematics},
    VOLUME = {130},
      YEAR = {2008},
    NUMBER = {5},
     PAGES = {1151--1169},
      ISSN = {0002-9327,1080-6377},
   MRCLASS = {57M25 (57R58)},
  MRNUMBER = {2450204},
       DOI = {10.1353/ajm.0.0016},
       URL = {https://doi.org/10.1353/ajm.0.0016},
}

@article{LBT,
    AUTHOR = {Gabai, David},
     TITLE = {The 4-dimensional light bulb theorem},
   JOURNAL = {J. Amer. Math. Soc.},
  FJOURNAL = {Journal of the American Mathematical Society},
    VOLUME = {33},
      YEAR = {2020},
    NUMBER = {3},
     PAGES = {609--652},
      ISSN = {0894-0347,1088-6834},
   MRCLASS = {57K40 (57N35)},
  MRNUMBER = {4127900},
MRREVIEWER = {Sergey\ M.\ Finashin},
       DOI = {10.1090/jams/920},
       URL = {https://doi.org/10.1090/jams/920},
}

@book {kosinski:differential_manifolds,
 author = {Kosinski, Antoni A.},
 title = {Differential manifolds},
 fseries = {Pure and Applied Mathematics (Academic Press)},
 series = {Pure Appl. Math., Academic Press},
 issn = {0079-8169},
 volume = {138},
 isbn = {0-12-421850-4},
 year = {1993},
 publisher = {Boston, MA: Academic Press},
 zbMATH = {108358},
 Zbl = {0767.57001}
}

@article {Gluck,
    AUTHOR = {Gluck, Herman},
     TITLE = {The embedding of two-spheres in the four-sphere},
   JOURNAL = {Trans. Amer. Math. Soc.},
  FJOURNAL = {Transactions of the American Mathematical Society},
    VOLUME = {104},
      YEAR = {1962},
     PAGES = {308--333},
      ISSN = {0002-9947,1088-6850},
   MRCLASS = {55.70},
  MRNUMBER = {146807},
MRREVIEWER = {Morton\ Brown},
       DOI = {10.2307/1993581},
       URL = {https://doi.org/10.2307/1993581},
}

@misc{mathoverflow,
      title={Equivalence of knotted spheres in {$S^4$}}, 
      author={Ryan Budney},
      year={2024},
      note={Mathoverflow, \url{https://mathoverflow.net/q/476480}},
}

@book {Rolfsen,
    AUTHOR = {Rolfsen, Dale},
     TITLE = {Knots and links},
    SERIES = {Mathematics Lecture Series},
    VOLUME = {No. 7},
 PUBLISHER = {Publish or Perish, Inc., Berkeley, CA},
      YEAR = {1976},
     PAGES = {ix+439},
   MRCLASS = {55-01},
  MRNUMBER = {515288},
}

@article {plotnick_suciu,
    AUTHOR = {Plotnick, Steven P. and Suciu, Alexander I.},
     TITLE = {{$k$}-invariants of knotted {$2$}-spheres},
   JOURNAL = {Comment. Math. Helv.},
  FJOURNAL = {Commentarii Mathematici Helvetici},
    VOLUME = {60},
      YEAR = {1985},
    NUMBER = {1},
     PAGES = {54--84},
      ISSN = {0010-2571,1420-8946},
   MRCLASS = {57Q45 (55P15)},
  MRNUMBER = {787662},
MRREVIEWER = {John\ G.\ Ratcliffe},
       DOI = {10.1007/BF02567400},
       URL = {https://doi.org/10.1007/BF02567400},
}

@article{Plotnick1983InfinitelyMD,
    AUTHOR = {Plotnick, Steven P.},
     TITLE = {Infinitely many disk knots with the same exterior},
   JOURNAL = {Math. Proc. Cambridge Philos. Soc.},
  FJOURNAL = {Mathematical Proceedings of the Cambridge Philosophical
              Society},
    VOLUME = {93},
      YEAR = {1983},
    NUMBER = {1},
     PAGES = {67--72},
      ISSN = {0305-0041,1469-8064},
   MRCLASS = {57Q45},
  MRNUMBER = {684276},
MRREVIEWER = {Martin\ Scharlemann},
       DOI = {10.1017/S0305004100060345},
       URL = {https://doi.org/10.1017/S0305004100060345},
}

@article{plotnick1983homotopy,
 author = {Plotnick, Steven P.},
 title = {The homotopy type of four-dimensional knot complements},
 fjournal = {Mathematische Zeitschrift},
 journal = {Math. Z.},
 issn = {0025-5874},
 volume = {183},
 pages = {447--471},
 year = {1983},
 doi = {10.1007/BF01173923},
 url = {https://eudml.org/doc/173332},
 zbMATH = {3816629},
 Zbl = {0516.57009}
}

@article {gordon,
    AUTHOR = {Gordon, C. McA.},
     TITLE = {Some higher-dimensional knots with the same homotopy groups},
   JOURNAL = {Quart. J. Math. Oxford Ser. (2)},
  FJOURNAL = {The Quarterly Journal of Mathematics. Oxford. Second Series},
    VOLUME = {24},
      YEAR = {1973},
     PAGES = {411--422},
      ISSN = {0033-5606,1464-3847},
   MRCLASS = {57C45},
  MRNUMBER = {326746},
MRREVIEWER = {D.\ W. L. Sumners},
       DOI = {10.1093/qmath/24.1.411},
       URL = {https://doi.org/10.1093/qmath/24.1.411},
}

@article {lomonaco,
    AUTHOR = {Lomonaco, Jr., S. J.},
     TITLE = {The homotopy groups of knots. {I}. {H}ow to compute the
              algebraic {$2$}-type},
   JOURNAL = {Pacific J. Math.},
  FJOURNAL = {Pacific Journal of Mathematics},
    VOLUME = {95},
      YEAR = {1981},
    NUMBER = {2},
     PAGES = {349--390},
      ISSN = {0030-8730,1945-5844},
   MRCLASS = {57Q45 (55Q52)},
  MRNUMBER = {632192},
MRREVIEWER = {Hale\ F.\ Trotter},
       URL = {http://projecteuclid.org/euclid.pjm/1102735075},
}

@ARTICLE{conway2022invariants2knots,
      title={Invariants of  {$2$}-knots}, 
      author={Anthony Conway},
      year={2022},
      note = {arXiv:2210.17447},
}

@article {Papakyriakopoulos:dehnslemma_asphericity,
    AUTHOR = {Papakyriakopoulos, C. D.},
     TITLE = {On {D}ehn's lemma and the asphericity of knots},
   JOURNAL = {Ann. of Math. (2)},
  FJOURNAL = {Annals of Mathematics. Second Series},
    VOLUME = {66},
      YEAR = {1957},
     PAGES = {1--26},
      ISSN = {0003-486X},
   MRCLASS = {55.0X},
  MRNUMBER = {90053},
MRREVIEWER = {R.\ H.\ Fox},
       DOI = {10.2307/1970113},
       URL = {https://doi.org/10.2307/1970113},
}

@book {fox_crowell,
    AUTHOR = {Crowell, Richard H. and Fox, Ralph H.},
     TITLE = {Introduction to knot theory},
      NOTE = {Based upon lectures given at Haverford College under the
              Philips Lecture Program},
 PUBLISHER = {Ginn and Company, Boston, MA},
      YEAR = {1963},
     PAGES = {x+182},
   MRCLASS = {55.10},
  MRNUMBER = {146828},
MRREVIEWER = {L.\ Neuwirth},
}

@article {fox:free_diff_calc,
    AUTHOR = {Fox, Ralph H.},
     TITLE = {Free differential calculus. {I}. {D}erivation in the free
              group ring},
   JOURNAL = {Ann. of Math. (2)},
  FJOURNAL = {Annals of Mathematics. Second Series},
    VOLUME = {57},
      YEAR = {1953},
     PAGES = {547--560},
      ISSN = {0003-486X},
   MRCLASS = {20.0X},
  MRNUMBER = {53938},
MRREVIEWER = {R.\ Bott},
       DOI = {10.2307/1969736},
       URL = {https://doi.org/10.2307/1969736},
}

@article {MR634760,
    AUTHOR = {Casson, Andrew J. and Harer, John L.},
     TITLE = {Some homology lens spaces which bound rational homology balls},
   JOURNAL = {Pacific J. Math.},
  FJOURNAL = {Pacific Journal of Mathematics},
    VOLUME = {96},
      YEAR = {1981},
    NUMBER = {1},
     PAGES = {23--36},
      ISSN = {0030-8730,1945-5844},
   MRCLASS = {57N10 (14J17)},
  MRNUMBER = {634760},
MRREVIEWER = {W.\ D.\ Neumann},
       URL = {http://projecteuclid.org/euclid.pjm/1102734944},
}

@article{Zeeman:unknotting,
    AUTHOR = {Zeeman, E. C.},
     TITLE = {Unknotting combinatorial balls},
   JOURNAL = {Ann. of Math. (2)},
  FJOURNAL = {Annals of Mathematics. Second Series},
    VOLUME = {78},
      YEAR = {1963},
     PAGES = {501--526},
      ISSN = {0003-486X},
   MRCLASS = {55.70},
  MRNUMBER = {160218},
MRREVIEWER = {A.\ Haefliger},
       DOI = {10.2307/1970538},
       URL = {https://doi.org/10.2307/1970538},
}

@article{stallings:unknotting,
    AUTHOR = {Stallings, John},
     TITLE = {On topologically unknotted spheres},
   JOURNAL = {Ann. of Math. (2)},
  FJOURNAL = {Annals of Mathematics. Second Series},
    VOLUME = {77},
      YEAR = {1963},
     PAGES = {490--503},
      ISSN = {0003-486X},
   MRCLASS = {55.20 (54.00)},
  MRNUMBER = {149458},
MRREVIEWER = {O.\ G.\ Harrold},
       DOI = {10.2307/1970127},
       URL = {https://doi.org/10.2307/1970127},
}

@article{haefliger_diff_embeddings,
author = {Haefliger, A.},
 title = {Differentiable embeddings of {{\({S}^ n\)}} in {{\(S^{n+q}\)}} for {{\(q > 2\)}}},
 fjournal = {Annals of Mathematics. Second Series},
 journal = {Ann. of Math. (2)},
 issn = {0003-486X},
 volume = {83},
 pages = {402--436},
 year = {1966},
 doi = {10.2307/1970475},
 zbMATH = {3243231},
 Zbl = {0151.32502}
}

@article{levine:class_diff_knots,
    AUTHOR = {Levine, J.},
     TITLE = {A classification of differentiable knots},
   JOURNAL = {Ann. of Math. (2)},
  FJOURNAL = {Annals of Mathematics. Second Series},
    VOLUME = {82},
      YEAR = {1965},
     PAGES = {15--50},
      ISSN = {0003-486X},
   MRCLASS = {57.20 (55.20)},
  MRNUMBER = {180981},
MRREVIEWER = {Morris\ W.\ Hirsch},
       DOI = {10.2307/1970561},
       URL = {https://doi.org/10.2307/1970561},
}

@book {Carter_Kamada_Saito,
 author = {Carter, Scott and Kamada, Seiichi and Saito, Masahico},
 title = {Surfaces in 4-space},
 fseries = {Encyclopaedia of Mathematical Sciences},
 series = {Encycl. Math. Sci.},
 issn = {0938-0396},
 volume = {142},
 isbn = {3-540-21040-7},
 year = {2004},
 publisher = {Berlin: Springer},
 zbMATH = {2121101},
 Zbl = {1078.57001}
}

@article{haefliger:plongements,
author = {Haefliger, A.},
 title = {Plongements diff{\'e}rentiables de vari{\'e}t{\'e}s dans vari{\'e}t{\'e}s},
 fjournal = {Commentarii Mathematici Helvetici},
 journal = {Comment. Math. Helv.},
 issn = {0010-2571},
 volume = {36},
 pages = {47--82},
 year = {1961},
 doi = {10.1007/BF02566892},
 url = {https://eudml.org/doc/139228},
 zbMATH = {3167572},
 Zbl = {0102.38603}
}

@article{Boyer1,
    AUTHOR = {Boyer, Steven},
     TITLE = {Simply-connected {$4$}-manifolds with a given boundary},
   JOURNAL = {Trans. Amer. Math. Soc.},
  FJOURNAL = {Transactions of the American Mathematical Society},
    VOLUME = {298},
      YEAR = {1986},
    NUMBER = {1},
     PAGES = {331--357},
      ISSN = {0002-9947,1088-6850},
   MRCLASS = {57N13 (57N10)},
  MRNUMBER = {857447},
MRREVIEWER = {V.\ G.\ Turaev},
       DOI = {10.2307/2000623},
       URL = {https://doi.org/10.2307/2000623},
}

@article{BardenClassify,
    AUTHOR = {Barden, D.},
     TITLE = {Simply connected five-manifolds},
   JOURNAL = {Ann. of Math. (2)},
  FJOURNAL = {Annals of Mathematics. Second Series},
    VOLUME = {82},
      YEAR = {1965},
     PAGES = {365--385},
      ISSN = {0003-486X},
   MRCLASS = {57.10},
  MRNUMBER = {184241},
MRREVIEWER = {S.\ Smale},
       DOI = {10.2307/1970702},
       URL = {https://doi.org/10.2307/1970702},
}

@book{cerf,
    AUTHOR = {Cerf, Jean},
     TITLE = {Sur les diff\'eomorphismes de la sph\`ere de dimension trois
              {$(\Gamma \sb{4}=0)$}},
    SERIES = {Lecture Notes in Mathematics},
    VOLUME = {No. 53},
 PUBLISHER = {Springer-Verlag, Berlin-New York},
      YEAR = {1968},
     PAGES = {xii+133},
   MRCLASS = {57.31},
  MRNUMBER = {229250},
MRREVIEWER = {N.\ Kuiper},
}

@article {smale:2-sphereMCG,
    AUTHOR = {Smale, Stephen},
     TITLE = {Diffeomorphisms of the {$2$}-sphere},
   JOURNAL = {Proc. Amer. Math. Soc.},
  FJOURNAL = {Proceedings of the American Mathematical Society},
    VOLUME = {10},
      YEAR = {1959},
     PAGES = {621--626},
      ISSN = {0002-9939,1088-6826},
   MRCLASS = {57.00},
  MRNUMBER = {112149},
MRREVIEWER = {G.\ T.\ Whyburn},
       DOI = {10.2307/2033664},
       URL = {https://doi.org/10.2307/2033664},
}

@article{smale,
    AUTHOR = {Smale, S.},
     TITLE = {On the structure of manifolds},
   JOURNAL = {Amer. J. Math.},
  FJOURNAL = {American Journal of Mathematics},
    VOLUME = {84},
      YEAR = {1962},
     PAGES = {387--399},
      ISSN = {0002-9327,1080-6377},
   MRCLASS = {57.10},
  MRNUMBER = {153022},
MRREVIEWER = {A.\ Dold},
       DOI = {10.2307/2372978},
       URL = {https://doi.org/10.2307/2372978},
}

@article{lashofshaneson,
    AUTHOR = {Lashof, Richard K. and Shaneson, Julius L.},
     TITLE = {Classification of knots in codimension two},
   JOURNAL = {Bull. Amer. Math. Soc.},
  FJOURNAL = {Bulletin of the American Mathematical Society},
    VOLUME = {75},
      YEAR = {1969},
     PAGES = {171--175},
      ISSN = {0002-9904},
   MRCLASS = {57.20},
  MRNUMBER = {242175},
MRREVIEWER = {D.\ W. L. Sumners},
       DOI = {10.1090/S0002-9904-1969-12197-X},
       URL = {https://doi.org/10.1090/S0002-9904-1969-12197-X},
}

@manual{sagemath,
  author       = {{S}age{M}ath},
  Title        = {The {S}age {M}athematics {S}oftware {S}ystem ({V}ersion 10.1)},
  note         = {\url{http://www.sagemath.org}},
  Year         = {2025},
}

@manual{GAP4,
    author = "{GAP~Group}, {The}",
    title        = "{GAP -- Groups, Algorithms, and Programming,
                    Version 4.15.1}",
    year         = 2025,
    note          = "\url{https://www.gap-system.org}",
    }

@article {freedman:top_4-mfds,
    AUTHOR = {Freedman, Michael Hartley},
     TITLE = {The topology of four-dimensional manifolds},
   JOURNAL = {J. Differential Geometry},
  FJOURNAL = {Journal of Differential Geometry},
    VOLUME = {17},
      YEAR = {1982},
    NUMBER = {3},
     PAGES = {357--453},
      ISSN = {0022-040X,1945-743X},
   MRCLASS = {57N12 (57R80 57R99)},
  MRNUMBER = {679066},
MRREVIEWER = {John\ J.\ Walsh},
       URL = {http://projecteuclid.org/euclid.jdg/1214437136},
}

@article {Gordon-Luecke,
    AUTHOR = {Gordon, C. McA. and Luecke, J.},
     TITLE = {Knots are determined by their complements},
   JOURNAL = {Bull. Amer. Math. Soc. (N.S.)},
  FJOURNAL = {American Mathematical Society. Bulletin. New Series},
    VOLUME = {20},
      YEAR = {1989},
    NUMBER = {1},
     PAGES = {83--87},
      ISSN = {0273-0979,1088-9485},
   MRCLASS = {57M25 (57M40)},
  MRNUMBER = {972070},
MRREVIEWER = {Martin\ Scharlemann},
       DOI = {10.1090/S0273-0979-1989-15706-6},
       URL = {https://doi.org/10.1090/S0273-0979-1989-15706-6},
}

@article {milnor_exotic_7-spheres,
    AUTHOR = {Milnor, John},
     TITLE = {On manifolds homeomorphic to the {$7$}-sphere},
   JOURNAL = {Ann. of Math. (2)},
  FJOURNAL = {Annals of Mathematics. Second Series},
    VOLUME = {64},
      YEAR = {1956},
     PAGES = {399--405},
      ISSN = {0003-486X},
   MRCLASS = {55.0X},
  MRNUMBER = {82103},
MRREVIEWER = {J.\ C.\ Moore},
       DOI = {10.2307/1969983},
       URL = {https://doi.org/10.2307/1969983},
}

@article{massey59,
  title={On the normal bundle of a sphere imbedded in Euclidean space},
  author={Massey, WS},
  journal={Proceedings of the American Mathematical Society},
  volume={10},
  number={6},
  pages={959--964},
  year={1959},
  publisher={JSTOR}
}

@article {palais,
    AUTHOR = {Palais, Richard S.},
     TITLE = {Extending diffeomorphisms},
   JOURNAL = {Proc. Amer. Math. Soc.},
  FJOURNAL = {Proceedings of the American Mathematical Society},
    VOLUME = {11},
      YEAR = {1960},
     PAGES = {274--277},
      ISSN = {0002-9939,1088-6826},
   MRCLASS = {57.00},
  MRNUMBER = {117741},
MRREVIEWER = {A.\ M.\ Gleason},
       DOI = {10.2307/2032968},
       URL = {https://doi.org/10.2307/2032968},
}

@article {brown,
    AUTHOR = {Brown, Morton},
     TITLE = {A proof of the generalized {S}choenflies theorem},
   JOURNAL = {Bull. Amer. Math. Soc.},
  FJOURNAL = {Bulletin of the American Mathematical Society},
    VOLUME = {66},
      YEAR = {1960},
     PAGES = {74--76},
      ISSN = {0002-9904},
   MRCLASS = {54.00 (57.00)},
  MRNUMBER = {117695},
MRREVIEWER = {S.\ Eilenberg},
       DOI = {10.1090/S0002-9904-1960-10400-4},
       URL = {https://doi.org/10.1090/S0002-9904-1960-10400-4},
}

@book {milnor:lectures_h-cob,
    AUTHOR = {Milnor, John},
     TITLE = {Lectures on the {$h$}-cobordism theorem},
      NOTE = {Notes by L. Siebenmann and J. Sondow},
 PUBLISHER = {Princeton University Press, Princeton, NJ},
      YEAR = {1965},
     PAGES = {v+116},
   MRCLASS = {57.10},
  MRNUMBER = {190942},
MRREVIEWER = {P.\ E.\ Conner},
}

@article{shaneson1968embeddings,
    AUTHOR = {Shaneson, Julius L.},
     TITLE = {Embeddings with codimension two of spheres in spheres and
              {$h$}-cobordisms of {$S\sp{1}\times S\sp{3}$}},
   JOURNAL = {Bull. Amer. Math. Soc.},
  FJOURNAL = {Bulletin of the American Mathematical Society},
    VOLUME = {74},
      YEAR = {1968},
     PAGES = {972--974},
      ISSN = {0002-9904},
   MRCLASS = {57.20},
  MRNUMBER = {230325},
MRREVIEWER = {H.\ Suzuki},
       DOI = {10.1090/S0002-9904-1968-12107-X},
       URL = {https://doi.org/10.1090/S0002-9904-1968-12107-X},
}

@article{behrens2020detecting,
    AUTHOR = {Behrens, M. and Hill, M. and Hopkins, M. J. and Mahowald, M.},
     TITLE = {Detecting exotic spheres in low dimensions using {${\rm
              coker}\,J$}},
   JOURNAL = {J. Lond. Math. Soc. (2)},
  FJOURNAL = {Journal of the London Mathematical Society. Second Series},
    VOLUME = {101},
      YEAR = {2020},
    NUMBER = {3},
     PAGES = {1173--1218},
      ISSN = {0024-6107,1469-7750},
   MRCLASS = {57R55 (55N34 55Q45 55Q51 55T15 57R60)},
  MRNUMBER = {4111938},
MRREVIEWER = {Torsten\ Asselmeyer-Maluga},
       DOI = {10.1112/jlms.12301},
       URL = {https://doi.org/10.1112/jlms.12301},
}

@book{friedl2024surveyfoundationsfourmanifoldtheory,
    AUTHOR = {Friedl, Stefan and Nagel, Matthias and Orson, Patrick and
              Powell, Mark},
     TITLE = {The foundations of four-manifold theory in the topological
              category},
    SERIES = {New York Journal of Mathematics. NYJM Monographs},
    VOLUME = {6},
 PUBLISHER = {State University of New York, University at Albany, Albany,
              NY},
      YEAR = {2025},
     PAGES = {152},
   MRCLASS = {99-06},
  MRNUMBER = {4981701},
}

@article {lackenby,
    AUTHOR = {Lackenby, Marc},
     TITLE = {Every knot has characterising slopes},
   JOURNAL = {Math. Ann.},
  FJOURNAL = {Mathematische Annalen},
    VOLUME = {374},
      YEAR = {2019},
    NUMBER = {1-2},
     PAGES = {429--446},
      ISSN = {0025-5831,1432-1807},
   MRCLASS = {57M25},
  MRNUMBER = {3961316},
MRREVIEWER = {Brandy\ Guntel\ Doleshal},
       DOI = {10.1007/s00208-018-1757-x},
       URL = {https://doi.org/10.1007/s00208-018-1757-x},
}

@article {litherland:LBT,
    AUTHOR = {Litherland, Rick},
     TITLE = {A generalization of the lightbulb theorem and {PL}
              {$I$}-equivalence of links},
   JOURNAL = {Proc. Amer. Math. Soc.},
  FJOURNAL = {Proceedings of the American Mathematical Society},
    VOLUME = {98},
      YEAR = {1986},
    NUMBER = {2},
     PAGES = {353--358},
      ISSN = {0002-9939,1088-6826},
   MRCLASS = {57Q45},
  MRNUMBER = {854046},
MRREVIEWER = {V.\ G.\ Turaev},
       DOI = {10.2307/2045711},
       URL = {https://doi.org/10.2307/2045711},
}

@article{pao1978non,
  title={Non-linear circle actions on the 4-sphere and twisting spun knots},
  author={Pao, Peter Sie},
  journal={Topology},
  volume={17},
  number={3},
  pages={291--296},
  year={1978},
  publisher={Elsevier},
}

@article{milnor19753,
  title={On the 3-dimensional {B}rieskorn manifolds {M}(p,q,r)},
  author={Milnor, John},
  journal={Knots, {Groups}, 3-{Manif}.; {Pap}. dedic. {Mem}. {R}. {H}. {Fox}.},
  volume={84},
  pages={175--225},
  year={1975}
}

@book{orlik2006seifert,
 author = {Orlik, Peter},
 title = {Seifert manifolds},
 fseries = {Lecture Notes in Mathematics},
 series = {Lect. Notes Math.},
 issn = {0075-8434},
 volume = {291},
 year = {1972},
 publisher = {Springer, Cham},
 doi = {10.1007/BFb0060329},
 zbMATH = {3415101},
 Zbl = {0263.57001}
}

@article{zeeman1965twisting,
 author = {Zeeman, E. C.},
 title = {Twisting spun knots},
 fjournal = {Transactions of the American Mathematical Society},
 journal = {Trans. Amer. Math. Soc.},
 issn = {0002-9947},
 volume = {115},
 pages = {471--495},
 year = {1965},
 doi = {10.2307/1994281},
 zbMATH = {3218554},
 Zbl = {0134.42902}
}

@article {Donaldson,
    AUTHOR = {Donaldson, S. K.},
     TITLE = {Irrationality and the {$h$}-cobordism conjecture},
   JOURNAL = {J. Differential Geom.},
  FJOURNAL = {Journal of Differential Geometry},
    VOLUME = {26},
      YEAR = {1987},
    NUMBER = {1},
     PAGES = {141--168},
      ISSN = {0022-040X,1945-743X},
   MRCLASS = {57R80 (32C10 32G13 32J15 57N13 58E15 58G30)},
  MRNUMBER = {892034},
MRREVIEWER = {David\ R.\ Morrison},
       URL = {http://projecteuclid.org/euclid.jdg/1214441179},
}

@article {kervaire_french,
    AUTHOR = {Kervaire, Michel A.},
     TITLE = {Les n\oe uds de dimensions sup\'erieures},
   JOURNAL = {Bull. Soc. Math. France},
  FJOURNAL = {Bulletin de la Soci\'et\'e{} Math\'ematique de France},
    VOLUME = {93},
      YEAR = {1965},
     PAGES = {225--271},
      ISSN = {0037-9484},
   MRCLASS = {57.20},
  MRNUMBER = {189052},
MRREVIEWER = {Edward\ M.\ Brown},
       URL = {http://www.numdam.org/item?id=BSMF_1965__93__225_0},
}

@incollection {kervaire:knot_cob,
    AUTHOR = {Kervaire, Michel A.},
     TITLE = {Knot cobordism in codimension two},
 BOOKTITLE = {Manifolds--{A}msterdam 1970 ({P}roc. {N}uffic {S}ummer
              {S}chool)},
    SERIES = {Lecture Notes in Math.},
    VOLUME = {Vol. 197},
     PAGES = {83--105},
 PUBLISHER = {Springer, Berlin-New York},
      YEAR = {1971},
   MRCLASS = {55.20 (57.00)},
  MRNUMBER = {283786},
MRREVIEWER = {J.\ P.\ Levine},
}

@article {Levine,
    AUTHOR = {Levine, J.},
     TITLE = {Unknotting spheres in codimension two},
   JOURNAL = {Topology},
  FJOURNAL = {Topology. An International Journal of Mathematics},
    VOLUME = {4},
      YEAR = {1965},
     PAGES = {9--16},
      ISSN = {0040-9383},
   MRCLASS = {57.20 (55.70)},
  MRNUMBER = {179803},
MRREVIEWER = {R.\ H.\ Rosen},
       DOI = {10.1016/0040-9383(65)90045-5},
       URL = {https://doi.org/10.1016/0040-9383(65)90045-5},
}

@book{FQ:book,
 author = {Freedman, Michael H. and Quinn, Frank S.},
 title = {Topology of 4-manifolds},
 fseries = {Princeton Mathematical Series},
 series = {Princeton Math. Ser.},
 volume = {39},
 isbn = {0-691-08577-3},
 year = {1990},
 publisher = {Princeton, NJ: Princeton University Press},
 language = {English},
 zbMATH = {193544},
 Zbl = {0705.57001}
}

@article {Trotter,
    AUTHOR = {Trotter, H. F.},
     TITLE = {On {$S$}-equivalence of {S}eifert matrices},
   JOURNAL = {Invent. Math.},
  FJOURNAL = {Inventiones Mathematicae},
    VOLUME = {20},
      YEAR = {1973},
     PAGES = {173--207},
      ISSN = {0020-9910,1432-1297},
   MRCLASS = {57C45 (15A36)},
  MRNUMBER = {645546},
MRREVIEWER = {Author's review},
       DOI = {10.1007/BF01394094},
       URL = {https://doi.org/10.1007/BF01394094},
}

@article {Levine2,
    AUTHOR = {Levine, J.},
     TITLE = {An algebraic classification of some knots of codimension two},
   JOURNAL = {Comment. Math. Helv.},
  FJOURNAL = {Commentarii Mathematici Helvetici},
    VOLUME = {45},
      YEAR = {1970},
     PAGES = {185--198},
      ISSN = {0010-2571,1420-8946},
   MRCLASS = {57.20},
  MRNUMBER = {266226},
MRREVIEWER = {D.\ W. L. Sumners},
       DOI = {10.1007/BF02567325},
       URL = {https://doi.org/10.1007/BF02567325},
}

@article {McCoy_char_hyp,
    AUTHOR = {McCoy, Duncan},
     TITLE = {On the characterising slopes of hyperbolic knots},
   JOURNAL = {Math. Res. Lett.},
  FJOURNAL = {Mathematical Research Letters},
    VOLUME = {26},
      YEAR = {2019},
     PAGES = {1517--1526},
      ISSN = {1073-2780,1945-001X},
   MRCLASS = {57K10 (57K32)},
  MRNUMBER = {4049819},
MRREVIEWER = {Brandy\ Guntel\ Doleshal},
       DOI = {10.4310/MRL.2019.v26.n5.a12},
       URL = {https://doi.org/10.4310/MRL.2019.v26.n5.a12},
}

@misc{Kegel_Piccirillo,
    title={Knots that share four surgeries},
    author={Marc Kegel and Lisa Piccirillo},
    year={2025},
note = {arXiv:2505.13168},
}

@misc {Wakelin_Picirillo_Hayden_any_slope,
AUTHOR = {Hayden, Kyle and Piccirillo, Lisa and Wakelin, Laura}, 
TITLE = {Dehn surgery functions are never injective},
YEAR = {2025}, 
note = {arXiv:2508.13369},
}

@article {Livingston_surgery,
    AUTHOR = {Livingston, Charles},
     TITLE = {More {$3$}-manifolds with multiple knot-surgery and
              branched-cover descriptions},
   JOURNAL = {Math. Proc. Cambridge Philos. Soc.},
  FJOURNAL = {Mathematical Proceedings of the Cambridge Philosophical
              Society},
    VOLUME = {91},
      YEAR = {1982},
     PAGES = {473--475},
      ISSN = {0305-0041,1469-8064},
   MRCLASS = {57M25 (57M12 57N10)},
  MRNUMBER = {654093},
MRREVIEWER = {G.\ Burde},
       DOI = {10.1017/S0305004100059533},
       URL = {https://doi.org/10.1017/S0305004100059533},
}

@article {Motegi_surgery,
    AUTHOR = {Motegi, Kimihiko},
     TITLE = {Homology {$3$}-spheres which are obtained by {D}ehn surgeries
              on knots},
   JOURNAL = {Math. Ann.},
  FJOURNAL = {Mathematische Annalen},
    VOLUME = {281},
      YEAR = {1988},
     PAGES = {483--493},
      ISSN = {0025-5831,1432-1807},
   MRCLASS = {57M25},
  MRNUMBER = {954154},
MRREVIEWER = {Charles\ Livingston},
       DOI = {10.1007/BF01457158},
       URL = {https://doi.org/10.1007/BF01457158},
}

@article {Osoinach_annulus,
    AUTHOR = {Osoinach, J. K.},
     TITLE = {Manifolds obtained by surgery on an infinite number of knots
              in {$S^3$}},
   JOURNAL = {Topology},
  FJOURNAL = {Topology. An International Journal of Mathematics},
    VOLUME = {45},
      YEAR = {2006},
     PAGES = {725--733},
      ISSN = {0040-9383},
   MRCLASS = {57M25 (57N10)},
  MRNUMBER = {2236375},
       DOI = {10.1016/j.top.2006.02.001},
       URL = {https://doi.org/10.1016/j.top.2006.02.001},
}

@article {Teragaito,
    AUTHOR = {Teragaito, Masakazu},
     TITLE = {A {S}eifert fibered manifold with infinitely many knot-surgery
              descriptions},
   JOURNAL = {Int. Math. Res. Not. IMRN},
  FJOURNAL = {International Mathematics Research Notices. IMRN},
      YEAR = {2007},
     PAGES = {Art. ID rnm 028, 16},
      ISSN = {1073-7928,1687-0247},
   MRCLASS = {57M25 (57N10)},
  MRNUMBER = {2347296},
MRREVIEWER = {Olivier\ Collin},
       DOI = {10.1093/imrn/rnm028},
       URL = {https://doi.org/10.1093/imrn/rnm028},
}

@article {AJLO2,
    AUTHOR = {Abe, Tetsuya and Jong, In Dae and Omae, Yuka and Takeuchi,
              Masanori},
     TITLE = {Annulus twist and diffeomorphic 4-manifolds},
   JOURNAL = {Math. Proc. Cambridge Philos. Soc.},
  FJOURNAL = {Mathematical Proceedings of the Cambridge Philosophical
              Society},
    VOLUME = {155},
      YEAR = {2013},
     PAGES = {219--235},
      ISSN = {0305-0041,1469-8064},
   MRCLASS = {57M25 (57N13)},
  MRNUMBER = {3091516},
MRREVIEWER = {Vyacheslav\ S.\ Krushkal},
       DOI = {10.1017/S0305004113000194},
       URL = {https://doi.org/10.1017/S0305004113000194},
}

@article {MillerPiccirillo,
    AUTHOR = {Miller, Allison N. and Piccirillo, Lisa},
     TITLE = {Knot traces and concordance},
   JOURNAL = {J. Topol.},
  FJOURNAL = {Journal of Topology},
    VOLUME = {11},
      YEAR = {2018},
     PAGES = {201--220},
      ISSN = {1753-8416,1753-8424},
   MRCLASS = {57M25 (57M27)},
  MRNUMBER = {3784230},
MRREVIEWER = {Christopher\ William\ Davis},
       DOI = {10.1112/topo.12054},
       URL = {https://doi.org/10.1112/topo.12054},
}

@article {Abe_Tagami,
    AUTHOR = {Abe, Tetsuya and Tagami, Keiji},
     TITLE = {Knots with infinitely many non-characterizing slopes},
   JOURNAL = {Kodai Math. J.},
  FJOURNAL = {Kodai Mathematical Journal},
    VOLUME = {44},
      YEAR = {2021},
     PAGES = {395--421},
      ISSN = {0386-5991,1881-5472},
   MRCLASS = {57K10},
  MRNUMBER = {4332684},
MRREVIEWER = {Makoto\ Ozawa},
       DOI = {10.2996/kmj/kmj44301},
       URL = {https://doi.org/10.2996/kmj/kmj44301},
}

@article {Baker_Motegi_non_char,
    AUTHOR = {Baker, Kenneth L. and Motegi, Kimihiko},
     TITLE = {Non-characterizing slopes for hyperbolic knots},
   JOURNAL = {Algebr. Geom. Topol.},
  FJOURNAL = {Algebraic \& Geometric Topology},
    VOLUME = {18},
      YEAR = {2018},
     PAGES = {1461--1480},
      ISSN = {1472-2747,1472-2739},
   MRCLASS = {57M25},
  MRNUMBER = {3784010},
MRREVIEWER = {Hyun-Jong\ Song},
       DOI = {10.2140/agt.2018.18.1461},
       URL = {https://doi.org/10.2140/agt.2018.18.1461},
}

@article {Lickorish_sharing_surgery,
    AUTHOR = {Lickorish, W. B. Raymond},
     TITLE = {Surgery on knots},
   JOURNAL = {Proc. Amer. Math. Soc.},
  FJOURNAL = {Proceedings of the American Mathematical Society},
    VOLUME = {60},
      YEAR = {1976},
     PAGES = {296--298 (1977)},
      ISSN = {0002-9939,1088-6826},
   MRCLASS = {57A10 (55A25)},
  MRNUMBER = {488060},
MRREVIEWER = {Deborah\ L.\ Goldsmith},
       DOI = {10.2307/2041161},
       URL = {https://doi.org/10.2307/2041161},
}

@article {AJLO,
    AUTHOR = {Abe, Tetsuya and Jong, In Dae and Luecke, John and Osoinach,
              John},
     TITLE = {Infinitely many knots admitting the same integer surgery and a
              four-dimensional extension},
   JOURNAL = {Int. Math. Res. Not. IMRN},
  FJOURNAL = {International Mathematics Research Notices. IMRN},
      YEAR = {2015},
    NUMBER = {22},
     PAGES = {11667--11693},
      ISSN = {1073-7928,1687-0247},
   MRCLASS = {57M25 (57N13)},
  MRNUMBER = {3456699},
MRREVIEWER = {Wolfgang\ H.\ Heil},
       DOI = {10.1093/imrn/rnv008},
       URL = {https://doi.org/10.1093/imrn/rnv008},
}

@article {Tagami,
 author = {Tagami, Keiji},
 title = {On annulus presentations, dualizable patterns and {RGB}-diagrams},
 fjournal = {Journal of Knot Theory and its Ramifications},
 journal = {J. Knot Theory Ramifications},
 issn = {0218-2165},
 volume = {33},
 number = {9},
 pages = {23},
 note = {Id/No 2497001},
 year = {2024},
 doi = {10.1142/S0218216524970010},
 zbMATH = {7935791},
 Zbl = {1553.57008}
}

@article{crowley2010smoothstructuresetsp,
 author = {Crowley, Diarmuid},
 title = {The smooth structure set of {{\(S^{p} {{\times}} {S}^{q}\)}}},
 fjournal = {Geometriae Dedicata},
 journal = {Geom. Dedicata},
 issn = {0046-5755},
 volume = {148},
 pages = {15--33},
 year = {2010},
 doi = {10.1007/s10711-010-9513-8},
 zbMATH = {5806800},
 Zbl = {1207.57043}
}

@book {KirbySiebenmann,
 author = {Kirby, Robion C. and Siebenmann, Laurence C.},
 title = {Foundational essays on topological manifolds, smoothings and triangulations},
 fseries = {Annals of Mathematics Studies},
 series = {Ann. Math. Stud.},
 volume = {88},
 year = {1977},
 publisher = {Princeton University Press, Princeton, NJ},
 doi = {10.1515/9781400881505},
 zbMATH = {3562124},
 Zbl = {0361.57004}
}

@article{laudenbachpoenaru,
 author = {Laudenbach, Fran{\c{c}}ois and Poenaru, Valentin},
 title = {A note on 4-dimensional handlebodies},
 fjournal = {Bulletin de la Soci{\'e}t{\'e} Math{\'e}matique de France},
 journal = {Bull. Soc. Math. Fr.},
 issn = {0037-9484},
 volume = {100},
 pages = {337--344},
 year = {1972},
 doi = {10.24033/bsmf.1741},
 url = {https://eudml.org/doc/87189},
 zbMATH = {3383315},
 Zbl = {0242.57015}
}

@misc{kim,
      title={$n$-knots in {$S^n\times S^2$} and contractible $(n+3)$-manifolds}, 
      author={Geunyoung Kim},
      year={2023},
      note = {arXiv:2306.06533}
}

\end{document}